\documentclass{amsart}
\usepackage[utf8]{inputenc}

\usepackage{amssymb}
\usepackage{amsmath}
\usepackage{color}
\usepackage{todonotes}
\usepackage{url}
\usepackage{bbm}
\graphicspath{{figuresSub/}}

\newcommand{\indiq}{{\mathbbm{1}}}
\def\E{{\mathbb E}}
\def\R{{\mathbb R}}

\def\N{{\mathbb N}}
\def\intot{\int_0^t}
\def\dd{{\mathrm d}}
\def\cP{{\mathcal P}}
\def\cI{{\mathcal I}}
\def\cJ{{\mathcal J}}
\def\cL{{\mathcal L}}
\def\cF{{\mathcal F}}
\def\cB{{\mathcal B}}
\def\cS{{\mathcal S}}

\def\cA{{\mathcal A}}
\def\cK{{\mathcal K}}
\def\cC{{\mathcal C}}
\def\cV{{\mathcal V}}
\def\tg{\tilde g}

\def\tV{\tilde V}
\def\e{\varepsilon}
\def\sm{{s-}}
\def\pre{\vartriangleleft}
\def\prel{{\trianglelefteq}}
\def\Pro{\mathbb{P}}

\newtheorem{theo}{Theorem}
\newtheorem{prop}[theo]{Proposition}
\newtheorem{rk}[theo]{Remark}
\newtheorem{lem}[theo]{Lemma}
\newtheorem{defin}[theo]{Definition}

\begin{document}

\title{On a toy network of neurons interacting through their dendrites}
\author{Nicolas Fournier, Etienne Tanr\'e and Romain Veltz}

\address{N. Fournier: LPMA, UMR 7599, UPMC, Case 188, 4 place Jussieu, 75252 Paris Cedex 5, France.}

\email{nicolas.fournier@upmc.fr}

\address{E. Tanr\'e: Universit\'e C\^ote d’Azur, INRIA, 2004 route des Lucioles, BP 93, 06902 Sophia-Antipolis, France.}

\email{Etienne.Tanre@inria.fr}

\address{R. Veltz: Universit\'e C\^ote d’Azur, INRIA, 2004 route des Lucioles, BP 93, 06902 Sophia-Antipolis, France.}

\email{Romain.Veltz@inria.fr}



\begin{abstract}
Consider a large number $n$ of neurons, each being connected to approximately $N$ other ones, chosen 
at random.
When a neuron spikes, which occurs randomly at some rate depending on its electric potential,
its potential is set to a minimum value $v_{min}$, and this initiates, after a small delay,
two fronts on the (linear) dendrites of all the neurons to which it is connected.
Fronts move at constant speed. When two fronts (on the dendrite of the same neuron) collide, they annihilate.
When a front hits the soma of a neuron, its potential is increased by a small value $w_n$.
Between jumps, the potentials of the neurons are assumed to drift in $[v_{min},\infty)$, 
according to some well-posed ODE.
We prove the existence and uniqueness of a heuristically derived mean-field 
limit of the system when $n,N \to \infty$ with $w_n \simeq N^{-1/2}$. 
We make use of some recent versions of the results of Deuschel and Zeitouni \cite{dz}
concerning the size of the longest increasing subsequence of an i.i.d. collection of points in the plan.
We also study, in a very particular case, a slightly different model where the neurons spike
when their potential reach some maximum value $v_{max}$, and find an explicit formula for the (heuristic) 
mean-field limit.
\end{abstract}

\maketitle

\section{Introduction and motivation}
Our goal is to establish the existence and uniqueness of the heuristically derived mean-field limits
of two closely related toy models of neurons interacting through their dendrites.

\subsection{Description of the particle systems}\label{mm}
We have $n$ neurons, each has a linear dendrite with length $L>0$ that is endowed with a soma at 
one of its two
extremities.
We have some i.i.d. Bernoulli random variables $(\xi_{ij})_{i,j \in \{1,\dots,n\}}$ with parameter 
$p_n \in (0,1)$,
as well as some i.i.d. $[0,L]$-valued random variables $(X_{ij})_{i,j \in \{1,\dots,n\}}$ with probability 
density $H$ on $[0,L]$.
If $\xi_{ij}=1$, then the neuron $i$ influences the neuron $j$, and the \emph{link} is located, 
on the dendrite 
of the $j$-th neuron, at distance $X_{ij}$ of its soma.

We have a \emph{minimum potential} $v_{min} \in \R$, 
an \emph{excitation parameter} $w_n>0$, a regular
\emph{drift} function $F: [v_{min},\infty)\mapsto \R$ such that $F(v_{min})\geq 0$,
a \emph{propagation velocity} $\rho>0$ and a \emph{delay} $\theta \geq 0$.

We denote by $V^{i,n}_t$ the electric potential of the $i$-th neuron at time $t\geq 0$. We assume 
that initially,
the random variables $(V^{i,n}_0)_{i=1,\dots,n}$ 
are i.i.d. with law $f_0 \in \cP([v_{min},\infty))$.

Between jumps (corresponding to \emph{spike} or \emph{excitation} events), 
the membrane potentials of all the neurons satisfy the ODE
$(V^{i,n}_t)'=F(V^{i,n}_t)$. Note that all the membrane potentials remain above $v_{min}$ thanks to 
the condition
$F(v_{min})\geq 0$.

When a neuron \emph{spikes} (say, the neuron $i$, at time $\tau$), its potential is set to $v_{min}$ (i.e.
$V^{i,n}_\tau=v_{min})$ and, for all $j$ such that $\xi_{ij}=1$, two fronts start, after some delay 
$\theta$ (i.e. at time $\tau+\theta$), on the dendrite of the
$j$-th neuron, at distance $X_{ij}$ of the soma.
Both fronts move with constant velocity 
$\rho$, one going down to the soma
(such a front is called \emph{positive front}), the other one going away from the soma (such a front is called 
\emph{negative front}).

On the dendrite of each neuron, we thus have fronts moving with velocity 
 $\rho$. When a negative front reaches the extremity
of the dendrite, it disappears. When a positive front meets a negative front, they both disappear.
When finally a positive front hits the soma (say, of the $j$-th neuron at time $\sigma$), 
the potential of $j$ is increased by $w_n$, i.e. $V^{j,n}_{\sigma}=V^{j,n}_{\sigma-}+w_n$ and the positive front 
disappears.
Such an occurrence is called an \emph{excitation event}.

We assume that at time $0$, there is no front on any dendrite. This is not very natural, but considerably
simplifies the study.

It remains to describe the spiking events, for which we propose two models.

{\bf Soft model.} We have an increasing regular \emph{rate function} 
$\lambda:[v_{min},\infty)\mapsto \R_+$. 
Each neuron (say the $i$-th one) spikes, independently of the others, 
during $[t,t+dt]$, with probability $\lambda(V^{i,n}_t)dt$.

{\bf Hard model.} There is a \emph{maximum electric potential} $v_{max}>v_{min}$. In such a case, we 
naturally assume that $f_0$ is supported in $[v_{min},v_{max}]$.
A neuron spikes each time its potential reaches $v_{max}$.
This can happen for two reasons, either due to the drift (because it continuously drives $V^{i,n}_t$ to
$v_{max}$ at some time $\tau$, i.e. $V^{i,n}_{\tau-}=v_{max}$), or due to an excitation event
(we have $V^{i,n}_{\tau-}<v_{max}$, a positive front hits the soma  of
the $i$-th neuron at time $\tau$ and $V^{i,n}_{\tau}=V^{i,n}_{\tau-}+w_n \geq v_{max}$).

Observe that the hard model can be seen as the soft model with the choice 
$\lambda(v)=\infty \indiq_{\{v\geq v_{max}\}}$. The soft model is thus a way to regularize the spiking events
by randomization. If, as we have in mind, $\lambda$ looks like $\lambda(v)=(\max\{v-v_0,0\}/(v_1-v_0))^p$ 
for some $v_1>v_0>v_{min}$ and some large value of $p\geq 1$, a neuron will never spike when its potential
is in $[v_{min},v_0]$ and, since $\lambda(v)$ is very small for $v<v_1$ and very large for $v>v_1$,
it will spike with high probability each time its potential is close to $v_1$ and only in such a situation.

\subsection{Biological background}
Although the above particle systems are toy models, they are strongly inspired by biology.

\begin{figure}[ht]
\noindent\fbox{\begin{minipage}{0.9\textwidth}
\centerline{\includegraphics[width=0.6\textwidth]{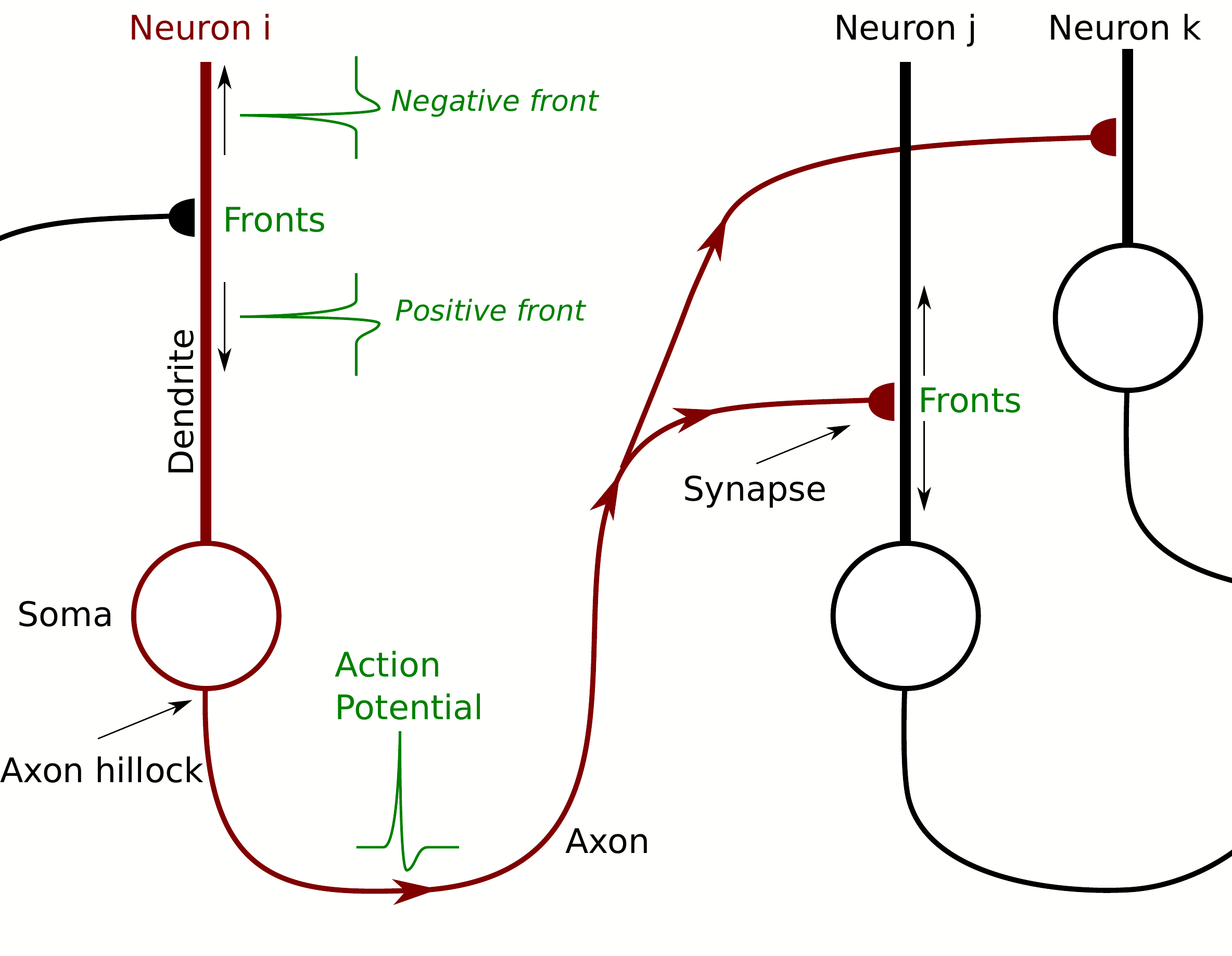}}
\caption{\label{fig:network}
\small{Schematic description of the network organization.}}
\end{minipage}}
\end{figure}

\emph{General organization.} 
A neuron is a specialized cell type of the central nervous system. It is composed of sub-cellular domains which 
serve different functions, see Kandel \cite{k}.  More precisely, the neuron is comprised of a \emph{dendrite}, a 
\emph{soma} (otherwise known as the cell body) and an \emph{axon}. See Figure~\ref{fig:network} for a 
schematic description. 
The neurons are connected with \emph{synapses} which are the interface between the 
axons and the dendrites. On Figure~\ref{fig:network}, 
the axon of the neuron $i$ is connected, through synapses, to the dendrites of the neurons $j$ and $k$.

The neurons transmit information using electrical impulses. 
When the difference of electrical potential across the membrane of the soma of one neuron is high enough,
a sequence of \emph{action potentials} (also called \emph{spikes}) is produced at the beginning of the
axon, at the \emph{axon hillock} and the potential of the somatic membrane is reset to an equilibrium value. 
This sequence of action potentials is then transmitted, without alteration 
(shape or amplitude), to the axon terminals where the excitatory
connections (e.g. synapses) with other (target) neurons are located.
We ignore inhibitory synapses in this work. It takes some time for the 
action potential to reach a synapse and to cross it. The action potential propagates in every 
branch of the axon. When an action potential reaches a synapse, it triggers a local increase of the membrane 
potential of the dendrite of the target neuron. This electrical activity then propagates along the dendrite 
in both ways, (see Figure 3 in Gorski \emph{et al.} \cite{g} for a simulation of this behaviour)
i.e. to the soma and to the other dendrite extremity, interacting with the 
other electrical activity of the dendrite.
The dendritic current reaching the soma increases its potential.

\emph{Generation of spikes.} 
We need to introduce a little bit of biophysics, see \cite{k}. Consider a small patch of cellular
membrane (somatic, dendritic or axonic) which marks the boundary between the extra-cellular space and the 
intracellular one. This piece of membrane contains different ion channel types 
which govern the flow of different ion types through them. 
These ion channels (partly) affect the flow of charges 
locally, and thus the membrane potential. The ion channels rates of opening and closing
depend on the membrane potential $V$ of the small patch under consideration.
Hence the time evolution of $V$ is complicated, one needs to introduce a 
$4$-dimensional ODE system, called the Hodgkin-Huxley equations \cite{hh1}, 
see also Koch \cite{ko}, 
involving other quantities
related to ion channels.
If $V$ is large enough and if there are enough channels, a specific cascade of 
opening/closing of ion channels occurs and this produces a spike.
In the axon, only one sort of spike is possible. For the
dendrite, the situation is more complicated and only some types of neurons have dendrites
that are able to produce spikes.

\emph{Propagation/annihilation of spikes.}
The above description is local in space and we considered that the patch of membrane under consideration 
was isolated. To treat a full membrane, for example a dendrite, 
a nonlinear PDE is generally used, see e.g. Stuart, Spruston and Ha\"user \cite{s} or Koch \cite{ko},
to describe the membrane potential $V(t,x)$ at location $x$ at time $t\geq 0$
(and some other quantities related to the ion 
channels), with some source terms at the positions of the synapses. 
Fronts are particular \emph{localized} solutions of the form $V(t,x)=\psi(x-\rho t)$,
see \cite{ko}. For tubular geometries, a spike
induced in the middle of the membrane
will produce two fronts propagating in opposite directions. In the axon, the 
fronts are
produced only at one extremity (the soma) hence yielding only one propagating front. 

Two fronts propagating in 
opposite directions, in a given dendrite, will cancel out when they collide,
because after the initiation of a spike, 
some ion channels deactivate and switch into a refractory state for a small time. 
Some consequences of this annihilation effect, yet to be
confirmed experimentally, were analyzed in Gorski \emph{et al.} \cite{g}.

Instead of solving a nonlinear PDE for the front propagation/annihilation, 
we consider an abstract 
model which captures the basic phenomena. This enables us to have some formulas
for the number of fronts reaching the soma even when annihilations are 
considered. Note that the same rationale
was used for the axon where we only retained the propagation delay as 
meaningful. Our last approximation 
concerns the dynamics of an isolated neuron, at the soma, between the spikes. We replace the $4$-dimensional ODE system for $V$, 
e.g. the Hodgkin-Huxley equations, by a simpler 
scalar \emph{piecewise deterministic Markov  process} 
where the jumps represent the spiking times and the membrane potential evolves as 
$\dot V=F(V)$ between the spikes.

\emph{The toy model.} 
We are now in position to explain our toy model. Each action potential of an afferent 
neuron produces, after a constant delay $\theta$, two fronts in all the dendrites that are connected to 
the extremities of its axon.
In each dendrite, these fronts propagate and interact (by annihilation), and the ones
reaching the soma increase its membrane potential by a given amount $w_n$.
When the somatic membrane potential is 
high enough, an action potential is created.
Observe that in nature, several action potentials reaching a single synapse are required to produce fronts.

Let us stress one more time that the model we consider is highly schematic. 
Actually, dendrites are not linear segments with constant length, but have a dense branching structure;
dendritic spikes are not the only carriers of information;
inhibition (that we completely neglect) plays an important role;
the delay needed for the information to cross the axon and the synapse is far from being constant;
the spatial structure of interaction is much more complicated than mean-field, etc.
However, it seems this is one of the first attempts to understand the effect of active dendrites
in a neural network.

\subsection{Heuristic scales and relevant quantities}\label{hsrq}
(a) Roughly, each neuron is influenced by $N=np_n$ others and we naturally consider the asymptotic 
$N\to \infty$.

(b) Using a recent version by Calder, Esedoglu and Hero \cite{ceh} 
of some results of Deuschel and Zeitouni \cite{dz}
concerning the length of the longest increasing subsequence in a cloud of i.i.d. points in $[0,1]^2$, we will 
deduce the following result. Consider a single linear dendrite with length $L$, as well
as a Poisson point process $(T_i,X_i)_{i\geq 1}$ on $[0,\infty)\times[0,L]$, with intensity measure
$N g(t) \dd t H(x) \dd x$, $H$ being the repartition density defined in Subsection~\ref{mm} and
$g$ being the spiking rate of one typical neuron in the network. For
each $i\geq 1$, one positive and one negative front start from $X_i$ at time $T_i$. 
Make the fronts move with velocity $\rho>0$, apply the annihilation
rules described in Subsection~\ref{mm} and call $L_N(t)$ the number of excitation events occurring 
during $[0,t]$,
i.e. the number of fronts hitting the soma before $t$. Under a few assumptions on $H$ and $g$, in probability,
$$
\lim_{N\to\infty}\frac{L_N(t)}{\sqrt N} = \Gamma_t(g),
$$
where $\Gamma_t(g)$ is deterministic and more or less explicit, see Definition~\ref{dfG}. 
Of course, $\Gamma_t(g)$ also depends on $H$, but $H$ is fixed in the whole paper so we do not indicate
explicitly this dependence.

(c) We want to consider a regime in which each neuron spikes around once per unit of time.
This implies that on each dendrite, there are around $N$ fronts starting per unit of time. Due to point (b),
even if we are clearly not in a strict Poissonian case, it seems reasonable to think that there will be
around $\sqrt N$ excitation events per unit of time (for each neuron).
Consequently, each neuron will see its potential increased by $w_n\sqrt N$ per unit of time and we 
naturally consider the asymptotic $w_n\sqrt N \to w \in (0,\infty)$.
Smaller values of $w_n$ would make negligible the influence of the excitation events, while 
higher values of $w_n$ would lead to explosion (infinite frequency of spikes).

One could be surprised by this normalization $N^{-1/2}$ (and not $N^{-1}$ for example) which is the right 
scaling for the electric current from the dendrite to the soma to be non trivial as the number of synapses 
goes to infinity. 

\subsection{Goal of the paper}
Of course, the networks presented in Subsection~\ref{mm} are interacting particle systems.
However, the influence of a given neuron (say, the one labeled $2$) on another one 
(say, the one labeled $1$) being small (because the neuron $2$ 
produces only a proportion $1/N\ll 1$
of the fronts influencing the neuron $1$), we expect that some 
asymptotic
independence should hold true. Such a phenomenon is usually called \emph{propagation of chaos}.
Our aim is to prove that, \emph{assuming} propagation of chaos,
as well as some conditions on the parameters of the models,
there is a unique possible reasonable limit process, for each model, in the regime
$N=np_n\to\infty$ and $w_n \sqrt N \to w \in (0,\infty)$. 

The soft model seems both easier and more realistic from a modeling point of view. 
However, we keep the hard model because we are able
to provide, in a very special case, a rather explicit limit, which is moreover in some sense periodic.

\subsection{Informal description of the main result for the soft model}\label{ppppptit}

Consider one given neuron in the system (say, the one labeled $1$), call $V_t^{1,n}$ its potential 
at time $t$ and denote by $J_t^{1,n}$ (resp. $K_t^{1,n}$) 
its number of spike events (resp. excitation events) before time $t$.
We hope that, by a law of large numbers, for $n$ very large, $\kappa^{n}_t:=w_n K_t^{1,n} 
\simeq w N^{-1/2}K_t^{1,n}$
which represents the increase of electric potential before time $t$ due to excitation events, should 
resemble some deterministic quantity $\kappa_t$. The map $t\mapsto \kappa_t$ should be non-decreasing, 
continuous (because $w_n\to 0$, even if this is of course not a rigorous argument) 
and starting from $0$. 
We thus should have
$V_t^{1,n}\simeq V_0^{1,n} + \int_0^t F(V_s^{1,n})ds +  \kappa_t + 
\sum_{s \in [0,t]} (v_{min}-V_{s-}^{1,n})\indiq_{\Delta J_s^{1,n} \neq 0}$, with furthermore $J^{1,n}_t$ jumping at rate
$\lambda(V^{1,n}_t)$.
{\color{black} Moreover,} $\kappa_t$ should also be obtained as the approximate value of $w_n K_t^{1,n}$,
where $K_t^{1,n}$ is the number of excitation events before time $t$, resulting from the 
influence of $N=np_n$ (informally) almost independent neurons, all behaving like the one under study.

We thus formulate the following \emph{nonlinear} problem.
Fix an initial distribution $f_0\in\cP([v_{min},\infty))$ for $V_0$.
Can one find a deterministic non-decreasing continuous function $(\kappa_t)_{t\geq 0}$ starting from $0$ 
such that, if considering the process
$V_t=V_0+\int_0^t F(V_s)ds +  \kappa_t + \sum_{s \in [0,t]} (v_{min}-V_{s-})\indiq_{\Delta J_s \neq 0}$, 
with furthermore the counting process $J_t$ jumping at rate $\lambda(V_t)$
(all this can be properly written using Poisson measures),
if denoting by $(T_k)_{k\geq 1}$ its {jumping times}, 
if considering an i.i.d. family $(X_i)_{i=1,\dots,N}$ with density
$H$ and an i.i.d. family $((T_k^i)_{k\geq 1})_{i=1,\dots,N}$ of copies of $(T_k)_{k\geq 1}$, 
if making start, on a single linear dendrite with length $L$, one positive front and one negative front 
from $X_i$ (for all $i=1,\dots,N$) at each instant $T_k^i+\theta$ (for all $k\geq 1$), 
if making the fronts move with velocity $\rho$, if applying the annihilation 
procedure described in Subsection~\ref{mm} and, if denoting by $K^N_t$ the resulting number of excitation
events occurring during $[0,t]$,
one has $\lim_{N \to \infty}w N^{-1/2}K^N_t=\kappa_t$ for all
$t\geq 0$?

Under a few conditions on $f_0$, $F$, $\lambda$ and $H$, we prove the 
existence of a unique solution $(\kappa_t)_{t\geq 0}$ to the
above problem. Furthermore, the process $(V_t)_{t\geq 0}$ solves a nonlinear Poisson-driven stochastic 
differential equation and $\kappa_t = \Gamma_t((\E[\lambda(V_s)])_{s\geq 0})$.
Our conditions are very general when the delay $\theta$ is positive, and rather restrictive,
at least from a mathematical point of view, when $\theta=0$.

\subsection{Informal description of the main result for the hard model}\label{gggggros}

Similarly to the soft model,
we formulate the following problem.
Fix an initial distribution $f_0\in\cP([v_{min},v_{max}])$ for $V_0$.
Can one find a deterministic non-decreasing continuous function $(\kappa_t)_{t\geq 0}$ starting from $0$ 
such that, if considering the process
$V_t=V_0 + \int_0^t F(V_s)ds +  \kappa_t + (v_{min}-v_{max})J_t$, where $J_t=\sum_{s\leq t} \indiq_{\{V_\sm=v_{max}\}}$,
if denoting by $(T_k)_{k\geq 1}$ its instants of spike, 
if considering an i.i.d. family $(X_i)_{i=1,\dots,N}$ with density
$H$ and an i.i.d. family $((T_k^i)_{k\geq 1})_{i=1,\dots,N}$ of copies of $(T_k)_{k\geq 1}$, 
if making start, on a single linear dendrite with length $L$, one positive front and one negative front 
from $X_i$ (for all $i=1,\dots,N$) at each instant $T_k^i+\theta$ (for all $k\geq 1$), 
if making the fronts move with velocity $\rho$, if applying the annihilation 
procedure described in Subsection~\ref{mm} and, if denoting by $K^N_t$ the resulting number of excitation
events occurring during $[0,t]$,
i.e. the number of fronts hitting the soma before $t$, one has $\lim_{N \to \infty}w N^{-1/2}K^N_t=\kappa_t$ 
for all $t\geq 0$?

As already mentioned, we restrict our study of the hard model
to a special case for which we end up with an explicit formula. 
Namely, we assume that the delay $\theta=0$, that 
the continuous repartition density $H$ attains its maximum at $0$, that
the drift $F$ is constant and positive and that the initial distribution $f_0$ has a regular density 
(on $[v_{min},v_{max}]$ seen as a torus). 
We prove that, there is a unique $C^1$-function $(\kappa_t)_{t\geq 0}$ solving the
above problem. Furthermore, $(\kappa_t)_{t\geq 0}$ is explicit, 
see Theorem~\ref{mr1}.

The function $(\kappa'_t)_{t\geq 0}$ is periodic. Observe that $\kappa'_t$ is proportional to the number of
excitation events (concerning a given neuron) during $[t,t+\dd t]$.
This suggests a synchronization phenomenon, or rather some stability of possible synchronization,
which is rather natural, since two neurons having initially the same potential 
spike simultaneously forever in this model.
Observe that such a periodic behavior cannot precisely
hold true for the particle system (before taking the limit $N\to \infty$)
because the dendrites are assumed to be empty at time $0$, so that some time is needed before some
(periodic) equilibrium is reached.

\subsection{Bibliographical comments}

Kac \cite{ka} introduced the notion of propagation of chaos
as a step toward the mathematical derivation of the Boltzmann equation. Some important steps of the general 
theory were made by McKean \cite{mk} and Sznitman \cite{sz}, see also M\'el\'eard \cite{m}.
The main idea is to approximate the time evolution of one particle, interacting with a large number of 
other particles, by the solution to a nonlinear equation. We mean \emph{nonlinear} in the sense of McKean,
i.e. that the law of the process is involved in its dynamics.
Here, our limit process $(V_t)_{t\geq 0}$ indeed solves a nonlinear stochastic differential equation, at least
concerning the soft model, see Theorem~\ref{mr2}. This nonlinear SDE is very original:
the nonlinearity is given by the functional $\Gamma(g_{(V_s)_{s\geq 0}})$ quickly described in Subsection~\ref{hsrq}, 
arising as a scaling 
limit of the longest subsequence in an i.i.d. cloud of points of which the distribution depends on a function
$g_{(V_s)_{s\geq 0}}$, which depends itself on the law of $(V_s)_{s\geq 0}$.

The problem of computing the length $L_N$ of the longest increasing sequence in a random permutation of
$\{1,\dots,N\}$ was introduced by Ulam \cite{ulam}. Hammersley \cite{h} understood that a clever way to
attack the problem is to note that $L_N$ is also the length of the longest
increasing sequence of a cloud composed of $N$ i.i.d. points uniformly distributed in the square $[0,1]^2$,
for the usual partial order in $\R^2$. He also proved the existence of a constant $c$ such that
$L_N \sim c \sqrt N$ as $N\to \infty$. Versik and Kerov \cite{vk} and Logan and Shepp \cite{ls} 
showed that $c=2$. Simpler proofs and/or stronger results were then found by Bollob\'as and Winkler \cite{bw}, 
Aldous and Diaconis \cite{ad}, Cator and Groeneboom \cite{cg}, etc. Let us also mention the recent work of
Basdevant, Gerin, Gou\'er\'e and Singh \cite{bggs}.

As already mentioned in Subsection~\ref{hsrq}, we use the results of Calder, Esedoglu and Hero \cite{ceh}, 
that generalize those of 
Deuschel and Zeitouni \cite{dz} and that concern the limit behavior of the longest ordered
increasing sequence of a cloud composed of $N$ i.i.d. points with general smooth distribution $g$ 
in the square $[0,1]^2$ (or in a compact domain). These results strongly rely on the fact that since $g$ is 
smooth, it is almost constant on small squares. Hence, on any small square, we can more or less apply
the results of \cite{vk,ls}. Of course, this is technically involved, but the main difficulty in all this work 
was to understand the constant $2$ (note that the value of the corresponding constant is still unknown in 
higher dimension).

Of course, a little work is needed: we cannot apply directly the results of \cite{ceh}, 
because we are not in presence of an i.i.d. cloud. However, as we will see, the situation is rather favorable.

The mean-field theory in networks of spiking neurons has been 
studied in the computational neuroscience community, see 
e.g. Renart, Brunel and Wang \cite{rbw}, Ostojic, Brunel and Hakim \cite{obh}
and the references therein.
A mathematical approach of mean-field effects in neuronal activity has  also
been developed. For instance, in Pakdaman, Thieullen and Wainrib \cite{ptw} and Riedler, Thieullen and 
Wainrib \cite{rtw_2012}, a class of stochastic hybrid systems is rigorously proved
to converge to some fluid limit equations.
In \cite{bft}, Bossy, Faugeras, and Talay prove similar results and 
the propagation of chaos property for networks of 
Hodgkin-Huxley type neurons with an additive white noise perturbation.
The mean-field limits of networks of spiking neurons
modeled by Hawkes processes has been intensively studied recently by 
Chevallier, C\'aceres, Doumic and Reynaud-Bouret \cite{ccdr}, Chevallier \cite{c2017},
Chevallier, Duarte, L\"ocherbach and Ost \cite{cdlo_2017} and Ditlevsen and L\"ocherbach \cite{dl2017}.
Besides, in \cite{l2011,lc2014,lc2016}, Lu\c{c}on and Stannat
obtain asymptotic results for networks of interacting neurons
in random environment.

Finally, we conclude this short bibliography of mathematical mean-field models
in neuroscience by some papers closer to our setting: models of networks of spiking neurons with soft 
(see De Masi {\it et al.} \cite{aaee} and Fournier and L\"ocherbach \cite{fl}) or hard 
(see C{\'a}ceres,  Carrillo and Perthame \cite{ccp}, Carrillo, Perthame, Salort and Smets \cite{cpss},
Delarue, Inglis, Rubenthaler and Tanr\'e \cite{dirt,dirt2} and Inglis and Talay \cite{it}) 
bounds on the membrane potential have also been studied. In particular, in \cite{it} the authors
introduced a model of propagation of membrane potentials along the dendrites but it is very 
different from ours. In particular, it does not model the annihilation of fronts
along the dendrites.

\subsection{Perspectives}
One important question remains open: does propagation of chaos hold true?
This seems very difficult to prove rigorously.
Indeed, the dynamics of the membrane potential  at the  soma depends also on the 
state of its dendrite and on their laws.
Thus, the state space of the dendrite is not a classical \(\mathbb{R}^d\).
Informally, the knowledge of the state of the dendrite is equivalent to knowing the
history of the membrane potential during time interval \([t-L/\rho,t]\). 
Such an intricate dependence is present in many models for which one is able to prove propagation of chaos.
However, in the present case, one would have to extend the results of Deuschel 
and Zeitouni \cite{dz} or Calder, Esedoglu and Hero \cite{ceh} to 
non-independent (although approximately independent)
clouds of random points, in order to understand how many excitation events 
occur for each neuron,
resulting from non-independent \textit{stimuli} creating fronts on its dendrite.
This seems extremely delicate, and we found no notion of \textit{approximate 
	independence} sufficiently strong so that we can extend the results 
of \cite{dz,ceh} but weak enough so that we can apply
it to our particle system.
%
%

\subsection{Plan of the paper}
In the next section, we precisely state our main results.
In Section~\ref{otfa}, we relate deterministically the number of fronts hitting a given soma to 
the length of the longest increasing (for some specific order) subsequence of the points
(time and space) from which these fronts start.
In Section~\ref{central}, which is very technical, we adapt to our context the result of
Calder, Esedoglu and Hero \cite{ceh}.
The proofs of our main results concerning the hard and soft models are handled in
Sections~\ref{ph} and \ref{ps}.
{We informally discuss the existence and uniqueness/non-uniqueness of stationary solutions
for the limit soft model in Section~\ref{id}.}
Finally, we present simulations, in Section~\ref{num}, showing that the 
particle systems described in Subsection~\ref{mm} indeed seem to be well-approached, when $n$ is large, 
by the corresponding limiting processes.

\subsection*{Acknowledgment}
We warmly thank the referees for their fruitful comments.
E. Tanr\'e and R.~Veltz   have received funding from the European Union's Horizon 2020 Framework Programme for Research and Innovation under the Specific Grant Agreement No. 785907 (Human Brain Project SGA2).

\section{Main result}

Here we expose our notation, assumptions and results in details.
The length $L>0$, the speed $\rho>0$ and the minimum potential $v_{min}$ are fixed.

\subsection{The functional $A$}\label{fa}
We first study the number of fronts hitting the soma of a linear dendrite.
We recall that a nonnegative measure $\nu$ on $[0,\infty)\times [0,L]$ is \emph{Radon}
if $\nu(B)<\infty$ for all compact subset $B$ of $[0,\infty)\times [0,L]$.

\begin{defin}\label{A}
We introduce the partial order $\preceq$ on $[0,\infty)\times [0,L]$ defined by 
$(s,x)\preceq (s',x')$ if $|x-x'| \leq \rho(s'-s)$. We say that $(s,x)\prec (s',x')$ if $(s,x)\preceq (s',x')$
and $(s,x)\neq (s',x')$.

For a Radon point measure $\nu=\sum_{i\in I} \delta_{M_i}$,
the set $\cS_\nu=\{M_i:i \in I\}$ consisting of distinct points of $[0,\infty)\times [0,L]$, we
define $A(\nu)\in\N\cup\{\infty\}$ 
as the length of the longest increasing subsequence of $\cS_\nu$.
In other words, 
$A(\nu)=\sup\{k\geq 0 :$ there exist $i_1,\dots,i_k \in I$ such that 
$M_{i_1}\prec\dots\prec M_{i_k}\}$.
For $t\geq 0$, we introduce $D_t=\{(s,x)\in [0,\infty)\times [0,L] : (s,x) \preceq (t,0)\}$
and set $A_t(\nu)=A(\nu|_{D_t})$.
\end{defin}

\noindent Note that $(s,x)\preceq (s',x')$ implies that $s\leq s'$.
The following fact, crucial to our study, is closely linked with Hammersley's lines,
see e.g. Cator and Groeneboom \cite{cg}.

\begin{prop}\label{tac}
Consider a  Radon point measure $\nu=\sum_{i\in I} \delta_{M_i}$,
the set $\cS_\nu=\{M_i=(t_i,x_i):i \in I\}$ consisting of distinct points of $[0,\infty)\times [0,L]$.
Consider a linear dendrite, represented by the segment $[0,L]$, with its soma located at $0$.
For each $i\in I$, make start two fronts from $x_i$ at time $t_i$, one \emph{positive front} going 
toward the soma
and one \emph{negative front} going away from the soma. Assume that all the fronts move with velocity $\rho$.
When two fronts meet, they disappear. When a front reaches one of the extremities of the dendrite, 
it disappears.

We assume that 
\begin{equation}\label{condich}
\hbox{for all $i,j\in I$ with $i\neq j$, $|x_j-x_i|\neq\rho|t_j-t_i|$,}
\end{equation} 
which implies that no front may start precisely from some (space/time) position where there is already a front.
Hence we do not need to prescribe what to do in such a situation.

The number of fronts hitting the soma is given by $A(\nu)$ and
the number of fronts hitting the
soma before time $t$ is given by $A_t(\nu)$.
\end{prop}

This proposition is proved in Section~\ref{otfa}.
The following observation is obvious by definition (although not completely obvious from the point of view of
fronts).

\begin{rk}\label{order}
Consider two Radon point measures $\nu$ and $\nu'$ on $[0,\infty)\times [0,L]$
such that $\nu \leq \nu'$ (i.e. $\cS_\nu \subset \cS_{\nu'}$).
Then $A(\nu)\leq A(\nu')$
and, for all $t\geq 0$,  $A_t(\nu)\leq A_t(\nu')$.
\end{rk}

\subsection{The functional $\Gamma$}
The role of $\Gamma$ was explained roughly in Subsection~\ref{hsrq},
see Section~\ref{central} for more details. See Deuschel and Zeitouni~\cite{dz} for quite similar 
considerations.

\begin{defin}\label{dfG}
Fix a continuous function $H:[0,L]\mapsto \R_+$.
For $g:[0,\infty)\mapsto \R_+$ measurable and $t\geq 0$, we set
$$
\Gamma_t(g)=\sup_{\beta \in \cB_t} \cI_t(g,\beta) \quad \hbox{where} \quad \cI_t(g,\beta)=
\sqrt{\frac 2\rho} \intot \sqrt{H(\beta(s))
g(s)[\rho^2-(\beta'(s))^2]} \dd s,
$$
$\cB_t$ being the set of $C^1$-functions $\beta:[0,t]\mapsto [0,L]$ 
such that $\beta(t)=0$ and $\sup_{[0,t]} |\beta'(s)|< \rho$.
\end{defin}

It is important, in the above definition, to require $H$ to be continuous. Modifying the value of $H$
at one single point can change the value of $\Gamma_t(g)$.
The following observations are immediate.

\begin{rk}\label{ttip}
(i) Consider $\beta \in \cB_t$. 
The condition that $\sup_{[0,t]} |\beta'(s)| <\rho$ implies that the map $s\mapsto (s,\beta(s))$
is increasing for the order $\prec$. The conditions that $\beta$ is $[0,L]$-valued and that  
$\beta(t)=0$ imply that for all $s\in [0,t]$,
$(s,\beta(s)) \in D_t$.

(ii) If $H(0)=\max_{[0,L]} H$, then $\Gamma_t(g)=\sqrt{2\rho H(0)}\int_0^t\sqrt{g(s)}\dd s$ for all $t\geq 0$.
\end{rk}

Concerning (ii), it suffices to note that one maximizes $\cI_t(g,\beta)$ with the choice $\beta\equiv 0$.

\subsection{The soft model}
We will impose some of the following conditions.

(S1): There are $p\geq 1$ and $C>0$ such that the  initial distribution
$f_0 \in \cP([v_{min},\infty))$ satisfies $\int_{v_{min}}^\infty (v-v_{min})^p f_0(\dd v)<\infty$
and such that the continuous rate function $\lambda:[v_{min},\infty)\mapsto \R_+$ satisfies 
$\lambda(v) \leq  C(1+(v-v_{min}))^p$ for all $v\geq v_{min}$, $C>0$ being a constant.
Also, $\lambda$ vanishes on a neighborhood of $v_{min}$, i.e. 
$\alpha=\inf\{v\geq v_{min}:\lambda(v)>0\} \in (v_{min},\infty)$.
The drift $F:[v_{min},\infty)\mapsto \R$ is locally Lipschitz continuous, satisfies $F(v_{min})\geq 0$
and $F(v) \leq C(1+(v-v_{min}))$ for all $v\geq v_{min}$. 
The repartition density $H$ of the connections is continuous on $[0,L]$.

(S2) The initial distribution $f_0$ is compactly supported, $f_0((\alpha,\infty))>0$,
$F(\alpha)\geq 0$ and $\lambda$ is locally Lipschitz continuous on $[v_{min},\infty)$ and positive and 
non-decreasing on $(\alpha,\infty)$.

\begin{prop}\label{debil2}
Assume (S1).
Consider $r:[0,\infty)\mapsto \R_+$ continuous, non-decreasing and such that $r_0=0$. 
Let $V_0$ be $f_0$-distributed and let $\pi(\dd t,\dd u)$ be a Poisson measure on 
$[0,\infty)\times[0,\infty)$
with intensity $\dd t \dd u$, independent of $V_0$.
Let $\cF_t=\sigma(\{V_0,\pi(A) : A\in\cB([0,t]\times[0,\infty))\})$.
There is a pathwise unique c\`adl\`ag $(\cF_t)_{t\geq 0}$-adapted process $(V_t^r)_{t\geq 0}$ 
solving
\begin{equation}\label{sde}
V^r_t=V_0+\int_0^t F(V_s^r)\dd s + r_t + \intot\int_0^\infty (v_{min}-V^r_\sm)\indiq_{\{u\leq \lambda(V^r_\sm)\}}
\pi(\dd s,\dd u).
\end{equation}
It takes values in $[v_{min},\infty)$ and satisfies  $\E[\sup_{[0,T]} (V_t^r-v_{min})^p]<\infty$ for all $T>0$.
We set $J^r_t=\sum_{s\leq t} \indiq_{\{\Delta V^r_{s}\neq 0\}}=\intot\int_0^\infty 
\indiq_{\{u\leq \lambda(V^r_\sm)\}}\pi(\dd s,\dd u)$. 
\end{prop}

The process $(V^r_t)_{t\geq 0}$ represents the time evolution of the
potential of one neuron, \emph{assuming} that the excitation resulting from the 
interaction with all the other neurons during $[0,t]$ produces an increase of potential equal to $r_t$,
and $J^r_t$ stands for its number of  spikes during $[0,t]$.
Indeed, between its spike instants, the electric potential $V^r_t$ evolves as $V'=F(V)+r_t'$.
The Poisson integral is precisely designed so that $V^r$ is reset to $v_{min}$
(since $V^r_{s-}+(v_{min}-V^r_\sm)=v_{min}$) at rate $\lambda(V^r_\sm)$.

\begin{prop}\label{prel2}
Assume (S1) and fix $\theta \geq 0$. 
Fix a non-decreasing continuous function $r:[0,\infty)\mapsto \R_+$ with $r_0=0$.
Consider an i.i.d. family $((J^{r,i}_t)_{t\geq 0})_{i\geq 1}$ of copies of $(J^r_t)_{t\geq 0}$. For each $i\geq 1$,
denote by $(T^i_k)_{k\geq 1}$ the jump instants of $(J^{r,i}_t)_{t\geq 0}$, written in the 
chronological order (i.e. $T_k^i < T_{k+1}^i$).
Consider an i.i.d. family $(X_i)_{i\geq 1}$ of random variables with density $H$, independent of the family
$((J^{r,i}_t)_{t\geq 0})_{i\geq 1}$. For $N \geq 1$, let  $\nu^r_N=\sum_{i=1}^N \sum_{k\geq 1} 
\delta_{(T^i_k+\theta,X_i)}$. Then for any $t\geq 0$,
$$
\lim_{N \to \infty} N^{-1/2} A_t(\nu_N^r)=\Gamma_t(h_r^\theta) \hbox{ a.s.,}
$$
where $h_r(t)=\E[\lambda(V^r_t)]$ and $h_r^\theta(t)=h_r(t-\theta)\indiq_{\{t\geq \theta\}}$.
\end{prop}

Let us explain this result. If we have $N$ independent neurons of 
which the electric potentials evolve as $(V^r_t)_{t\geq 0}$, of which $(J^r_t)_{t\geq 0}$ counts the 
number of spikes, if all these spikes make start some fronts (after a delay $\theta)$
on the dendrite of another neuron and 
that these fronts
evolve and annihilate as described in Proposition~\ref{tac}, then the number of fronts hitting the soma 
of the neuron under consideration between $0$ and $t$ 
equals $A_t(\nu_N^r)$. If each of these excitation events makes increase
the potential of the neuron by $w_N=w N^{-1/2}$ (with $w>0$), then, at the limit, the electric potential of 
the neuron will be increased, due to excitation, by $w \Gamma_t(h_r^\theta)$ during $[0,t]$.

\begin{theo}\label{mr2}
Assume (S1) and fix $w>0$ and $\theta \geq 0$.

(i) A non-decreasing continuous 
$\kappa:[0,\infty)\mapsto \R_+$ such that $\kappa_0=0$ solves $w\Gamma_t(h_\kappa^\theta)=\kappa_t$ 
for all $t\geq 0$
if and only if $\kappa=w \Gamma((\E[\lambda(V_{s-\theta})]\indiq_{\{s\geq \theta\}})_{s\geq 0})$, 
for some $[v_{min},\infty)$-valued c\`adl\`ag $(\cF_t)_{t\geq 0}$-adapted
solution $(V_t)_{t\geq 0}$ to the nonlinear SDE
(here $V_0$, $\pi$ and $(\cF_t)_{t\geq 0}$ are as in Proposition~\ref{debil2})
\begin{equation}\label{nsde}
V_t \!=\!V_0+\int_0^t F(V_s) \dd s + w \Gamma_t((\E[\lambda(V_{s-\theta})]\indiq_{\{s\geq \theta\}})_{s\geq 0})
+ \intot\int_0^\infty (v_{min}-V_\sm)\indiq_{\{u\leq \lambda(V_\sm)\}}\pi(\dd s,\dd u)
\end{equation}
satisfying $\E[\sup_{[0,T]} (V_t-v_{min})^p]<\infty$ for all $T>0$.

(ii) Assume either that $\theta>0$ or (S2).
Then there exists a unique $[v_{min},\infty)$-valued c\`adl\`ag $(\cF_t)_{t\geq 0}$-adapted
solution $(V_t)_{t\geq 0}$ to \eqref{nsde}
such that for all $T>0$, $\E[\sup_{[0,T]} (V_t-v_{min})^p]<\infty$.
\end{theo}

Observe that if the repartition density $H$ attains its maximum at $0$, then \eqref{nsde} has a simpler
form, since 
$w \Gamma_t((\E[\lambda(V_{s-\theta})]\indiq_{\{s\geq \theta\}})_{s\geq 0})=
\int_{0}^{(t-\theta)\lor 0} \sqrt{\gamma \E[\lambda(V_s)]} \dd s$
with $\gamma=2 \rho H(0) w^2$, see Remark~\ref{ttip}. In this case \eqref{nsde} writes: 
\begin{equation*}
V_t \!=\!V_0+\int_0^t F(V_s) \dd s + \int_{0}^{(t-\theta)\lor 0} \sqrt{\gamma \E[\lambda(V_s)]} \dd s
+ \intot\int_0^\infty (v_{min}-V_\sm)\indiq_{\{u\leq \lambda(V_\sm)\}}\pi(\dd s,\dd u),
\end{equation*}
where the second term involves the non locally Lipschitz square root.

Consider the $n$-particle system described in Subsection~\ref{mm} (soft model) and denote by
$(V^{1,n}_t)_{t\geq 0}$ the time-evolution of the membrane potential of the first neuron and by
$(J_t^{1,n})_{t\geq 0}$ the process counting its spikes.
Theorem~\ref{mr2} tells us that, \emph{if propagation of chaos holds true}, under our assumptions, 
$(V_t^{1,n})_{t\geq 0}$ should tend in law (in the regime $N=np_n \to \infty$ and $w_n=wN^{-1/2}$)
to the unique solution $(V_t)_{t\geq 0}$ of \eqref{nsde}. See
Subsection~\ref{ppppptit} for more explanations.

Assumption (S1) seems rather realistic.
Our assumption  that $\lambda$ vanishes in a neighborhood of
$v_{min}$ actually implies that a neuron cannot spike again immediately after one spike.
Indeed, after being set to \(v_{min}\), we observe a refractory period corresponding to
the time the potential needs to exceed \(\alpha\). In addition, it allows us to consider
some time intervals \([a_k,a_{k+1}]\), in our proof of Proposition~\ref{prel2}, such that the
restriction of \(\nu_N^r\) to \([a_k,a_{k+1}]\times[0,L]\) is more or less an i.i.d. cloud of random points.
This is crucial in order to use the results of Calder, Esedoglu and Hero \cite{ceh},
who deal with i.i.d. clouds of random points. More precisely, the proof of Proposition~\ref{prel2}
(as well as that of Proposition~\ref{prel1} below) relies on Lemma~\ref{lcru}, in which we show
how to apply \cite{ceh} (or rather its immediate consequence Lemma~\ref{base})
to a possibly correlated {\it concatenation} of i.i.d. clouds of random points.

The growth condition on $F$ is one-sided and sufficiently general to our opinion, 
however, it is only here to prevent
us from explosion (we mean an infinite number of jumps during a finite time interval)
and it should be possible to replace it by weaker condition like 
$F(v) \leq (v-v_{min})\lambda(v)+C(1+(v-v_{min}))$, at the price of more complicated proofs.
So we believe that when $\theta>0$, our assumptions are rather reasonable.

On the contrary, when $\theta=0$, our conditions are restrictive, at least from a mathematical point of view. 
This comes from two problems when studying
the nonlinear SDE \eqref{nsde}. First, the term 
$\Gamma_t((\E[\lambda(V_{s})])_{s\geq 0})$ involves something like $\int_0^t\sqrt{\E[\lambda(V_{s})]}\dd s$,
and the square root is rather unpleasant. To solve this  problem, we use that 
$f_0((\alpha,\infty))>0$ and $F(\alpha)\geq 0$
imply that $s\mapsto \E[\lambda(V_{s})]$ is \emph{a priori} bounded from below on each compact time interval.
Since $\alpha$ is thought to be rather close to $v_{min}$, we believe these two conditions are not too 
restrictive in practice. Second, the coefficients of \eqref{nsde} are only locally Lipschitz continuous,
which is always a problem for nonlinear SDEs. Here we roughly solve the problem by assuming that
$f_0$ is compactly supported, which propagates with time. Again, we believe this is not too restrictive in
practice, since $F(v)$ should rather tend to $-\infty$ as $v\to \infty$ and in such a case, it should not be
difficult to show that any invariant distribution for \eqref{nsde} has a compact support.
However, one may use the ideas of \cite{fl} to remove this compact support assumption, here again,
at the price of a much more complicated proof.

\subsection{The hard model}
This case is generally difficult, but under the following quite restrictive assumptions
and when $\theta=0$,
it has the advantage to be explicitly solvable.

(H1): The initial distribution $f_0 \in \cP([v_{min},v_{max}])$
has a density, still denoted by $f_0$, continuous on $[v_{min},v_{max}]$.
The repartition density $H$ is continuous on $[0,L]$.
There is a constant $I>0$ such that the drift $F(v)=I$ for all $v \in [v_{min},v_{max}]$.

(H2): The density $f_0$ satisfies $f_0(v_{min})=f_0(v_{max})$,
the repartition density $H$ attains its maximum at $x=0$ and, setting $\sigma=\rho H(0) w^2$ 
the function $G_0=\sigma f_0+\sqrt{\sigma^2 f_0^2+ 2\sigma I  f_0}$
is Lipschitz continuous on $[v_{min},v_{max}]$.

Note that if the density \(f_0\) is Lipschitz continuous and bounded from below by a positive constant, then
\(G_0\) is also Lipschitz continuous.

\begin{prop}\label{debil1}
Assume only that $f_0\in \cP([v_{min},v_{max}))$
and consider a $f_0$-distributed random variable $V_0$.
For a continuous non-decreasing function $r:[0,\infty)\mapsto \R_+$ with $r_0=0$ there is
a unique c\`adl\`ag process $(V_t^r,J_t^r)_{t\geq 0}$, with values in $[v_{min},v_{max})\times \N$ solving
$$
V^r_t=V_0+It+r_t+(v_{min}-v_{max})J_t^r, \quad and \quad J_t^r=\sum_{s\leq t} \indiq_{\{V_\sm^r = v_{max}\}}.
$$
\end{prop}

Again, $(V^r_t)_{t\geq 0}$ represents the time-evolution of the
potential of one neuron, \emph{assuming} that the excitation resulting from the 
interaction with all the other neurons during $[0,t]$ produces, in the asymptotic where there
are infinitely many neurons, an increase of potential equal to $r_t$. And of course,
$J^r_t$ stands for the number of times the neuron under consideration spikes during $[0,t]$.

\begin{prop}\label{prel1}
Assume (H1) and fix a non-decreasing $C^1$-function 
$r:[0,\infty)\mapsto \R_+$ with $r_0=0$.
Consider an i.i.d. family $((J^{r,i}_t)_{t\geq 0})_{i\geq 1}$ of copies of $(J^r_t)_{t\geq 0}$ as introduced
in Proposition~\ref{debil1}. For each $i\geq 1$,
denote by $(T^i_k)_{k\geq 1}$ the jump instants of $(J^{r,i}_t)_{t\geq 0}$, written in the 
chronological order.
Consider an i.i.d. family $(X_i)_{i\geq 1}$ of random variables with density $H$, independent of the family
$((J^{r,i}_t)_{t\geq 0})_{i\geq 1}$. For $N \geq 1$, let $\nu^r_N=\sum_{i=1}^N \sum_{k\geq 1} 
\delta_{(T^i_k,X_i)}$.
For any $t\geq 0$,
$$
\lim_{N \to \infty} N^{-1/2} A_t(\nu_N^r)=\Gamma_t(g_r) \hbox{ a.s.,}
$$
with $\Gamma_t$ introduced in Definition~\ref{dfG} and with $g_r$ defined on $[0,\infty)$ by
$$
g_r(t)=\sum_{k\geq 0} f_0(k(v_{max}-v_{min})+v_{max}-It-r_t)(I+r'_t) \indiq_{\{t\in [a_k,a_{k+1})\}}, 
$$
with
$a_k$ uniquely defined by $I a_k+r_{a_k}=k(v_{max}-v_{min})$  (observe that $0=a_0<a_1<a_2<\dots$).
\end{prop}

Assume that we have $N$ independent neurons, of 
which the electric potentials evolve as $(V^r_t)_{t\geq 0}$ and that spike as $(J^r_t)_{t\geq 0}$. 
If all these spikes make start, without delay,
some fronts on the dendrite of another neuron and that these fronts 
evolve and annihilate as described in Proposition~\ref{tac}, then the number of fronts hitting the soma 
(of the neuron under consideration) equals $A_t(\nu_N^r)$. If each of these excitation events makes increase
the potential of the neuron by $w_N=w N^{-1/2}$ (with $w>0$), 
then, at the limit, the electric potential of the neuron
will be increased, due to excitation, by $w \Gamma_t(g_r)$ during $[0,t]$.

\begin{theo}\label{mr1}
Assume (H1)-(H2) and let $w>0$. There exists a unique non-decreasing $C^1$-function
$\kappa:[0,\infty)\mapsto \R_+$ such that $\kappa_0=0$ and $w \Gamma_t(g_\kappa)=\kappa_t$
for all $t\geq 0$.
Furthermore,
$$
\kappa_t=\sum_{k\geq 0} \big( k(v_{max}-v_{min})+v_{max}-It-\varphi_0^{-1}(t-ka)\big)
\indiq_{\{t \in [ka,(k+1)a)\}},
$$
where, with $G_0$ was defined in (H2),
we have set $\varphi_0(x)=\int_x^{v_{max}} \dd v/[G_0(v)+I]$ on $[v_{min},v_{max}]$ and
$a=\varphi_0(v_{min})$. Observe that $\varphi_0^{-1}$ is defined on $[0,a]$, and that
$\kappa'$ is periodic with period $a$.
\end{theo}

Consider the $n$-particle system described in Subsection~\ref{mm} (hard model), under the conditions
(H1)-(H2) and with $\theta=0$.
Denote by
$(V^{1,n}_t)_{t\geq 0}$ the time-evolution of the electric potential of the first neuron and by
$(J_t^{1,n})_{t\geq 0}$ the process counting its spikes.
Theorem~\ref{mr1} tells us that, \emph{if propagation of chaos holds true}, 
$(V_t^{1,n},J^{1,n}_t)_{t\geq 0}$ should tend in law (in the regime $N=np_n \to \infty$ and $w_n=wN^{-1/2}$)
to $(V_t^\kappa,J_t^\kappa)_{t\geq 0}$ as defined in Proposition~\ref{debil1} and with the 
above explicit
$\kappa$. See Subsection~\ref{gggggros} for a discussion, in particular concerning the noticeable
fact that $\kappa'$ is periodic.

The assumptions that $\theta=0$, that $F(v)=I$ and that $H(0)=\max_{[0,L]} H$ are crucial, at least to get
an explicit formula. It might be possible to study the case where $F(v)=I-Av$ for some $A>0$ 
(maybe with the condition $I-Av_{max}>0$), but it does not seem so friendly.
On the contrary, we assumed for convenience that $f_0(v_{min})=f_0(v_{max})$, which guarantees that 
$\kappa$ is of class $C^1$.
This assumption seems rather reasonable because the 
potentials directly jump from $v_{max}$ to $v_{min}$ so are in some sense valued
in the \emph{torus} $[v_{min},v_{max})$. However, it may be possible to relax it.

\section{Annihilating fronts and longest subsequences}\label{otfa}

The goal of this section is to prove Proposition~\ref{tac}.
We first introduce a few notation. For $M=(r,y) \in [0,\infty)\times [0,L]$, we denote by 
$M_s=r$ its time coordinate and by $M_x=y$ its space coordinate.
We recall that $M \preceq N$ if $|M_x-N_x|\leq \rho (N_s-M_s)$, which means that $N$
belongs to the cone with apex $M$ delimited by the half-lines 
$\{(r,M_x+\rho (r-M_s)) : r\geq M_s\}$ and $\{(r,M_x-\rho (r-M_s)) : r\geq M_s\}$; and that 
$M\prec N$ if $M\preceq N$ and $M\neq N$. 
We say that $M\perp N$ if $M$ and $N$ are not comparable, i.e. if neither $M\preceq N$ nor $N \preceq M$.
Observe that $M \perp N$ if and only if $|M_x-N_x|> \rho |M_s-N_s|$, whence in particular $M_x\neq N_x$ and 
$|M_s-N_s|\leq L/\rho$.

For $M \in [0,\infty)\times [0,L]$, we introduce the four sets, see Figure~\ref{fs},
\begin{gather*}
M^{\downarrow} = \{Q\in [0,\infty)\times [0,L] : Q \prec M\},\quad
M^\uparrow = \{Q\in[0,\infty)\times [0,L] : M \prec Q\}, \\
M^+\!=\{ Q\in[0,\infty)\times [0,L] : Q\perp M, Q_x>M_x \},\quad  M^-\!=\{Q\in[0,\infty)\times [0,L] : 
Q\perp M, Q_x<M_x\}.
\end{gather*}

\begin{figure}[ht]
\noindent\fbox{\begin{minipage}{0.44\textwidth}
\centerline{\includegraphics[width=0.8\textwidth]{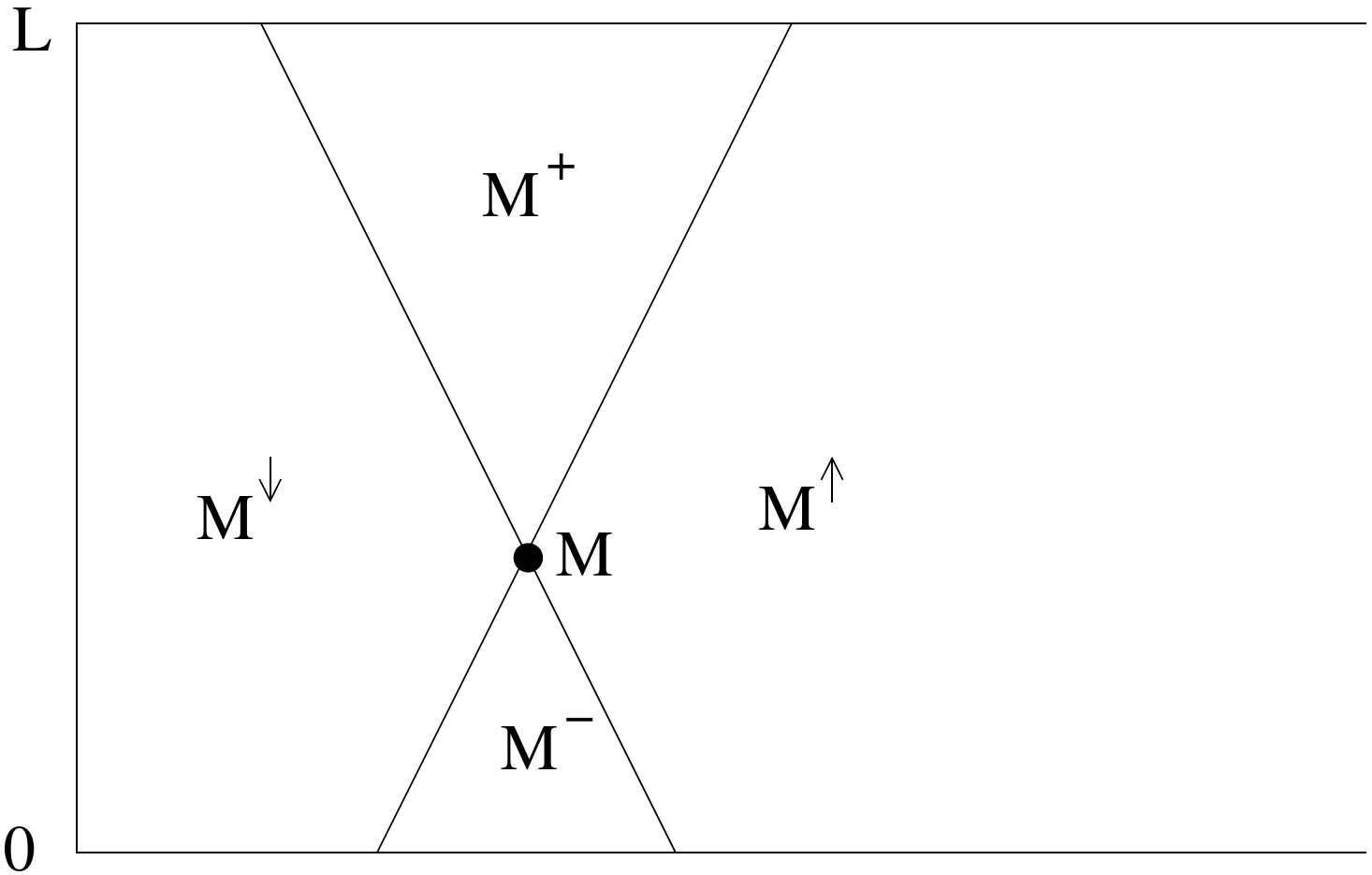}}
\end{minipage}
\hfill
\begin{minipage}{0.53\textwidth}
\caption{\label{fs}\small{We have drawn the four sets $M^{\downarrow}$, $M^\uparrow$, $M^+$ and $M^-$. 
The two oblique segments have slopes $\rho$ and $-\rho$.}}
\end{minipage}}
\end{figure}

\begin{proof}[Proof of Proposition~\ref{tac}]
Let $\nu=\sum_{i\in I} \delta_{M_i}$ be Radon, 
the set $\cS_\nu=\{M_i=(t_i,x_i):i \in I\}$ consisting of distinct points of 
$[0,\infty)\times [0,L]$.
We assume that $\nu \neq 0$ (because otherwise the result is obvious) and \eqref{condich}.
We recall that $A(\nu)\in\N\cup\{\infty\}$ and $A_t(\nu)\in\N$ were introduced in Definition~\ref{A}.
We call $B(\nu)\in\N\cup\{\infty\}$ the total number of fronts hitting the soma and $B_t(\nu)\in\N$ 
the number of fronts hitting the soma
before $t$. 

If two fronts start from $M$ (i.e. start from $M_x\in[0,L]$ at time $M_s\geq 0$), 
the positive one is, if not previously annihilated, 
at position $M_x-\rho(r-M_s)$ at time $r \in [M_s, M_s+M_x/\rho)$ and hits the soma at time $M_s+M_x/\rho$; the 
negative one is, if not previously annihilated, at position $M_x+\rho(r-M_s)$ at time 
$r \in [M_s, M_s+(L-M_x)/\rho)$ 
and disappears at time $M_s+(L-M_x)/\rho$.

We have the two following rules: for two distinct points $M,N \in \cS_\nu$,

(a) if $M\prec N$, i.e. $M \in N^\downarrow$ or, equivalently, $N \in M^\uparrow$, 
the fronts starting from $M$ cannot meet those
starting from $N$. Indeed, $M\prec N$ and
\eqref{condich} imply that 
$|M_x-N_x|<\rho(N_s-M_s)$ and a little study shows that for all $r\geq N_s$, $\{M_x-\rho (r-M_s), M_x+\rho (r-M_s)\}
\cap \{ N_x-\rho (r-N_s),N_x+\rho (r-N_s)\}=\emptyset$;

(b) if $M \perp N$ and $M_x<N_x$ (i.e. if $M \in N^-$ or, equivalently, $N \in M^+$)
the positive front starting from $N$ meets the negative front starting from $M$
\emph{if none of these two fronts have been previously annihilated}.
More precisely, they meet at $[N_x+M_x+\rho(N_s-M_s)]/2\in [0,L]$ at time 
$(N_x-M_x+\rho(N_s+M_s))/(2\rho)$,
which is greater than $M_s\lor N_s$.

\emph{Step 1.} Here we prove that $A(\nu)=B(\nu)$.

\begin{figure}[ht]
\noindent\fbox{\begin{minipage}{0.9\textwidth}
\centerline{\includegraphics[width=0.7\textwidth]{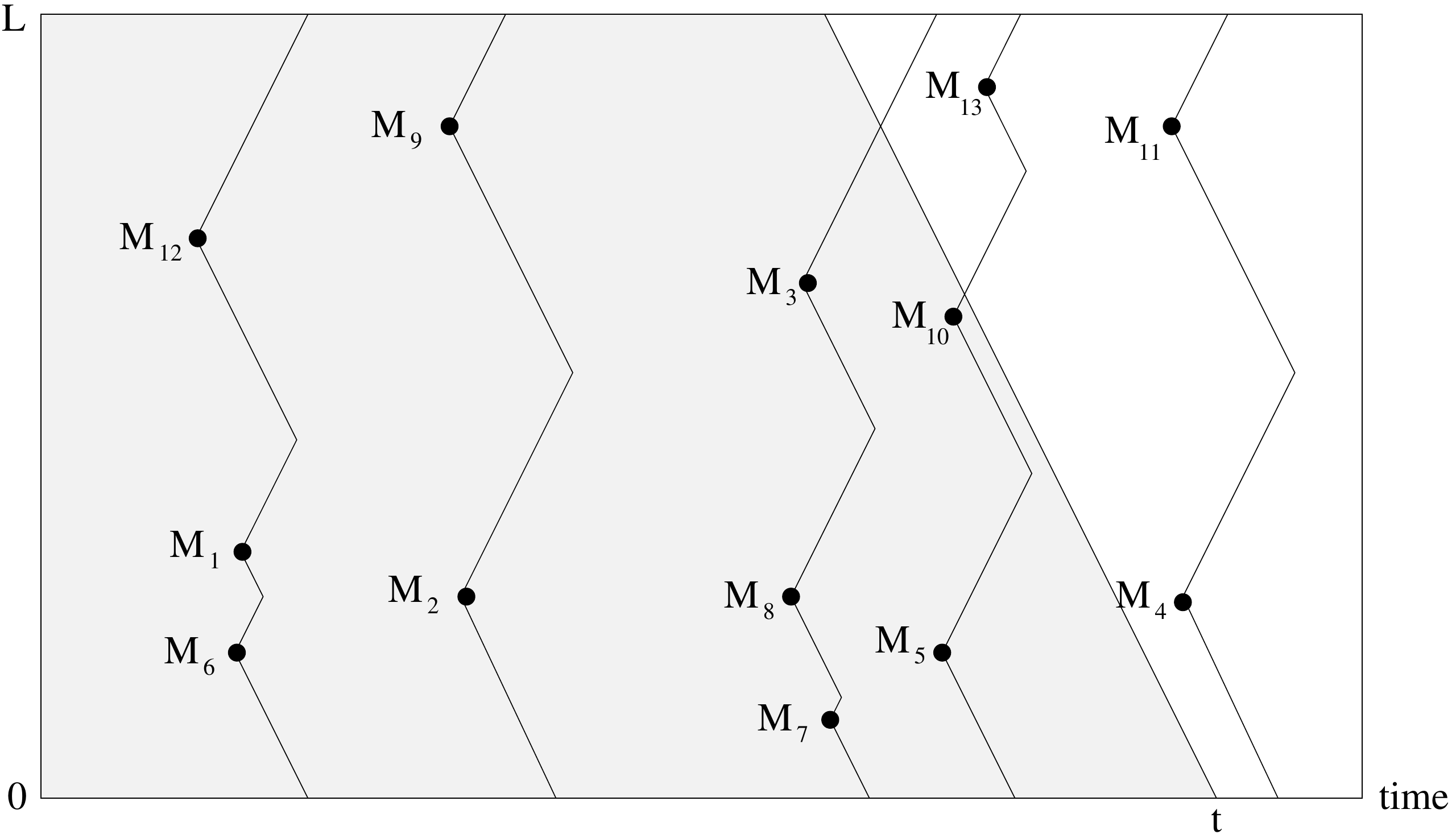}}
\caption{\label{fig1}
\small{All the (broken) lines have the same slope $\rho$ (or $-\rho$). The domain in gray is $D_t$.
The positive fronts are those going down, the negative fronts are those going up.
Here we have $G_1=\{M_1,M_6,M_{12}\}$, $G_2=\{M_2,M_9\}$, $G_3=\{M_3,M_7,M_8\}$, $G_4=\{M_5,M_{10},M_{13}\}$
and $G_5=\{M_4,M_{11}\}$.
And $P_1=M_6$, $P_2=M_1$ and $P_3=M_{12}$.}}
\end{minipage}}
\end{figure}

\emph{Step 1.1.}
We introduce $G_1$, the set of all minimal (for $\prec$) elements of $\cS_\nu$. See Figure~\ref{fig1}. This set is 
non empty because $\nu\neq 0$. It is also bounded (and thus finite since $\#(G_1)=\nu(G_1)$ and since
$\nu$ is Radon): 
fix $M \in G_1$ and observe that $G_1 \subset\{ M \} \cup  M^+\cup M^-\subset [0,M_s+L/\rho]\times [0,L]$.
We thus may write $G_1=\{P^1,\dots,P^k\}$, ordered in such a way that $P^1_x < P^2_x< \dots < P^k_x$.

We now show that all the fronts starting in $G_1$ annihilate,
except the positive one starting from $P^1$ (it reaches the soma at time $P^1_s+P^1_x/\rho$) and the
negative one starting from $P^k$ (it reaches the other extremity of the dendrite). See Figure~\ref{fig1}.

$\bullet$ We first verify by contradiction that the positive front starting from $P^1$ hits the soma.
If this is not the case, then, due to the above rules (a)-(b),
it has been annihilated by some front starting from some $Q \in \cS_\nu \cap P^{1-}$.
This is not possible, because $\cS_\nu \cap P^{1-}=\emptyset$. 

Indeed, assume that  $\cS_\nu \cap P^{1-}\neq\emptyset$
and consider a minimal (for $\prec$) element $Q$ of $\cS_\nu \cap P^{1-}$. Then
$Q$ is minimal in $\cS_\nu$ because else, we could find $M \in \cS_\nu\cap Q^\downarrow$,
whence $M \in \cS_\nu\cap Q^\downarrow \cap (P^{1-})^c$ (since $Q$ is minimal in $\cS_\nu\cap P^{1-}$)
whence $M \prec P^1$ (because $M \prec Q$, $Q \in P^{1-}$ and $M \notin P^{1-}$ implies that $M\prec P^1$), which 
is not possible because $P^1$ is minimal. So $Q$ is minimal in $\cS_\nu$, i.e. $Q\in G_1$,
and we furthermore have $Q_x<P^1_x$. This contradicts the definition of $P^1$.

$\bullet$ Similarly, one verifies that the negative front starting from $P^k$ hits the other extremity
($x=L$) of the dendrite.

$\bullet$ We finally fix $i\in \{1,\dots,k-1\}$ and show by contradiction that the negative front starting from 
$P^{i}$ does meet the positive front starting from $P^{i+1}$. 
Assume for example that the negative front starting at $P^i$ is annihilated before it 
meets the positive front starting from $P^{i+1}$. Then there is a point 
$Q \in \cS_\nu\cap P^{i+}\cap (P^{i+1 -} \cup P^{i+1\downarrow})$.
Indeed, $Q$ has to be in $P^{i+}$ so that the positive front
starting from $Q$ kills the negative front starting from $P^i$, and $Q$ has to be in $P^{i+1 -} \cup P^{i+1\downarrow}$
so that the killing occurs before the negative front starting from $P^i$ meets the 
positive front starting from $P^{i+1}$. 
But $\cS_\nu \cap P^{i+1\downarrow}=\emptyset$, since $P^{i+1}$ is minimal
in $\cS_\nu$. Hence $Q \in \cS_\nu \cap P^{i+}\cap P^{i+1-}$, so that $\cS_\nu \cap P^{i+}\cap P^{i+1-}$ is not empty.

We thus may consider a minimal (for $\prec$) element $R\in \cS_\nu \cap P^{i+}\cap P^{i+1-}$.
But then $R$ is minimal in $\cS_\nu$ because else, we could find $M \in \cS_\nu\cap R^\downarrow$
whence $M \in \cS_\nu\cap R^\downarrow \cap (P^{i+}\cap P^{i+1-})^c$ (since $R$ is minimal in 
$\cS_\nu\cap P^{i+}\cap P^{i+1-}$) whence $M \prec P^i$ or $M \prec P^{i+1}$ 
(because $M \prec R$, $R \in P^{i+}\cap P^{i+1-}$ and $M \notin P^{i+}\cap P^{i+1-}$ implies that $M\prec P^i$
or $M\prec P^{i+1}$), which is not possible because $P^i$ and $P^{i+1}$ are minimal.
At the end, we conclude that $R$ is minimal in $\cS_\nu$, i.e. 
$R \in G_1$, with furthermore $P^i_x<R_{x}<P^{i+1}_x$, which contradicts
the definition of $P^i$ and $P^{i+1}$.

\emph{Step 1.2.}
If $\cS_\nu \setminus G_1 = \emptyset$, we go directly to the concluding step. 
Otherwise, we introduce the (finite) set 
$G_2$ of all the minimal elements of $\cS\setminus G_1$. 
The fronts starting from a point in $G_2$ cannot be annihilated by those starting from a point in $G_1$
(because as seen in Step 1.1,  all the fronts in $G_1$ do annihilate together, except one that does hit the  
soma and one that does hit the other extremity: the fronts starting in $G_1$ do not interact with those
starting in $\cS_\nu \setminus G_1$). And
one can show, exactly as in Step 1.1, that
all the fronts starting in $G_2$ annihilate, except one positive front that hits the soma
and one negative front that hits the other extremity.

\emph{Step 1.3.}
If $\cS_\nu \setminus (G_1\cup G_2)=\emptyset$, we go directly to the concluding step. 
Otherwise, we introduce the (finite) set $G_3$ of all the minimal elements of $\cS\setminus (G_1\cup G_2)$. 
As previously, the fronts starting from a point in $G_3$ cannot be annihilated by those starting from 
a point in $G_1\cup G_2$.
And one can show, exactly as in Step 1.1, that
all the fronts starting in $G_3$ annihilate, except one positive front that hits the soma
and one negative front that hits the other extremity.

\emph{Step 1.4.} If $\cS_\nu \setminus (G_1\cup G_2\cup G_3)=\emptyset$, etc.

\emph{Concluding step.}
If the procedure stops after a finite number of steps, then there exists $n\in \N_*$ such that
$\cS_\nu=\cup_{k=1}^n G_k$, where $G_1$ is the set of all minimal elements of $\cS_\nu$ and, for all $k=2,\dots,n$,
$G_k$ is the set of all minimal elements of $\cS_\nu \setminus (\cup_{i=1}^{k-1}G_i)$.
We have seen that for each $k=1,\dots,n$, exactly one front starting from a point in $G_k$ hits the soma,
so that $B(\nu)=n$. And we also have $A(\nu)=n$. Indeed, choose $Q_n \in G_n$, there is necessarily
$Q_{n-1} \in G_{n-1}$ such that $Q_{n-1}\prec Q_n$, ..., and there is necessarily $Q_1 \in G_1$ such that
$Q_1\prec Q_2$. We end with an increasing sequence $Q_1\prec \dots \prec Q_n$ of points of $\cS_\nu$,
whence $A(\nu)\geq n$. We also have $A(\nu)\leq n$ because otherwise, we could find a 
sequence $R_1\prec \dots \prec R_{n+1}$ of points of $\cS_\nu$, and $\cS_\nu\setminus (\cup_{k=1}^n G_k)$
would contain at least $R_{n+1}$ and thus would not be empty.

If the procedure never stops, we have $A(\nu)=B(\nu)=\infty$ in which case $A(\nu)=B(\nu)$, as desired.

\emph{Step 2.} We now fix $t\geq 0$. By Step 1 applied to $\nu|_{D_t}$, we know that $B(\nu|_{D_t})=A(\nu|_{D_t})$, 
which equals $A_t(\nu)$ by definition. To conclude the proof, it thus only remains to check that 
$B_t(\nu)=B(\nu|_{D_t})$. This is clear when having a look at figure~\ref{fig1}: removing the points
$M_4,M_{11},M_{13}$ would not modify the number of fronts hitting the soma before $t$.
Here are the main arguments. We recall that for $M\in [0,\infty)\times [0,L]$, we have $M \in D_t$
if and only if $M \preceq (t,0)$ (i.e. $M_x \leq \rho (t-M_s)$).

$\bullet$ A (positive) front hitting the soma does it before time $t$ if and only if 
it starts from some point $M\in\cS_\nu \cap D_t$ (because such a front
hits the soma at time $M_s+M_x/\rho$, which is smaller than $t$ if and only if $M\preceq (t,0)$).

$\bullet$ A positive front starting from some  $M\in\cS_\nu \cap D_t$ always remains in $D_t$
(because $M_x\leq \rho(t-M_s)$ implies that $M_x-\rho(r-M_s) \leq \rho(t-r)$ for all $r \in [M_s,M_s+M_x/\rho]$).

$\bullet$ A front starting from some $M\in\cS_\nu \setminus D_t$ 
always remains outside $D_t$ (for e.g. the positive front starting from $M$, 
$M_x > \rho(t-M_s)$ implies that
$M_x-\rho(r-M_s) > \rho(t-r)$ for all $r \in [M_s,M_s+M_x/\rho]$).
\end{proof}

\section{Number of fronts in the piecewise i.i.d. case}\label{central}

The goal of this section is to check the following result, relying on \cite{ceh}.

\begin{lem}\label{lcru}
Let $H$ be a continuous probability density on $[0,L]$. Fix $0\leq b_0<b_1<\dots$ and consider, 
for each $k\geq 0$,
a probability density $g_k$ on $[b_k,\infty)$, continuous on $[b_k,b_{k+1}]$.
Consider an i.i.d. family $(X_i)_{i\geq 1}$ of $[0,L]$-valued random variables with density $H$ and, for each 
$k\geq 0$, an i.i.d. family $(T_k^i)_{i\geq 1}$ of $[b_k,\infty)$-valued random variables with density $g_k$.
We assume that for each $k\geq 0$, the family $(X_i)_{i\geq 1}$ is independent of the family $(T_k^i)_{i\geq 1}$
(but the families $(X_i,T_k^i)_{i\geq 1}$ and $(X_i,T_\ell^i)_{i\geq 1}$, with $k\neq \ell$, 
are allowed to be correlated in any possible way).
For each $N\geq 1$, we set 
$\nu_N= \sum_{i=1}^N \sum_{k\geq 0} \indiq_{\{T^i_k \leq b_{k+1}\}} \delta_{(T^i_k,X_i)}$. Then
$$
\lim_{N \to \infty} \frac{A_t(\nu_N)}{\sqrt N} = \Gamma_t(g) \quad \hbox{a.s.}
$$
for each $t\geq 0$, where $g(s)=\sum_{k\geq 0}g_k(s)\indiq_{\{s\in [b_k,b_{k+1}]\}}$.
\end{lem}

This result will be applied, more or less directly, to prove our two
main results, {\it via} Propositions~\ref{prel2} and~\ref{prel1}. 
In both cases, we will indeed be able to partition time in a family of intervals $[b_k,b_{k+1})$
during which the {\it stimuli} arrive in an i.i.d. manner on the dendrite under consideration, 
even if the whole family of those stimuli is not independent.
In the case of the soft model, this uses crucially the fact that Assumption (S1) induces a refractory period:
a neuron spiking at time \(t\) cannot spike again during $(t,t+\delta]$ for some deterministic $\delta>0$ (depending on
$t\geq 0$ and on many other parameters).

This section is the most technical of the paper. We have to be very careful, because as already mentioned,
$\Gamma_t(g)$ is rather sensitive. For example, modifying the density $H$ at one point does of course not
affect the empirical measure $\nu_N$, while it may drastically modify the value of $\Gamma_t(g)$
(recall that $\Gamma_t(g)$ depends on $H$, see Definition~\ref{dfG}).

In the whole section, the continuous density $H$ on $[0,L]$ is fixed.
We first adapt the result of \cite{ceh}.
\begin{lem}\label{base}
Fix $0\leq a<b$ and a continuous density
$h$ on $[a,b]$. 
Consider an i.i.d. family $(Z_i)_{i\geq 1}$ of  $[a,b]\times[0,L]$-valued random variables with density 
$h(s)H(x)$. For $N\geq 1$, define $\pi_N=\sum_1^N\delta_{Z_i}$.
For any bounded open domain $B \subset [0,\infty)\times[0,L]$ with Lipschitz boundary,
$$
\lim_{N\to\infty} \frac{A(\pi_N|_B)}{\sqrt N} = \Lambda_B(h) \quad \hbox{a.s.},
$$
where $\Lambda_B(h)=\sup_{\beta \in \cB} \cJ_B(h,\beta)$, $\cB$ being the set of $C^1$-functions
defined on a closed bounded interval $I_\beta\subset \R$ into $\R$ and 
satisfying $\sup_{s\in I_\beta}|\beta'(s)|< \rho$ and, for $\beta \in \cB$,
$$
\cJ_B(h,\beta)=\sqrt{\frac 2\rho}\int_{I_\beta} \sqrt{h(s)H(\beta(s))\indiq_{\{s\in(a,b),\beta(s)\in(0,L)\}}
\indiq_{\{(s,\beta(s))\in B\}}
[\rho^2-(\beta'(s))^2]} \dd s.
$$
Of course, we set $h(s)H(x)\indiq_{\{s\in(a,b),x\in(0,L)\}} = 0$ if $(s,x)\notin (a,b)\times(0,L)$, even if
$h(s)H(x)$ is not defined.
\end{lem}

\begin{proof}
We first recall a \(2d\) version of~\cite[Theorem 1.2]{ceh}, which concerns the length of the longest
increasing subsequence 
(for the usual partial order $\pre$ of $\R^2$) one can find in a cloud of $N$ i.i.d. points with 
positive continuous density on a regular domain $O\subset \R^2$.
In a second step, we easily deduce the behavior of the length of the longest increasing subsequence (for the same 
random variables and the same order) included in a subset $G$ of $O$.
It only remains to use a diffeomorphism that maps the usual order $\pre$ on $\R^2$ onto our order
$\prec$: we study how the density of the random variables is modified in Step 3, and how this modifies
the limit functional in Step 4.

For $y=(y_1,y_2)$ and $y'=(y'_1,y'_2)$ in 
$\R^2$, we say that $y\prel y'$ if $y_1\leq y_1'$ and $y_2\leq y_2'$. We say that $y \pre y'$ 
if $y\prel y'$ and $y\neq y'$.

\emph{Step 1.} Consider a bounded open subset $O \subset \R^2$ with 
Lipschitz boundary, as well as a probability density $\phi$ on $\R^2$, vanishing outside $O$ and
uniformly continuous on $O$. Consider an i.i.d. family $(Y_i)_{i\geq 1}$ of random variables with 
density $\phi$.
For $N\geq 1$, denote by 
\[
L_N=\sup\{k\geq 1 : \exists \;i_1,\dots,i_k \in \{1,\dots,N\} \hbox{ such that }
Y_{i_1}\pre\dots\pre Y_{i_k}\}. 
\]
Then $\lim_N N^{-1/2}L_N=\sup_{\gamma \in \cA} \cK (\gamma)$ a.s., where 
$\cK(\gamma)=2 \int_0^1 \sqrt{\phi(\gamma(r))\gamma_1'(r)\gamma_2'(r)}\dd r$ and $\cA$ consists of all
$C^1$-maps $\gamma=(\gamma_1,\gamma_2)$ from $[0,1]$ into $\R^2$ such that 
$\gamma_1'(r)\geq 0$ and $\gamma_2'(r)\geq 0$ for all $r\in [0,1]$ (see \cite{ceh}).

\emph{Step 2.} Consider some bounded open $G\subset \R^2$ with Lipschitz boundary. 
Adopt the same notation and conditions
as in Step 1. For each $N\geq1$, set
$$
L_N(G)=\sup\{k\geq 0 : \exists \;i_1,\dots,i_k \in \{1,\dots,N\} \hbox{ such that }
Y_{i_1}\pre\dots\pre Y_{i_k} \hbox{ and } Y_{i_j} \in G \hbox{ for all } j\}. 
$$
Then
$\lim_N N^{-1/2}L_N(G)=\sup_{\gamma \in \cA} \cK_G(\gamma)$ a.s., with
$\cK_G(\gamma)=2 \int_0^1 \sqrt{\phi(\gamma(r))\indiq_{\{\gamma(r)\in G\}}\gamma_1'(r)\gamma_2'(r)}\dd r$.

Indeed, if $c_G=\int_G \phi(y)\dd y=0$, both quantities equal $0$ (because $\phi\equiv 0$ on $G\cap O$ by 
continuity, and $\phi\indiq_G=0$ on $O^c$ by definition). 
Else, $\phi_G=c_G^{-1} \phi\indiq_G$
satisfies the assumptions of Step 1. For each $N\geq 1$, we set $S_N=\{ i \in \{1,\dots,N\} : 
Y_i \in G\}$. Since the law of the sub-sample $(Y_i)_{i \in S_N}$ knowing $|S_N|$ is that of a
family of $|S_N|$ i.i.d. random variables with density $\phi_G$,
we have $\lim_{N} |S_N|^{-1/2} L_N(G) =
\sup_{\gamma \in \cA} 2 \int_0^1 \sqrt{\phi_G(\gamma(r))\gamma_1'(r)\gamma_2'(r)}\dd r$ a.s.
But $\lim_N N^{-1}|S_N|= c_G$ a.s., whence the conclusion.

\emph{Step 3.} We now introduce the $C^\infty$-diffeomorphism  $\psi(s,x)=
(\rho s- x,\rho s+ x)$ from $\R^2$ into itself.
For all $i\geq 1$, we set $Y_i=\psi(Z_i)$. The density $\phi$ of $Y_1$ is given by 
$\phi(y)=R(\psi^{-1}(y))/(2\rho)$, where we have set 
$R(s,x)=h(s)H(x)\indiq_{\{s\in (a,b),x\in (0,L)\}}$ for all
$(s,x)\in\R^2$. This density $\phi$ satisfies the conditions of Step 1, by continuity of $h$ on $[a,b]$
and of $H$ on $[0,L]$, with $O=\psi((a,b)\times(0,L))$.

We next observe that for any $(s,x),(s',x') \in \R^2$, we have $(s,x)\prec (s',x')$ if and only if 
$\psi(s,x)\pre \psi(s',x')$. 
Hence, by Definition~\ref{A}, we have $A(\pi_N|_B)=L_N(\psi(B))$ (with the notation of Step 2 and the choice 
$Y_i=\psi(Z_i)$).
Clearly, $\psi(B)$ is a bounded open domain of $\R^2$. By Step 2, we thus have
$\lim_N  N^{-1/2}A(\pi_N|_B) =\sup_{\gamma \in \cA} \cK_{\psi(B)}(\gamma)$ a.s.

\emph{Step 4.} It remains to verify that $\sup_{\gamma \in \cA} \cK_{\psi(B)}(\gamma)=\Lambda_B(h)$.
Recall that for $\gamma \in \cA$, 
$$
\cK_{\psi(B)}(\gamma)=\sqrt{\frac2\rho}\int_0^1 \sqrt{R(\psi^{-1}(\gamma(r)))
\indiq_{\{\gamma(r) \in \psi(B)\}}\gamma_1'(r)\gamma_2'(r)} \dd r.
$$

One easily checks that $\gamma \in \cA$ if and only if $\alpha=\psi^{-1}\circ \gamma \in \cC$
and that $\cK_{\psi(B)}(\gamma)=\cL_{B}(\alpha)$, where
$\cC$ is the set of all $C^1$-maps $\alpha: [0,1]\mapsto\R^2$ such that $|\alpha_2'(r)| \leq \rho \alpha_1'(r)$
for all $r\in[0,1]$ and
$$
\cL_B(\alpha)=\sqrt{\frac2\rho}
\int_0^1\sqrt{R(\alpha(r))\indiq_{\{\alpha(r) \in B\}} [\rho^2(\alpha_1'(r))^2- (\alpha_2'(r))^2]}\dd r.
$$

But $\sup_{\alpha \in \mathring{\cC}} \cL_{B}(\alpha)=\sup_{\alpha \in \cC} \cL_{B}(\alpha)$,
where $\mathring{\cC}$ consists of the elements of $\cC$ such that  $|\alpha_2'(r)| < \rho \alpha_1'(r)$
on $[0,1]$. Indeed, it suffices to approximate $\alpha \in \cC$ 
by $\alpha^n(r)=(\alpha_1(r)+r/n ,\alpha_2(r))$,
that belongs to $\mathring{\cC}$, and to observe that $\cL_B(\alpha) \leq \liminf_n \cL_B(\alpha_n)$ 
by the Fatou Lemma and since
$R(\alpha(r))\indiq_{\{\alpha(r) \in B\}} \leq \liminf_n R(\alpha_n(r))\indiq_{\{\alpha_n(r) \in B\}}$ 
for each $r\in[0,1]$,
because $R\indiq_B=R\indiq_{O\cap B}$ with $R$ continuous on the open set $O\cap B$.

Finally, one easily verifies that for $\alpha \in \mathring{\cC}$, the map $\beta=\alpha_2 
\circ \alpha_1^{-1}$ 
(defined on the interval $I_\beta=[\alpha_1(0),\alpha_1(1)]$) belongs to $\cB$, with
furthermore $\cL_B(\alpha)=\cJ_B(h,\beta)$. And for $\beta \in \cB$ (defined on $I_\beta=[a,b]$), the map
$\alpha=(\alpha_1,\alpha_2)$ defined on $[0,1]$ by $\alpha_1(r)=a+r(b-a)$ and $\alpha_2(r)=\beta(a+r(b-a))$ belongs
to $\mathring{\cC}$ and we have $\cL_B(\alpha)=\cJ_B(h,\beta)$.

All in all,
$\sup_{\gamma \in \cA} \cK_{\psi(B)}(\gamma)=\sup_{\alpha \in \cC} \cL_{B}(\alpha)
=\sup_{\alpha \in \mathring{\cC}} \cL_{B}(\alpha)=\sup_{\beta \in \cB} \cJ_{B}(h,\beta)$.
\end{proof}

We can now give the

\begin{proof}[Proof of Lemma~\ref{lcru}.]
Let us explain the main ideas of the proof.
The main tool consists in applying Lemma~\ref{base} in any
reasonable  subset of \([b_k,b_{k+1}]\times[0,L]\), for any $k\geq 0$, which we do in Step 1 for a sufficiently large
family of such subsets. In Step 4, we prove that 
$\Lambda_{\mathring{D}_t}(g)=\lim_{\delta \downarrow 0} \Lambda_{\mathring{D}_{t+\delta}}(g)=\Gamma_t(g)$, which is
very natural but tedious. The lowerbound 
$\liminf_N N^{-1/2}A_t(\nu_N) \geq \Gamma_t(g)$ is proved in Step 2: we consider some
$\beta \in \cB$ such that $\cJ_{\mathring{D}_t}(g,\beta) \geq \Lambda_{\mathring{D}_t}(g)-\e$,
we introduce a tube $B_{\beta,\delta}$ around the path $\{(s,\beta(s)) : s\in [0,t]\}$
and observe that $\cJ_{\mathring{D}_t}(g,\beta) = \cJ_{\mathring{D}_t\cap B_{\beta,\delta}}(g,\beta)$.
Using Step 1, we deduce that in each $B_{\beta,\delta} \cap ([b_k,b_{k+1}]\times [0,L])$, we can find
an increasing subsequence of points with the correct length, that is, more or less,
$N^{1/2} \cJ_{\mathring{D}_t\cap B_{\beta,\delta}\cap ([b_k,b_{k+1}]\times [0,L])}(g,\beta)$. We then concatenate these subsequences
(with a small loss to be sure the concatenation is fully increasing) and find that, very roughly, 
$A_t(\nu_N) \geq N^{1/2} \sum_{k\geq 1} \cJ_{\mathring{D}_t\cap B_{\beta,\delta}\cap ([b_k,b_{k+1}]\times [0,L])}(g,\beta)
\geq \cJ_{\mathring{D}_t}(g,\beta)$
as desired.
The upperbound is more complicated be uses similar ideas: if one could find an increasing subsequence 
with length significantly greater than $N^{1/2} \cJ_{\mathring{D}_t\cap B_{\beta,\delta}\cap ([b_k,b_{k+1}]\times [0,L])}(g,\beta)$,
this would mean that somewhere, in some $[b_k,b_{k+1}]\times[0,L]$, there would be an increasing subsequence 
with length significantly greater than established in Lemma~\ref{base}.

{\it Notation.} Changing the value of $g_k$ on $(b_{k+1},\infty)$ does clearly not modify the definitions of $g$ and of $\nu_N$,
since $T^i_k$ is not taken into account if greater than $b_{k+1}$.
Hence we may (and will) assume that for each $k\geq 0$, $g_k$ is a density, continuous on $[b_k,b_{k+1}+1]$
and vanishing outside $[b_k,b_{k+1}+1]$.

We fix $t>0$ and call $k_0$ the integer such that $t \in [b_{k_0},b_{k_0+1})$. We assume that $k_0\geq 1$, the
situation being much easier when $k_0=0$.

For $\beta \in \cB$ and $\delta>0$, we define $B_{\beta,\delta}=\{(s,x) : s\in I_\beta,
x \in (\beta(s)-\delta,\beta(s)+\delta)\}$.
For $k=0,\dots,k_0$ and $a\geq 0$, we also introduce 
$B_{\beta,a,\delta}^k= B_{\beta,\delta} \cap ((b_k,b_{k+1}-a\delta)\times \R)$,
with the convention that $(x,y)=\emptyset$ if $x\geq y$.
All these sets are open, bounded and have a Lipschitz boundary
because $\beta$ is of class $C^1$.

\emph{Step 1.} For each $k=0,\dots,k_0$, we may apply Lemma~\ref{base}, to the family
$(T^i_k,X_i)_{i\geq 1}$. Introducing $\pi_N^k=\sum_{i=1}^N \delta_{(T^i_k,X_i)}$, we have, 
for any $\beta\in \cB$, any $\delta\in(0,1)$, any $a\geq 0$, a.s.
$$
\lim_N N^{-1/2} A(\pi_N^k|_{B_{\beta,a,\delta}^k\cap \mathring{D}_t})=\Lambda_{B_{\beta,a,\delta}^k\cap \mathring{D}_t}(g_k).
$$
Observing now that $\nu_N|_{B_{\beta,a,\delta}^k} = \pi_N^k|_{B_{\beta,a,\delta}^k}$ 
and $g|_{B_{\beta,a,\delta}^k}=g_k|_{B_{\beta,a,\delta}^k}$, we also have a.s.
\begin{equation*}
\lim_N N^{-1/2} A(\nu_N|_{B_{\beta,a,\delta}^k\cap \mathring{D}_t})=\Lambda_{B_{\beta,a,\delta}^k\cap \mathring{D}_t}(g).
\end{equation*}

\emph{Step 2. Lowerbound.} Here we prove that a.s., $\liminf_N N^{-1/2}A_t(\nu_N) \geq \Lambda_{\mathring{D}_t}(g)$.
For $\e\in (0,1)$,
we can find $\beta \in \cB$ such that $\cJ_{\mathring{D}_t}(g,\beta) \geq \Lambda_{\mathring{D}_t}(g)-\e$.
Let $\eta\in (0,1)$ be such that $\sup_{I_\beta}|\beta'|\leq (1-\eta)\rho$ and let 
$a=2/(\rho\eta)$.

We first claim that
\begin{equation}\label{oobb}
\hbox{$0\leq k < \ell \leq k_0$ and $(s,x)\in B_{\beta,a,\delta}^k$ and $(s',x')\in B_{\beta,a,\delta}^\ell$
imply that $(s,x)\prec (s',x')$.}
\end{equation}
It suffices to check that for any $(s,x),(s',x')\in B_{\beta,\delta}$ with $s'\geq s+a\delta$, 
we have $(s,x)\prec (s',x')$. This follows from the facts that $|x-\beta(s)|< \delta$,  
$|x'-\beta(s')|< \delta$ and $|\beta(s)-\beta(s')|\leq (1-\eta)\rho(s'-s)$,
whence $|x-x'|<2\delta+(1-\eta)\rho(s'-s) \leq \rho (s'-s)$, because $2\delta \leq 2(s'-s)/a =\rho \eta (s'-s)$.

Hence a.s.,  $A_t(\nu_N)=A(\nu_N|_{D_t})\geq \sum_{k=0}^{k_0} A(\nu_N|_{B_{\beta,a,\delta}^k \cap \mathring{D}_t})$. 
Indeed, it suffices to recall Definition~\ref{A}, to call $S_N$ the set of points in the support of
$\nu_N$ intersected with $\mathring{D}_t$, 
and to observe that thanks to \eqref{oobb}, the concatenation of the longest
increasing (for $\prec$) subsequence of $S_N\cap B_{\beta,a,\delta}^0$ with the longest increasing subsequence of 
$S_N\cap B_{\beta,a,\delta}^1$ ... with the longest increasing subsequence of $S_N\cap B_{\beta,a,\delta}^{k_0}$
indeed produces an increasing subsequence of $S_N$.

Due to Step 1, we conclude that a.s., for all $\delta>0$,
$$
\liminf_N N^{-1/2} A_t(\nu_N)\geq \sum_{k=0}^{k_0} \Lambda_{B_{\beta,a,\delta}^k\cap \mathring{D}_t}(g) \geq 
\sum_{k=0}^{k_0}\cJ_{B_{\beta,a,\delta}^k\cap \mathring{D}_t}(g,\beta).
$$
But for all $k=0,\dots,k_0$, we have 
$(s,\beta(s)) \in B_{\beta,a,\delta}^k$ for all $s\in I_\beta \cap [b_k,(b_{k+1}-a\delta)]$, whence
$$
\liminf_N N^{-1/2} A_t(\nu_N)\geq 
\sum_{k=0}^{k_0}\sqrt{\frac2\rho}\int_{I_\beta \cap(b_k,b_{k+1}-a\delta)}  
\sqrt{g(s)H(\beta(s))\indiq_{\{(s,\beta(s))\in \mathring{D}_t\}}
[\rho^2-(\beta'(s))^2]} \dd s.
$$
Letting $\delta$ decrease to $0$, we find that a.s.,
$$
\liminf_N N^{-1/2} A_t(\nu_N)\geq \sqrt{\frac2\rho}\int_{I_\beta \cap [b_0,b_{k_0+1}]} \sqrt{g(s)H(\beta(s))
\indiq_{\{(s,\beta(s))\in \mathring{D}_t\}}[\rho^2-(\beta'(s))^2]} \dd s =\cJ_{\mathring{D}_t}(g,\beta),
$$
the last inequality following from the facts that $g(s)\indiq_{\{(s,\beta(s)) \in \mathring{D}_t\}}$
vanishes if $s\in [b_0,b_{k_0+1}]^c \subset [b_0,t]^c$.
Recalling the beginning of the step,
$\liminf_N N^{-1/2} A_t(\nu_N)\geq \Lambda_{\mathring{D}_t}(g)-\e$ as desired.

\emph{Step 3. Upperbound.} We next check that a.s., $\limsup_N N^{-1/2}A_t(\nu_N)\leq 
\lim_{\delta \downarrow 0} \Lambda_{\mathring{D}_{t+\delta}}(g)$.
We introduce $\cB_{[0,t]}=\{ \beta \in \cB : I_\beta =[0,t]$ and $\beta([0,t])\subset (0,L)\}$.

\emph{Step 3.1.} Here we prove that for any $\delta\in(0,1)$, there is a finite subset 
 $\cB^\delta_t \subset \cB_{[0,t]}$ such that $A(\nu|_{\mathring{D}_t})
\leq \max_{\beta \in \cB_t^\delta} A(\nu|_{B_{\beta,\delta}\cap \mathring{D}_t})$
for all Radon point measures $\nu$ on $[0,\infty)\times[0,L]$.

For all Radon point measures $\nu$ on $[0,\infty)\times[0,L]$, we have 
$A(\nu|_{\mathring{D}_t})\leq \sup_{\beta \in \cB_{[0,t]}} A(\nu|_{B_{\beta,\delta} \cap {\mathring{D}_t}})$.
Indeed, consider an
increasing subsequence $(t_1,x_1) \prec \dots \prec (t_\ell,x_\ell)$ of points in the support of $\nu$ 
intersected with
$\mathring{D}_t$ such that $\ell=A_t(\nu)$. Consider $\beta_0: [0,t]\mapsto (0,L)$
of which the graph is the broken line linking $(0,x_1)$, $(t_1,x_1)$, $(t_2,x_2)$, ..., $(t_\ell,x_\ell)$ and
$(t,x_\ell)$.
Then $\beta_0$ is $\rho$-Lipschitz continuous (because the points are ordered for 
$\prec$). Hence it is not hard to find $\beta_\delta \in \cB_{[0,t]}$ such that
$\sup_{[0,t]}|\beta_\delta(s)-\beta_0(s)|<\delta$.
And  $\{(t_1,x_1), \dots, (t_\ell,x_\ell)\}\subset B_{\beta_\delta,\delta}$, whence
$A(\nu|_{B_{\beta_\delta,\delta}\cap\mathring{D}_t}) \geq \ell=A(\nu|_{\mathring{D}_t})$.

Next, $\cB_{[0,t]}$ is dense, for the uniform convergence topology, in $\bar \cB_{[0,t]}$,
the set of $\rho$-Lipschitz continuous functions from $[0,t]$ into $[0,L]$.
We thus may write  $\bar \cB_{[0,t]}=\cup_{\beta \in \cB_{[0,t]}} \cV(\beta,\delta)$,
where $\cV(\beta,\delta)=\{\alpha \in \bar \cB_{[0,t]} : \sup_{[0,t]} |\alpha(s)-\beta(s)| < \delta/2\}$.
But $\bar\cB_{[0,t]}$ is compact (still for the uniform convergence topology), so that there is 
a finite subset $\cB^\delta_t \subset \cB_{[0,t]}$ such that 
$\cB_{[0,t]}\subset \bar \cB_{[0,t]}=\cup_{\beta \in \cB^\delta_t}
\cV(\beta,\delta)$. The conclusion follows, using the previous paragraph, since then for any $\nu$,
we have $A(\nu|_{\mathring{D}_t})\leq \sup_{\alpha \in \cB_{[0,t]}} A(\nu|_{B_{\alpha,\delta/2}\cap \mathring{D}_t}) 
\leq \sup_{\beta \in \cB^\delta_t} A(\nu|_{B_{\beta,\delta} \cap \mathring{D}_t})$
because for each $\alpha \in \cB_{[0,t]}$, we can find $\beta \in \cB^\delta_t$ such that
$\sup_{[0,t]} |\alpha(s)-\beta(s)| < \delta/2$, which implies that $B_{\alpha,\delta/2}\subset B_{\beta,\delta}$.

\emph{Step 3.2.} For all $\beta \in \cB_{[0,t]}$, $\delta\in (0,1)$,
$\limsup_N N^{-1/2}A(\nu_N|_{B_{\beta,\delta}\cap\mathring{D}_t}) 
\leq \sum_{k=0}^{k_0} \Lambda_{B_{\beta,0,\delta}^k\cap\mathring{D}_t}(g)$ a.s.

Indeed, recalling that $B_{\beta,\delta}=\cup_{k=0}^{k_0}B_{\beta,0,\delta}^k$ 
(up to a Lebesgue-null set in which our random variables a.s. never fall),
we a.s. have $A(\nu_N|_{B_{\beta,\delta}\cap\mathring{D}_t}) 
\leq \sum_{k=0}^{k_0}A(\nu_N|_{B_{\beta,0,\delta}^k\cap\mathring{D}_t})$,
because the longest increasing sequence of points in the support of $\nu_N$ intercepted 
with $B_{\beta,\delta}\cap\mathring{D}_t$
is less long than the concatenation (for $k=0,\dots,k_0$) of the longest increasing sequences
of points in the support of $\nu_N$ intercepted with $B_{\beta,0,\delta}^k\cap\mathring{D}_t$.
The conclusion follows from Step 1.

\emph{Step 3.3.} Gathering Steps 3.1. and 3.2, we deduce that for all $\delta \in (0,1)$, 
a.s., 
\begin{align*}
\limsup_N N^{-1/2}A_t(\nu_N)=&\limsup_N N^{-1/2}A(\nu_N|_{\mathring{D}_t})\\
\leq& \limsup_N N^{-1/2} \sup_{\beta \in \cB^\delta_t} A(\nu_N|_{B_{\beta,\delta}\cap \mathring{D}_t})\\
=&\sup_{\beta \in \cB^\delta_t} \limsup_N N^{-1/2}A(\nu_N|_{B_{\beta,\delta}\cap \mathring{D}_t})\\
\leq &\sup_{\beta \in \cB^\delta_t} \sum_{k=0}^{k_0} \Lambda_{B_{\beta,0,\delta}^k\cap\mathring{D}_t}(g).
\end{align*}
The first equality is obvious, because all our random variables have densities and thus a.s. never fall in
$D_t\setminus\mathring{D}_t$.
The second equality uses that the set $\cB^\delta_t$ is finite.

\emph{Step 3.4.} We prove here the existence of a function $\varphi:(0,1)\mapsto \R_+$ and $c>0$ such that
$\lim_{\delta\to 0}\varphi(\delta)=0$ and, for all $\beta \in \cB_{[0,t]}$, all $\delta\in(0,1)$,
$$
\sum_{k=0}^{k_0} \Lambda_{B_{\beta,0,\delta}^k\cap\mathring{D}_t}(g) \leq \Lambda_{\mathring{D}_{t+c \delta}}(g)+\varphi(\delta).
$$

Let $\delta \in(0,1)$ and $\beta \in \cB_{[0,t]}$. 
For each $k=0,\dots,k_0$, let $\alpha_k \in \cB$ be such that 
$\cJ_{B_{\beta,0,\delta}^k\cap \mathring{D}_t}(g,\alpha_k)\geq  
\Lambda_{B_{\beta,0,\delta}^k\cap\mathring{D}_t}(g)-\delta$. 
It is tedious but not difficult to check that we can choose 
$\alpha_k$ defined on $I_{\alpha_k}=[b_k,b_{k+1}\land t]$ and such that 
$(s,\alpha_k(s)) \in B_{\beta,0,\delta}^k\cap \mathring{D}_t$ for all $s\in I_{\alpha_k}$.
In particular, $\alpha_{k-1}(b_k) \in (\beta(b_k)-\delta,\beta(b_k)+\delta)$ and
$\alpha_{k}(b_{k}) \in (\beta(b_{k})-\delta,\beta(b_{k})+\delta)$ for each $k=1,\dots,k_0$.

We then define $\alpha$ on $[0,t]$ as the following continuous concatenation
of the functions $\alpha_k$: we put
$\alpha(s)=\alpha_0(s)$ on $[b_0,b_1)$,
$\alpha(s)=\alpha_1(s)+\alpha_0(b_1)-\alpha_1(b_1)$ on $[b_1,b_2)$,
etc, and $\alpha(s)=\alpha_{k_0}(s)+ \sum_{\ell=1}^{k_0} [\alpha_{\ell-1}(b_{\ell})-\alpha_{\ell}(b_\ell)]$ 
on $[b_{k_0},t]$.
The resulting $\alpha$ is $\rho$-Lipschitz continuous on $[0,t]$, satisfies
$\alpha'(s)=\alpha_k'(s)$ for all $k=0,\dots,k_0$ and all $s \in (b_k,b_{k+1}\land t)$ and
$$
\sup_{k=0,\dots,k_0}\sup_{[b_k,b_{k+1}\land t)} |\alpha(s)-\alpha_k(s)|\leq \sum_{\ell=1}^{k_0}
|\alpha_{\ell-1}(b_{\ell})-\alpha_{\ell}(b_\ell)|\leq 2k_0 \delta,
$$
since for each $\ell=1,\dots,k_0$, both $\alpha_{\ell-1}(b_{\ell})$ and 
$\alpha_{\ell}(b_\ell)$ belong to $(\beta(b_\ell)-\delta,\beta(b_\ell)+\delta)$.

Finally, we set $\gamma(s)= [\alpha(s) \land (L-\delta) ] \lor \delta$ for $s\in [0,t]$.
It is $\rho$-Lipschitz continuous, satisfies that 
$|\gamma'(s)|\leq |\alpha'(s)|=|\alpha_k'(s)|$ for all 
$k=0,\dots,k_0$ and almost all $s \in (b_k,b_{k+1}\land t)$ and
\begin{equation}\label{ddii}
\sup_{k=0,\dots,k_0}\sup_{[b_k,b_{k+1}\land t)} |\gamma(s)-\alpha_k(s)|\leq 2k_0 \delta.
\end{equation}
Indeed, it suffices to note that $x \in (0,L)$ and $|y-x|< 2k_0\delta$ imply that
$|[y\land (L-\delta) ] \lor \delta - x|<2k_0\delta$ (apply this principle to $x=\alpha_k(s)$ and
$y=\alpha(s)$).

We have $(s,\gamma(s)) \in \mathring{D}_{t+c\delta}$ for all $s\in [0,t]$, with $c=2k_0/\rho$, because
$\gamma(s) \in (0,L)$ and because for $s\in [b_k,b_{k+1}\land t)$, $\gamma(s) \leq \alpha_k(s)+2k_0\delta <
\rho(t-s)+2k_0\delta= \rho(t+c\delta-s)$. We used that
$(s,\alpha_k(s)) \in \mathring{D}_t$, whence $\alpha_k(s) <\rho(t-s)$.

We thus may write, using that $(s,\alpha_k(s))\in B_{\beta,0,\delta}^k\cap\mathring{D}_t$ for all $k=0,\dots k_0$
and all $s\in [b_k,b_{k+1}\land t)$,
\begin{align*}
\sum_{k=0}^{k_0} \Lambda_{B_{\beta,0,\delta}^k\cap\mathring{D}_t}(g) \leq &
\sum_{k=0}^{k_0} (\cJ_{B_{\beta,0,\delta}^k\cap \mathring{D}_t}(g,\alpha_k)+\delta)\\
=& (k_0+1)\delta+ \sqrt{\frac2\rho}\sum_{k=0}^{k_0} \int_{b_k}^{b_{k+1}\land t} 
\sqrt{H(\alpha_k(s))g(s)[\rho^2-(\alpha_k'(s))^2]} \dd s
\\
\leq & (k_0+1)\delta + \sqrt{\frac2\rho}\sum_{k=0}^{k_0} \int_{b_k}^{b_{k+1}\land t} 
\sqrt{H(\alpha_k(s))g(s)[\rho^2-(\gamma'(s))^2]} \dd s,
\end{align*}
because $|\gamma'(s)|\leq |\alpha_k'(s)|$ for a.e. $s\in [b_k,b_{k+1}\land t)$.
We then set 
$$
\varphi(\delta)=(k_0+1)\delta + \Big(\sqrt{2 \rho} \intot \sqrt{g(s)}\dd s  \Big)
\sup_{x,y \in (0,L),|x-y|\leq 2 k_0 \delta} 
|\sqrt{H(x)} - \sqrt{H(y)}|,
$$
which tends to $0$ as $\delta\to 0$ because $H$ is continuous on $[0,L]$. Recalling \eqref{ddii},
\begin{align*}
\sum_{k=0}^{k_0} \Lambda_{B_{\beta,0,\delta}^k\cap\mathring{D}_t}(g) \leq & \varphi(\delta) 
+ \sqrt{\frac2\rho}\intot
\sqrt{H(\gamma(s))g(s)[\rho^2-(\gamma'(s))^2]} \dd s\\
=& \varphi(\delta) 
+ \sqrt{\frac2\rho}\intot
\sqrt{H(\gamma(s))g(s)\indiq_{\{(s,\gamma(s)) \in \mathring{D}_{t+c\delta}\}}[\rho^2-(\gamma'(s))^2]} \dd s\\
=& \varphi(\delta) + \cJ_{\mathring{D}_{t+c\delta}}(g,\gamma)\\
\leq & \varphi(\delta) + \Lambda_{\mathring{D}_{t+c\delta}}(g).
\end{align*}
There is a little work to prove the last inequality because $\gamma \notin \cB$.
Since $\gamma$ is $\rho$-Lipschitz continuous, it is easily approximated by a family $\gamma_\ell$
of elements of $\cB$ (with $I_{\gamma_\ell}=[0,t]$) in such a way that $\gamma_\ell$ tends to $\gamma$
uniformly and $\gamma_\ell'$ tends to $\gamma'$ a.e. Using that $H$ is continuous, that
$\mathring{D}_{t+c\delta}$ is open and the Fatou lemma, we conclude that
$\cJ_{\mathring{D}_{t+c\delta}}(g,\gamma)\leq \liminf_\ell \cJ_{\mathring{D}_{t+c\delta}}(g,\gamma_\ell) \leq 
\Lambda_{\mathring{D}_{t+c\delta}}(g)$.

\emph{Step 3.5.} Gathering Steps 3.3 and 3.4, we find that for all $\delta \in (0,1)$,
$\limsup_N N^{-1/2}A_t(\nu_N) \leq  \Lambda_{\mathring{D}_{t+c\delta}(g)} 
+\varphi(\delta)$ a.s. Letting $\delta$ decrease to $0$
completes the step.

\emph{Step 4.} Finally, we verify that $\Lambda_{\mathring{D}_t}(g)=\lim_{\delta \downarrow 0} 
\Lambda_{\mathring{D}_{t+\delta}}(g)
=\Gamma_t(g)$ and this will complete the proof. Recall that $\Gamma_t(g)=\sup_{\beta \in \cB_t} \cI_t(g,\beta)$
was introduced in Definition~\ref{dfG}, while $\Lambda_B(g)=\sup_{\beta \in \cB} \cJ_B(g,\beta)$ was defined
in Lemma~\ref{base}. We will verify the four following inequalities, which is sufficient:
$\Gamma_t(g) \leq \Lambda_{\mathring{D}_t}(g)$, $\Lambda_{\mathring{D}_{t}}(g) \leq \Gamma_t(g)$,
$\Lambda_{\mathring{D}_{t}}(g) \leq \lim_{\delta \downarrow 0} \Lambda_{\mathring{D}_{t+\delta}}(g)$
and $\lim_{\delta \downarrow 0} \Gamma_{t+\delta}(g)\leq \Gamma_t(g)$.

We first check that $\Gamma_t(g) \leq \Lambda_{\mathring{D}_t}(g)$. We thus fix $\beta \in \cB_t$.
For $\ell \geq 1$, we introduce
$$
\beta_\ell(s)=\frac{\beta(s)+(t-s)/\ell}{1+2\max\{1/(\rho\ell),t/(L\ell)\}},
$$
which still belongs to $\cB_t$ 
(and thus to $\cB$, with $I_{\beta_\ell}=[0,t]$), 
but additionally satisfies that $(s,\beta_\ell(s)) \in \mathring{D}_t$ (and in particular 
$\beta_\ell(s) \in (0,L)$) for all $s\in (0,t)$, so that $\cI_{t}(g,\beta_\ell)
=\cJ_{ \mathring{D}_t}(g,\beta_\ell) \leq \Lambda_{\mathring{D}_t}(g)$.
And one immediately checks, by dominated convergence and because $H$ is continuous on $[0,L]$, that 
$\cI_{t}(g,\beta)=\lim_\ell \cI_{t}(g,\beta_\ell)$.

We next verify that $\Lambda_{\mathring{D}_{t}}(g) \leq \Gamma_t(g)$.
Fix $\beta \in \cB$, defined on some interval $I_\beta=[x,y]$. We define $\bar \beta$ as the 
restriction/extension
of $\beta$ to $[0,t]$ defined as follows: we set $\bar\beta(s)=\beta(s)$ for $s\in[x\lor 0,y\land t]$,
$\bar\beta(s)=\beta(x\lor 0)$ for $s\in[0,x\lor 0]$ and $\bar\beta(s)=\beta(y\land t)$ for $s\in[y\land t,t]$.
Clearly, $\cJ_{\mathring{D}_t}(g,\beta) \leq \cJ_{ \mathring{D}_t}(g,\bar\beta)$. 
Next, we set $\gamma(s)= (0 \lor \bar\beta (s)) \land L \land (\rho(t-s))$ for all $s\in [0,t]$.
We then have $\cJ_{\mathring{D}_t}(g,\bar\beta) = \cJ_{ \mathring{D}_t}(g,\gamma)$, because 
$\bar \beta(s)=\gamma(s)$ and $\bar \beta'(s)=\gamma'(s)$ 
for all $s\in [0,t]$ such that $(s,\bar \beta(s))\in \mathring{D}_t$
and since $(s,\bar \beta(s))\in \mathring{D}_t$ if and only if $(s,\gamma(s))\in \mathring{D}_t$.
And clearly, $\cJ_{ \mathring{D}_t}(g,\gamma) \leq \cI_t(g,\gamma)$. 
Finally, $\cI_t(g,\gamma) \leq \Gamma_t(g)$, because even if $\gamma \notin \cB_t$,
it is defined from $[0,t]$ into $[0,L]$, vanishes at $t$ and is $\rho$-Lipschitz continuous.
Hence we can find a sequence $(\gamma_\ell)_{\ell\geq 1}$ of elements of $\cB_t$ such that
$\gamma_\ell \to \gamma$ uniformly and $\gamma_\ell'\to\gamma'$ a.e., which is sufficient to ensure
us that $\cI_t(g,\gamma)=\lim_\ell \cI_t(g,\gamma_\ell)$ by dominated convergence.

We obviously have $\Lambda_{\mathring{D}_{t}}(g) \leq \lim_{\delta \downarrow 0} \Lambda_{\mathring{D}_{t+\delta}}(g)$.

It only remains to verify that $\lim_{\delta \downarrow 0} \Gamma_{t+\delta}(g)\leq \Gamma_t(g)$.
For $\beta \in \cB_{t+\delta}$, we introduce the function $\beta_\delta\in \cB_{t}$ defined by
$\beta_\delta(s)=\max\{\beta(s)-\rho\delta,0\}$.
It satisfies
$$
\cI_t(g,\beta_\delta)\geq \sqrt{\frac 2\rho}\intot
\sqrt{H(\beta_\delta(s)) g(s) [\rho^2-(\beta'(s))^2]} \dd s
$$
because $|\beta_\delta'(s)|\leq |\beta'(s)|$ a.e. Hence
$\cI_t(g,\beta_\delta)\geq \cI_{t+\delta}(g,\beta) - \psi_t(\delta)$, where
\begin{align*}
\psi_t(\delta)=& \sqrt{2\rho} \Big(\intot \sqrt{g(s)}\dd s\Big) 
\sup_{x,y\in [0,L],|x-y|\leq \rho \delta} |\sqrt{H(x)}-\sqrt{H(y)}| + \sqrt{2\rho ||H||_\infty}\int_t^{t+\delta}
\sqrt{g(s)}\dd s,
\end{align*}
which tends to $0$ as $\delta\downarrow 0$, because $H$ is continuous and $g$ is locally bounded.
Furthermore, $\beta_\delta$ is $\rho$-Lipschitz continuous,
$[0,L]$-valued, and $\beta_\delta(t)=0$ (because $\beta(t+\delta)=0$, whence $\beta(t)\leq \delta$
since $\beta$ is $\rho$-Lipschitz continuous). 
Hence, as a few lines above, we can approximate $\beta_\delta$ by a sequence of elements
of $\cB_t$ and deduce that $\Gamma_t(g)\geq \cI_t(\beta_\delta)$.
Consequently, for all $\beta \in \cB_{t+\delta}$, we have
$\cI_{t+\delta}(\beta) \leq \cI_t(\beta_\delta) + \psi_t(\delta) \leq \Gamma_t(g) + \psi_t(\delta)$,
whence $\Gamma_{t+\delta}(g) \leq \Gamma_t(g) + \psi_t(\delta)$.
\end{proof}

\section{The hard model}\label{ph}

We first give the

\begin{proof}[Proof of Proposition~\ref{debil1}]
Let $f_0\in\cP([v_{min},v_{max}))$ and $r:[0,\infty)\mapsto \R_+$, continuous, non-decreasing and such 
that $r_0=0$.
Consider $V_0\sim f_0$. The process $(V^r_t,J^r_t)_{t\geq 0}$ can be built as follows (and is unique 
because there
is no choice in the construction): set $Z^0_t=V_0+I t+r_t$ (for all $t\geq 0$) 
and $S_0=\inf\{ t \geq 0 : Z^0_t= v_{max}\}$, 
which is positive and finite, put $V^r_t=Z^0_t$ and $J^r_t=0$ for $t \in [0,S_0)$; 
set $Z^1_t=v_{min}+I(t-S_0)+(r_t-r_{S_0})$ (for all $t\geq S_0$) and $S_1=\inf\{ t\geq S_0 : 
Z^1_t= v_{max}\}$, 
put $V^r_t=Z_t^1$ and $J^r_t=1$ for $t \in [S_0,S_1)$; 
set $Z^2_t=v_{min}+I(t-S_1)+(r_t-r_{S_1})$ (for all $t\geq S_1$) and $S_2=\inf\{t\geq S_1 : 
Z^2_t= v_{max}\}$, 
put $V^r_t=Z_t^2$ and $J^r_t=2$ for $t \in [S_1,S_2)$, etc.

Observe that
\begin{equation}\label{conc}
V^r_t= v_{min}+ (v_{max}-v_{min})\Big\{\frac{V_0+I t+r_t-v_{min}}{v_{max}-v_{min}}\Big\}
\quad \hbox{and}\quad J_t^r=\Big\lfloor\frac{V_0+I t+r_t-v_{min}}{v_{max}-v_{min}}\Big\rfloor
\end{equation}
for all $t\geq 0$, where $\lfloor x \rfloor$ and $\{ x \}$ stand for the integer and 
fractional part of $x \in [0,\infty)$.
\end{proof}

We next handle the

\begin{proof}[Proof of Proposition~\ref{prel1}]
We recall that a non-decreasing $C^1$-function $r:[0,\infty)\mapsto {[0,\infty)}$
with $r_0=0$ is fixed, as well as the density $f_0$ on $[v_{min},v_{max}]$ of $V_0$,
and that $V_t^r=V_0+It+r_t+(v_{min}-v_{max})J^r_t$, where 
$J^r_t=\sum_{s\leq t} \indiq_{\{V^r_\sm=v_{max}\}}$.
We recall that the increasing sequence $(a_k)_{k\geq 0}$ is defined by 
$Ia_k+r_{a_k}=k(v_{max}-v_{min})$.

\emph{Step 1.} We first observe that for all $k\geq 0$, $V^r_{a_k}=V_0$. This is immediate from
\eqref{conc}, since 
$$
V^r_{a_k}=v_{min} + (v_{max}-v_{min})\Big\{\frac{V_0+k(v_{max}-v_{min})-v_{min}}{v_{max}-v_{min}}\Big\}
=v_{min} + (v_{max}-v_{min})\Big\{\frac{V_0-v_{min}}{v_{max}-v_{min}}\Big\},
$$
which equals $V_0$ since because $V_0 \in [v_{min},v_{max})$ a.s.

\emph{Step 2.} We denote by $0\leq S_0<S_1<S_2<\dots$ 
the instants of jump of $(J^r_t)_{t\geq 0}$ (so that $S_k$ is the $(k+1)$-th 
instant of jump).
Here we prove by induction that for all $k\geq 1$, $S_k$ a.s. belongs to $[a_{k},a_{k+1}]$ 
and that its law has the density
$$h_{k}(s)= f_0(k(v_{max}-v_{min})+v_{max}-I s-r_s)(I+r'_s)\indiq_{\{s\in [a_k,a_{k+1}]\}}.$$

To this end, we introduce the function $m(t)=It+r_t$ (which increases from 
$[0,\infty)$ into itself), and its inverse function
$m^{-1}:[0,\infty)\mapsto [0,\infty)$. We have $m(a_k)=k(v_{max}-v_{min})$ for all $k\geq 0$.

First, we have $v_{max}=V^r_{S_0-}=V_0+IS_0+r_{S_0}=V_0+m(S_0)$, whence $S_0=m^{-1}(v_{max}-V_0)$.
But $m^{-1}$ is increasing and $V_0 \geq v_{min}$, so that $S_0 \leq m^{-1}(v_{max}-v_{min})=a_1$.
Thus $S_0 \in [a_0,a_1]$ a.s. (recall that $a_0=0$) and a simple computation shows that its density is given by
$h_0(s)=f_0(v_{max}-m(s))m'(s)\indiq_{\{s\in [a_0,a_1]\}}$ as desired.

We next fix $k\geq 1$ and assume that $S_{k-1} \in [a_{k-1},a_{k}]$ a.s. Then we write
$v_{max}=V^r_{S_{k}-}=v_{min}+ I(S_{k}-S_{k-1})+r_{S_{k}}-r_{S_{k-1}}=v_{min}
+m(S_{k})-m(S_{k-1})$, so that $m(S_{k}) = m(S_{k-1})+v_{max}-v_{min}$.
Using that $m(S_{k-1})\in [m(a_{k-1}),m(a_{k})]=[(k-1)(v_{max}-v_{min}),k(v_{max}-v_{min})]$, 
we conclude that $m(S_{k})$ belongs to $[k(v_{max}-v_{min}),(k+1)(v_{max}-v_{min})]$, 
which precisely means that $S_{k} \in [a_{k},a_{k+1}]$. 

We now write $v_{max}=V_{S_k-}^r=V^r_{a_k}+I (S_k-a_{k})+ r_{S_k}-r_{a_{k}}=V^r_{a_k}+m(S_k)-m(a_{k}) = 
V_0+m(S_k)-k(v_{max}-v_{min})$ by Step 1,
whence $S_k=m^{-1}(v_{max}+k(v_{max}-v_{min})-V_0)$, and a computation shows that
the density of $S_k$ is given by $h_k(s)=f_0(k(v_{max}-v_{min})+v_{max}-m(s))m'(s)
\indiq_{\{s\in [a_k,a_{k+1}]\}}$.

\emph{Step 3.} For each $k\geq 0$, $h_k$ is continuous on $[a_k,a_{k+1}]$, since 
$r$ is of class $C^1$ by assumption, since $k(v_{max}-v_{min})+v_{max}-Is-r_s$
takes values, during $[a_k,a_{k+1}]$, in $[v_{min},v_{max}]$ and since $f_0$ is continuous
on $[v_{min},v_{max}]$ by (H1).

\emph{Step 4.} We can apply Lemma~\ref{lcru}, of which all the assumptions are satisfied,
with $b_k=a_k$. Indeed, recalling the statement,
$\nu_N^r=\sum_{i=1}^N \sum_{k\geq 1} \delta_{(T_k^i,X_i)}$, which can be written as
$\nu_N^r= \sum_{i=1}^N \sum_{k\geq 0} \delta_{(S_k^i,X_i)}= \sum_{i=1}^n\sum_{k\geq 0} 
\indiq_{\{S_k^i  \leq b_{k+1}\}}\delta_{(S_k^i,X_i)}$, since
$S_k^i=T^i_{k+1}$
a.s. belongs to $[b_k,b_{k+1}]$. Since the density of $S_k^i$ is nothing but
$h_k$ by Step 2 and since $h_k$ is continuous on $[b_k,b_{k+1}]$ by Step 3,
we conclude that for all $t\geq 0$, 
$\lim_{N \to \infty} N^{-1/2} A_t(\nu_N^r)=\Gamma_t(g_r)$ a.s., 
where $g_r(t)=\sum_{k\geq 0} h_k(t)\indiq_{\{t \in [b_k,b_{k+1})\}}$,
as desired.
\end{proof}

We finally give the

\begin{proof}[Proof of Theorem~\ref{mr1}]
We fix $w>0$.
By (H2) and Remark~\ref{ttip}, $\Gamma_t(g)=\sqrt{2\rho H(0)}\int_0^t\sqrt{g(s)}\dd s$.
We say that $\kappa:[0,\infty)\mapsto[0,\infty)$ is a solution if it is of class $C^1$, non-decreasing,
if $\kappa_0=0$ and if $\kappa_t=w \sqrt{2\rho H(0)}\int_0^t\sqrt{g_\kappa(s)}\dd s$ 
for all $t\geq 0$, $g_\kappa$ being defined in Proposition~\ref{prel1}. 
To be as precise as possible, we indicate in superscript that $(a_k)_{k\geq 0}$ depends on $\kappa$.
For all $k\geq 0$, $a_k^\kappa$ is thus defined by $Ia_k^\kappa+\kappa_{a_k^\kappa}=k(v_{max}-v_{min})$.
We always have $a_0^\kappa=0$.
We recall that $\sigma=\rho H(0) w^2$, that $G_0=\sigma f_0+\sqrt{\sigma^2 f_0^2+ 2 \sigma I  f_0}: 
[v_{min},v_{max}]\mapsto \R_+$ and that 
$\varphi_0(x)=\int_x^{v_{max}}\dd v /[I+G_0(v)] : [v_{min},v_{max}]\mapsto [0,a]$,
where $a=\varphi_0(v_{min})$.

\emph{Step 1.} For any solution $\kappa$, it holds that $a_1^\kappa=a$ and $\kappa_t=v_{max}-It-\varphi_0^{-1}(t)$
on $[0,a]$.

Indeed, we have $\kappa'_t=w\sqrt{2\rho H(0)g_\kappa(t)}
=\sqrt{2\sigma f_0(v_{max}-It-\kappa_t)(I+\kappa'_t)}$ on $[0,a_1^\kappa]$,
from which $\kappa'_t=G_0(v_{max}-It-\kappa_t)$.
Thanks to (H2), $G_0$ is Lipschitz continuous, so that this ODE has a unique solution such that 
$\kappa_0=0$, given by $\kappa_t=v_{max}-It-\varphi_0^{-1}(t)$. We also deduce that necessarily,
$v_{max}-v_{min}=Ia_1^\kappa+\kappa_{a_1^\kappa}=I a_1^\kappa+v_{max}-Ia_1^\kappa-\varphi_0^{-1}(a_1^\kappa)$, i.e. 
$\varphi_0^{-1}(a_1^\kappa)=v_{min}$, whence $a_1^\kappa =a$. 

\emph{Step 2.} For any solution $\kappa$, $a_2^\kappa=2 a$ and 
$\kappa_t=(v_{max}-v_{min})+v_{max}-It-\varphi_0^{-1}(t-a)$ on $[a,2a]$.

Indeed, $\kappa'_t=w\sqrt{2\rho H(0)g_\kappa(t)}
=\sqrt{2\sigma f_0((v_{max}-v_{min})+v_{max}-It-\kappa_t)(I+\kappa'_t)}$ on $[a,a_2^\kappa]$.
This implies that 
$\kappa'_t=G_0((v_{max}-v_{min})+v_{max}-It-\kappa_t)$ on $[a,a_2^\kappa]$.
Since $G_0$ is Lipschitz continuous by (H2), this ODE has a unique solution such that 
$\kappa_a=v_{max}-Ia-\varphi_0^{-1}(a)=v_{max}-v_{min}-Ia$ (we require that $\kappa$ is continuous and $\kappa_{a-}$ 
has been determined in Step 1), which is given by
$\kappa_t=(v_{max}-v_{min})+v_{max}-It-\varphi_0^{-1}(t-a)$
(observe that $\varphi_0^{-1}(0)=v_{max}$).
Also, we deduce that
$a_2^\kappa=2a$, because $2(v_{max}-v_{min})=I a_2^\kappa+\kappa_{a_2^\kappa}=(v_{max}-v_{min})+v_{max}
-\varphi_0^{-1}(a_2^\kappa-a)$, 
i.e. $\varphi_0^{-1}(a_2^\kappa-a)=v_{min}$, whence $a_2^\kappa-a=a$.

\emph{Step 3.} Iterating the procedure, we conclude that for any solution, we have $a_k^\kappa=ka$
for all $k\geq 0$ and $\kappa_t=k(v_{max}-v_{min})+v_{max}-It-\varphi_0^{-1}(t-ka)$ on $[ka,(k+1)a]$.
Thus uniqueness is checked, and we only have to verify that this function is indeed a solution.
It is continuous by construction, it is of course $C^1$ and non-decreasing on each interval $(ka,(k+1)a)$,
because $\kappa_t'= -I-(\varphi_0^{-1})'(t-ka)=G_0(\varphi_0^{-1}(t-ka))\geq 0$.
It is actually $C^1$ on $[0,\infty)$ because for each $k\geq 1$, 
we have $\kappa'_{ka+}=\kappa'_{ka-}$. Indeed, $\kappa'_{ka+}=G_0(\varphi_0^{-1}(0))=G_0(v_{max})$, while 
$\kappa'_{ka-}=G_0(\varphi_0^{-1}(a))=G_0(v_{min})$, and the two values coincide because $f_0(v_{max})=f_0(v_{min})$ 
by (H2).

Finally, we have $\kappa_t=w \sqrt{2\rho H(0)}\int_0^t\sqrt{g_\kappa(s)}\dd s$ for all $t\geq 0$,
since $\kappa $ is continuous, starts from $0$, since $\kappa'_t=w \sqrt{2\rho H(0)}\sqrt{g_\kappa(t)}$ 
for all $t \in \R_+\setminus \{ka : k\geq 1\}$ by construction and since both $\kappa'$ and
$g_\kappa$ are continuous. Recalling the definition of $g_\kappa$, this last assertion easily follows from the 
facts that
$\kappa \in C^1([0,\infty))$, that $f_0$ is continuous on $[v_{min},v_{max}]$, that $f_0(v_{max})=f_0(v_{min})$ and that 
for all $k\geq 1$, 
$a_k=ka$ and $\kappa'_{ka+}=\kappa'_{ka-}$.
\end{proof}

\section{The soft model}\label{ps}

We start with the

\begin{proof}[Proof of Proposition~\ref{debil2}]
The existence of a pathwise unique solution $(V^r_t)_{t \geq 0}$ to \eqref{sde}, 
with values in $[v_{min},\infty)$, is 
classical and relies on the following main arguments (here the continuity of $\lambda$ is not required,
one could assume only that $\lambda: [v_{min},\infty) \mapsto \R_+$ is measurable and locally bounded).

$\bullet$ Extend $F$ to a locally Lipschitz continuous function on $\R$ and $\lambda$ to a locally bounded
function on $\R$. There is obviously \emph{local} existence of a pathwise unique solution to \eqref{sde}.
The only problem is to check non-explosion (i.e. to check that a.s., $\sup_{[0,T]} |V^r_t|<\infty$ 
for all $T>0$).

$\bullet$ Any solution remains in $[v_{min},\infty)$, because (a) $r_t$ is non-decreasing, (b) $F$ is locally 
Lipschitz continuous and $F(v_{min}) \geq 0$ and (c) each jump sends the solution to $v_{min}$.

$\bullet$ Since $F(v) \leq C(1+(v-v_{min}))$
and since all the jumps are negative, any solution $(V^r_t)_{t \geq 0}$ satisfies
$V^r_t \leq V_0 + r_t + C\intot (1+(V^r_s-v_{min}))\dd s$ for all $t\geq 0$, whence, 
$\sup_{[0,T]} (V^r_t-v_{min}) \leq (V_0-v_{min} +CT+r_T)e^{CT}$ by the Gronwall lemma.

$\bullet$ The two previous points prevent us from explosion, so that the pathwise unique solution is global.
Furthermore, $\E[\sup_{[0,T]} (V^r_t-v_{min})^p] <\infty$ because $\E[(V_0-v_{min})^p]<\infty$
by assumption.
\end{proof}

We next give the
\begin{proof}[Proof of Proposition~\ref{prel2}]
We recall that a non-decreasing continuous function $r:[0,\infty)\mapsto 0$
with $r_0=0$ is fixed, as well as the initial distribution $f_0$ on $[v_{min},\infty)$ of $V_0$,
that $(V^r_t)_{t\geq 0}$ is the unique solution to \eqref{sde} and that
$J^r_t=\sum_{s\leq t} \indiq_{\{\Delta V^r_\sm \neq 0 \}}$.

\emph{Step 1.} For $t_0\geq 0$ and $v_0 \geq v_{min}$ we define $(z_{t_0,v_0}(t))_{t\geq t_0}$ as the unique
solution to $z_{t_0,v_0}(t)=v_0+\int_{t_0}^t F(z_{t_0,v_0}(s))\dd s + r_t-r_{t_0}$. It is valued in $[v_{min},\infty)$
because $F(v_{min})\geq 0$ (and $r$ is non-decreasing).
For all $t_0<t_1\leq t$, we have $z_{t_1,v_{min}}(t) \leq z_{t_0,v_{min}}(t)$. This follows from the comparison 
theorem, because $z_{t_1,v_{min}}(t_1)=v_{min}\leq z_{t_0,v_{min}}(t_1)$ and since  
$(z_{t_1,v_{min}}(t))_{t\geq t_1}$ and $(z_{t_0,v_{min}}(t))_{t\geq t_1}$ solve the same Volterra equation (with different initial
conditions).
Also,  since $F(v) \leq C(1+(v-v_{min}))$, we have 
$z_{t_0,v_{0}}(t)-v_{min} \leq [v_0-v_{min}+r_t-r_{t_0}+C(t-t_0)]\exp(C(t-t_0))$ for all $t\geq t_0$, 
all $v_0\geq v_{min}$.

\emph{Step 2.} By (S1), we have $\lambda(v)=0$ on $[v_{min},\alpha]$, with $\alpha>v_{min}$.
We claim that there is an increasing sequence $(a_k)_{k\geq 0}$ such that 
$\lim_k a_k=\infty$ and a.s., for all $k\geq 0$, $J^r_{a_{k+1}}-J^r_{a_k}\in\{0,1\}$.

We introduce the increasing sequence $(a_k)_{k\geq 0}$
defined recursively by $a_0=0$ and, for $k\geq 0$, 
$a_{k+1}=\inf\{t\geq a_k : z_{a_k,v_{min}}(t) \geq \alpha  \} \land (a_{k}+1)$, with the convention that
$\inf \emptyset = \infty$.

Fix $k\geq 0$, denote by $\tau_1<\tau_2$ 
the first and second instants of jump of $V^r$ after $a_k$. If $\tau_1 > a_{k+1}$, then $J^r_{a_{k+1}}-J^r_{a_k}=0$.
Otherwise, we have $V^r_{\tau_1}=v_{min}$ and, during $[\tau_1,\tau_2)$, we have
$V^r_t=z_{\tau_1,v_{min}}(t)$, whence $V^r_t\leq z_{a_k,v_{min}}(t)$ by Step 1 and since $\tau_1\geq a_k$, and thus 
$V^r_t\leq \alpha$ during $[\tau_1,\tau_2\land a_{k+1})$. Thus the rate of jump $\lambda(V^r_t)=0$
during $[\tau_1,\tau_2\land a_{k+1})$,
so that $\tau_2 >a_{k+1}$ and $J^r_{a_{k+1}}-J^r_{a_k}=1$.

It remains to verify that $\lim_k a_k=\infty$.
We fix $\eta\in (0,1)$ such that $v_{min}+C\eta e^{C\eta}\leq (v_{min}+\alpha)/2$ and we set
$\e=(\alpha-v_{min})e^{-C\eta}/2$. We claim that for all $k\geq 0$, we have either $a_{k+1}-a_k\geq \eta$
or $r_{a_{k+1}}-r_{a_k}\geq \e$. Indeed, if $a_{k+1}-a_k < \eta\leq 1$, then 
$a_{k+1}= \inf\{t\geq a_k : z_{a_k,v_{min}}(t) \geq \alpha  \}$, whence
$\alpha=z_{a_k,v_{min}}(a_{k+1})\leq v_{min}+ (r_{a_{k+1}}-r_{a_k} + C(a_{k+1}-a_k))e^{C(a_{k+1}-a_k)}$ by Step 1.
Hence $\alpha \leq v_{min}+ (r_{a_{k+1}}-r_{a_k})e^{C\eta}+C\eta e^{C\eta}\leq (v_{min}+\alpha)/2
+(r_{a_{k+1}}-r_{a_k})e^{C\eta}$,
whence $r_{a_{k+1}}-r_{a_k} \geq \e$.

One easily concludes that $\lim_k a_k=\infty$: if  $a_\infty=\lim_k a_k<\infty$, then there is $k_0$ such that
$a_{k+1}-a_k < \eta$ (and thus  $r_{a_{k+1}}-r_{a_k} \geq \e$) for all $k\geq k_0$, whence
$r_{a_\infty}=\sum_{k\geq 0} (r_{a_{k+1}}-r_{a_k})=\infty$. This is not possible since $r$ is $\R_+$-valued.

\emph{Step 3.} For $k\geq 0$, let $S_k=\inf\{ t \geq a_k : \Delta V^r_t \neq 0 \}=
\inf\{ t \geq a_k : \Delta J^r_t =1\}$. The law of $S_k$ has a continuous density
$g_k$ on $[a_k,\infty)$. 

Since $V^r_t=z_{a_k,V^r_{a_k}}(t)$ during $[a_k,S_k)$ and since $V^r$ jumps, at time $t$, at rate 
$\lambda(V^r_{t-})$, 
we have $\Pro(S_{k}\geq t |V^r_{a_k})=\exp(-\int_{a_k}^t \lambda(z_{a_k,V^r_{a_k}}(s))\dd s)$ 
for $t\geq a_k$, whence 
$$
\Pro(S_{k}\geq t)=\E\Big[\exp\Big(-\int_{a_k}^t \lambda(z_{a_k,V^r_{a_k}}(s))\dd s\Big)\Big].
$$ 
But $t \mapsto \lambda (z_{a_k,V^r_{a_k}}(t))$ is a.s. continuous on $\R_+$.
Furthermore, $\E[\sup_{[a_k,T]}\lambda (z_{a_k,V^r_{a_k}}(t))] <\infty$ for all $T>a_k$:
by (S1) and Step 1, $\sup_{[a_k,T]}\lambda (z_{a_k,V^r_{a_k}}(t))
\leq C(1+\sup_{[0,T]}(z_{a_k,V^r_{a_k}}(t)-v_{min})^p)\leq C(1+[V^r_{a_k}-v_{min}+r_{T}-r_{a_k} +C(T-a_k))e^{C(T-a_k)}]^p)$,
which has a finite expectation because $\E[(V^r_{a_k}-v_{min})^p]<\infty$ by Proposition~\ref{debil2}.
We easily deduce that indeed,
$S_k$ has the continuous density $g_k(t)=\E[\lambda(z_{a_k,V^r_{a_k}}(t))
\exp(-\int_{a_k}^t \lambda(z_{a_k,V^r_{a_k}}(s))\dd s)]$ on $[a_k,\infty)$.

\emph{Step 4.} Setting $h_r(t)=\sum_{k\geq 0} g_k(t)\indiq_{\{t \in [a_k,a_{k+1})\}}$, it holds
that $h_r(t)=\E[\lambda(V^r_t)]$ for a.e. $t\geq 0$.

On the one hand, we have $J^r_t=\intot\int_0^\infty \indiq_{\{u\leq \lambda(V^r_\sm)\}}\pi(\dd s,\dd u)$, so that
$\E[J^r_t]=\intot \E[\lambda(V^r_s)]\dd s$. On the other hand, since $V^r$ has at most one jump in each
time interval $[a_k,a_{k+1})$, one easily checks that $J^r_t=\sum_{k\geq 0} \indiq_{\{S_k \leq t, S_k < a_{k+1}\}}$.
Hence $\E[J^r_t]=\sum_{k\geq 0} \Pro(S_k \leq t \land a_{k+1})$, so that
$\E[J^r_t]=\intot (\sum_{k\geq 0} g_k(s)\indiq_{\{s \leq a_{k+1}\}})\dd s= \intot h_r(s)\dd s$. 
We thus have $\intot \E[\lambda(V^r_s)]\dd s= \intot h_r(s)\dd s$ for all $t\geq 0$, which completes the step.

\emph{Step 5.}
Observe that for $T_1<T_2<\dots$ the successive instants of jump of $(V^r_t)_{t\geq 0}$,
we have 
$$
\sum_{\ell\geq 1} \delta_{T_\ell+\theta}= \sum_{k\geq 1} \delta_{S_k+\theta}\indiq_{\{S_k+\theta<a_{k+1}+\theta\}},
$$
because for each $k\geq1$, $S_k$ is the first instant of jump of $(V^r_t)_{t\geq 0}$ after $a_k$
and since $(V^r_t)_{t\geq 0}$ has at most one jump during $[a_k,a_{k+1})$.

Hence, coming back to the notation of the statement, 
$$
\nu_N^r=\sum_{i=1}^N \sum_{k\geq 1} \delta_{(T_k^i+\theta,X_i)} = \sum_{i=1}^N \sum_{k\geq 0} 
\indiq_{\{S_k^i+\theta < b_{k+1}\}}\delta_{(S_k^i+\theta,X_i)},
$$ 
with $b_k=a_k+\theta$. 
We thus can directly apply Lemma~\ref{lcru} to conclude that indeed,
for any $t\geq 0$, $\lim_{N \to \infty} N^{-1/2} A_t(\nu_N^r)=\Gamma_t(h_r^\theta)$ a.s., where
$h_r^\theta(t)=\sum_{k\geq 0} g_k(t-\theta)\indiq_{\{t \in [a_k+\theta,a_{k+1}+\theta)\}}$ (observe that the density of
$S_k^i+\theta$ is $g_k(t-\theta)\indiq_{\{t\geq \theta\}}= g_k(t-\theta)\indiq_{\{t \geq a_k+\theta\}}$),
whence indeed $h_r^\theta(t)=\E[\lambda(V^r_{t-\theta})]\indiq_{\{t\geq\theta\}}$ by Step 4.
\end{proof}

Before concluding, we need a few preliminaries on the functional $\Gamma$.

\begin{lem}\label{prelg}
For any measurable locally bounded $g,\tg:[0,\infty)\mapsto \R_+$ and all $t\geq 0$, 
$$
\Gamma_t(g) \leq \sqrt{2\rho||H||_\infty}\intot \sqrt{g(s)}\dd s, \qquad
|\Gamma_t(g)-\Gamma_t(\tg)| \leq \sqrt{2\rho||H||_\infty}\intot |\sqrt{g(s)}-\sqrt{\tg(s)}|\dd s,
$$
and, for $0\leq t \leq t+\delta$,
\begin{align*}
0\leq \Gamma_{t+\delta}(g)-\Gamma_{t}(g) \leq &
\sqrt{2\rho} \Big(\intot \sqrt{g(s)}\dd s\Big) \sup_{x,y\in [0,L],|x-y|\leq \rho \delta} |\sqrt{H(x)}-
\sqrt{H(y)}| \\
&+ \sqrt{2\rho ||H||_\infty}\int_t^{t+\delta} \sqrt{g(s)}\dd s.
\end{align*}
\end{lem}

\begin{proof}
The two first inequalities follow from the facts that for all
$\beta \in \cB_t$, we have
$\cI_t(g,\beta)\leq \sqrt{2\rho||H||_\infty}\intot \sqrt{g(s)}\dd s$ and
$|\cI_t(g,\beta)-\cI_t(\tg,\beta)| \leq \sqrt{2\rho||H||_\infty}\intot |\sqrt{g(s)}-\sqrt{\tg(s)}|\dd s$,
both facts being obvious by definition of $\cI_t$, see Definition~\ref{dfG}.
The last one has already been verified at the end of the proof of Lemma~\ref{lcru}.
\end{proof}

We finally provide the

\begin{proof}[Proof of Theorem~\ref{mr2}]
\emph{Point (i)}. First assume that we have 
$\kappa:[0,\infty)\mapsto[0,\infty)$ continuous, non-decreasing, starting from $0$ and such that
$\kappa_t=w \Gamma_t((h_\kappa^\theta(s))_{s\geq 0})$ for all $t\geq 0$, where
$h_\kappa^\theta(t)=\E[\lambda(V^\kappa_{t-\theta})]\indiq_{\{t\geq \theta\}}$, 
$(V^\kappa_t)_{t\geq 0}$ being the unique solution
to \eqref{sde} with $r=\kappa$. Then $(V^\kappa_t)_{t\geq 0}$ is obviously a solution to \eqref{nsde}.
It is $[v_{min},\infty)$-valued and satisfies $\E[\sup_{[0,T]} (V^\kappa_t-v_{min})^p]<\infty$ for all $T>0$
by Proposition~\ref{debil2}, and we indeed have 
$\kappa_t=w \Gamma_t((\E[\lambda(V^\kappa_{s-\theta})]\indiq_{\{s\geq \theta\}})_{s\geq 0})$.

Consider a  $[v_{min},\infty)$-valued solution  $(V_t)_{t\geq 0}$ to \eqref{nsde} such that
$\E[\sup_{[0,T]} (V_t-v_{min})^p]<\infty$ for all $T>0$ (so that $(\E[\lambda(V_{s-\theta})]
\indiq_{\{s\geq \theta\}})_{s\geq 0}$ is locally bounded
thanks to (S1)) and define $\kappa_t=w \Gamma_t((\E[\lambda(V_{s-\theta})]
\indiq_{\{s\geq \theta\}})_{s\geq 0})$, 
which is non-decreasing and continuous by  Lemma~\ref{prelg}
(because $H$ is continuous by (S1)).
Then $(V_t)_{t\geq 0}$ solves \eqref{sde} with $r=\kappa$, so that $(V_t)_{t\geq 0}=(V^\kappa_t)_{t\geq 0}$.
Consequently, $\kappa_t=w \Gamma_t((\E[\lambda(V^\kappa_{s-\theta})]\indiq_{\{s\geq \theta\}})_{s\geq 0})$ for all $t\geq 0$
as desired.

\emph{Point (ii) when $\theta>0$}. We first recall, see Definition~\ref{dfG}, 
that for any $g:[0,\infty)\mapsto \R_+$,
any $t\geq 0$, $\Gamma_t(g)$ actually depends only on $(g(s))_{s\in [0,t]}$. Moreover, $\Gamma_t(g)=0$
if $g(s)=0$ for all $s\in [0,T]$. Consequently, $\kappa_t=w\Gamma_t((\E[\lambda(V_{s-\theta})]
\indiq_{\{s\geq \theta\}})_{s\geq 0})=0$ for all $t\in [0,\theta]$ and \eqref{nsde} rewrites, 
during $[0,\theta]$,
$$
V_t=V_0+ \intot F(V_s)\dd s + \intot\int_0^\infty (v_{min}-V_\sm)\indiq_{\{u\leq \lambda(V_\sm)\}} \pi(\dd s, \dd u).
$$
This equation has a pathwise unique solution, see Proposition~\ref{debil2}, which is furthermore 
$[v_{min},\infty)$-valued and we have $\E[\sup_{[0,\theta]} (V_t-v_{min})^p]<\infty$. This determines 
$\E[\lambda(V_{s})]$ for all $s\in [0,\theta]$, and this quantity is well-defined and bounded, since
$\lambda(v) \leq C(1+(v-v_{min}))^p$ for all $v\geq v_{min}$.

Hence $\kappa_t=w\Gamma_t((\E[\lambda(V_{s-\theta})]\indiq_{\{s\geq \theta\}})_{s\geq 0})$ 
is entirely determined for all $t\in [\theta,2\theta]$. 
It is furthermore non-decreasing and continuous (by Lemma~\ref{prelg}, since $H$ is continuous
and since $\E[\lambda(V_{s})]$ is bounded on $[0,\theta]$). And
\eqref{nsde} rewrites, on $[\theta,2\theta]$,
$$
V_t=V_{\theta}+ \int_\theta^t F(V_s)\dd s + (\kappa_t-\kappa_\theta)
+\int_\theta^t\int_0^\infty (v_{min}-V_\sm)\indiq_{\{u\leq \lambda(V_\sm)\}} \pi(\dd s, \dd u).
$$
This equation has a pathwise unique solution, see Proposition~\ref{debil2}, which is furthermore 
$[v_{min},\infty)$-valued and we have $\E[\sup_{[\theta,2\theta]} (V_t-v_{min})^p]<\infty$.
This determines $\E[\lambda(V_{s})]$ for all $s\in [\theta,2\theta]$ 
(and this quantity is well-defined and bounded).

Hence $\kappa_t=w\Gamma_t((\E[\lambda(V_{s-\theta})]\indiq_{\{s\geq \theta\}})_{s\geq 0})$ 
is entirely determined for all $t\in [2\theta,3\theta]$. 
It is furthermore non-decreasing and continuous. And
\eqref{nsde} rewrites, on $[2\theta,3\theta]$,
$$
V_t=V_{2\theta}+ \int_{2\theta}^t F(V_s)\dd s + (\kappa_t-\kappa_{2\theta})
+\int_{2\theta}^t\int_0^\infty (v_{min}-V_\sm)\indiq_{\{u\leq \lambda(V_\sm)\}} \pi(\dd s, \dd u).
$$
This equation has a pathwise unique solution, see Proposition~\ref{debil2}, etc.

Working recursively on the time intervals $[k\theta,(k+1)\theta]$,
we see that there is a pathwise unique $(V_t)_{t\geq 0}$ solving \eqref{nsde}, it
is $[v_{min},\infty)$-valued and satisfies $\E[\sup_{[0,T]} (V_t-v_{min})^p]<\infty$ for all $T>0$.

\emph{Point (ii) when $\theta=0$ under (S2)}. We fix $T>0$ and work on $[0,T]$. 

First, for any solution $(V_t)_{t\geq 0}$ to \eqref{nsde} such that  $\E[\sup_{[0,T]} (V_t-v_{min})^p]<\infty$, there
exists $M>0$ such that a.s., $\sup_{[0,T]} V_t\leq M$. Indeed, we observe that
$K=\sup_{[0,T]}\Gamma_t(\E[\lambda(V_{s})])_{s\geq 0}) <\infty$ by Lemma~\ref{prelg} and since
$\lambda(v) \leq C(1+(v-v_{min}))^p$.
Hence, $V_t\leq V_0+K+C\intot (1+(V_s-v_{min}))\dd s$ by (S1).
Since $V_0$ is bounded by (S2), the conclusion follows from the Gronwall Lemma.

Next, we prove that for any solution $(V_t)_{t\geq 0}$ to \eqref{nsde}
such that  $\E[\sup_{[0,T]} (V_t-v_{min})^p]<\infty$, there
exists $c>0$ such that $\inf_{[0,T]} \E[\lambda(V_t)]\geq c$. To this end, we consider
$M$ such that a.s., $V_t\in(v_{min},M]$ for all $t\in [0,T]$, we set $K=\max_{[v_{min},M]} \lambda$ and we
introduce
\[
\Omega_T=\{V_0>\alpha\; \hbox{ and } \;\pi([0,T]\times [0,K])=0\}.
\] 
By assumption, $\Pro(\Omega_T)=e^{-KT}f_0((\alpha,\infty))>0$. And on $\Omega_T$, 
$\int_0^T \int_0^\infty V_{\sm}\indiq_{\{u\leq \lambda(V_\sm)\}}\pi(\dd s,\dd u)=0$, since $\lambda(V_\sm)\leq K$ a.s. 
for all $s\in (0,T]$. Consequently, still on $\Omega_T$, 
$$
V_t=V_0+\intot F(V_s)\dd s + w\Gamma_t((\E[\lambda(V_s)])_{s\geq 0}) \geq V_0+\intot F(V_s)\dd s
$$
for all $t\in [0,T]$. Since $V_0>\alpha$ and since $F(\alpha)\geq 0$ by (S2), we conclude that, on $\Omega_T$,
$\inf_{[0,T]} V_t >\alpha$ a.s. The conclusion follows, since $\lambda$ is continuous, increasing and 
strictly positive on $(\alpha,\infty)$: 
$\inf_{[0,T]}\E[\lambda(V_t)] \geq \E[\indiq_{\Omega_T} \lambda(\inf_{[0,T]} V_t)]>0$.

We now prove uniqueness. For two $[v_{min},\infty)$-valued 
solutions $(V_t)_{t\geq 0}$ and $(\tV_t)_{t\geq 0}$  to \eqref{nsde} such 
that  $\E[\sup_{[0,T]} ((V_t-v_{min})^p+(\tV_t-v_{min})^p]<\infty$, we consider $M>0$ such that both $V_t$ 
and $\tV_t$ a.s. belong to $[v_{min},M]$ for all $t\in [0,T]$ and $c>0$ such that both $\E[\lambda(V_t)]\geq c$ and 
$\E[\lambda(\tV_t)]\geq c$ for all $t\in [0,T]$. 
We write $\E[|V_t-\tV_t|]\leq I^1_t+I^2_t+I^3_t$, where 
\begin{gather*}
I^1_t=\E\Big[\intot |F(V_s)-{F}(\tV_s)|\dd s\Big],\qquad
I^2_t= w|\Gamma_t((\E[\lambda(V_s)])_{s\geq 0})  - \Gamma_t((\E[\lambda(\tV_s)])_{s\geq 0})|,\\
I^3_t=\E\Big[\intot\int_0^\infty |(v_{min}-V_\sm)\indiq_{\{u\leq \lambda(V_\sm)\}}-(v_{min}-\tV_\sm)
\indiq_{\{u\leq \lambda(\tV_\sm)\}}|\pi(\dd s,\dd u)\Big].
\end{gather*}
Since $F$ is globally Lipschitz on $[v_{min},M]$ by (S1), 
$I^1_t\leq C \intot\E[|V_s-\tV_s|]\dd s$.
By Lemma~\ref{prelg},  $I^2_t\leq C\intot |\E[\lambda(V_s)]^{1/2}-\E[\lambda(\tV_s)]^{1/2}| \dd s
\leq C \intot |\E[\lambda(V_s)]-\E[\lambda(\tV_s)]| \dd s$, since $x\mapsto \sqrt x$ is Lipschitz continuous
on $[c,\infty)$. Hence $I^2_t\leq C\intot\E[|V_s-\tV_s|]\dd s$, since $\lambda$ is globally Lipschitz on 
$[v_{min},M]$ by (S2).
Finally, for all $t\in [0,T]$,
\begin{align*}
I^3_t \leq & \intot \E\Big[ (\lambda(V_s) \land \lambda(\tV_s))|V_s-\tV_s| 
+ |\lambda(V_s)-\lambda(\tV_s)|((V_s-v_{min})+(\tV_s-v_{min}))   \Big] \dd s\\
\leq & C \intot\E[|V_s-\tV_s|]\dd s,
\end{align*}
since $\lambda$ is bounded and globally Lipschitz on $[v_{min},M]$. All in all,
$\E[|V_t-\tV_t|]\leq C \intot\E[|V_s-\tV_s|]\dd s$, and pathwise uniqueness follows from the Gronwall 
lemma.

To prove existence, we fix $K>0$
and we set $\lambda_K=\lambda(\cdot \land K)$. Using a Picard iteration,
it is not too difficult to prove existence of a $[v_{min}, \infty)$-valued and bounded (by a deterministic constant)
solution $(V^K_t)_{t\in [0,T]}$ to 
$$
V^{K}_t=V_0+\intot F(V^K_s)\dd s + \Gamma_t((\E[\lambda_K(V_s^K)])_{s\geq 0})) +
\intot\int_0^\infty (v_{min}-V^K_\sm)\indiq_{\{u\leq \lambda_K(V_\sm^K)\}}\pi(\dd s,\dd u).
$$
The main steps are as follows: we set $V^{K,0}_t=V_0$ for all $t\in[0,T]$ and, for $k\geq 0$, and 
$t\in[0,T]$, we consider the unique solution to 
$V^{K,k+1}_t=V_0+\intot F(V^{K,k+1}_s)\dd s + \Gamma_t((\E[\lambda_K(V_s^{K,k})])_{s\geq 0})) +
\intot\int_0^\infty (v_{min}-V^{K,k+1}_\sm)\indiq_{\{u\leq \lambda_K(V_\sm^{K,k+1})\}}\pi(\dd s,\dd u)$.
Using that $V_0$ is bounded, that $\lambda_K$ is bounded, that $F(v_{min})\geq 0$ 
and that $F(v)\leq C(1+(v-v_{min}))$,
one easily verifies that $V^{K,k}_t\geq v_{min}$ a.s. for all $t\in [0,T]$ and all $k\geq 0$ and that,
for some constant $M_K>0$, $\sup_{[0,T]}\sup_{k\geq 0} V^{K,k}_t \leq M_K$ a.s.
Then, one easily deduces, as a few lines above, that there is $c_K>0$ such that
$\inf_{[0,T]}\inf_{k\geq 0} \E[\lambda(V^{K,k}_t)] \geq c_K$. The conclusion then follows by classical
arguments using the same computation as in the proof of uniqueness.

We next prove that there is $M>0$ such that a.s., $\sup_{K\geq 1} \sup_{[0,T]} V^K_t\leq M$.
We start from 
\begin{align*}
\E[V^K_t]=&\E[V_0]+ \intot \E[F(V^K_s)]\dd s + \Gamma_t((\E[\lambda_K(V^K_s)])_{s\geq 0})
+\intot \E[(v_{min}-V^K_s)\lambda_K(V^K_s)] \dd s \\
\leq & \E[V_0]+  \intot \Big( C \E[1+(V^K_s-v_{min})] + C \sqrt{\E[\lambda_K(V^K_s)]} +  
\E[(v_{min}-V^K_s)\lambda_K(V^K_s)]\Big) \dd s.
\end{align*}
We used that $F(v)\leq C(1+(v-v_{min}))$ and Lemma~\ref{prelg}.
But, the value of $C$ (not depending on $K$) being allowed to vary, 
\begin{align*}
C \sqrt{\E[\lambda_K(V^K_s)]} + \E[(v_{min}-V^K_s)\lambda_K(V^K_s)] \leq& 
C + C \E[\lambda_K(V^K_s)] - \E[(V^K_s-v_{min})\lambda_K(V^K_s)]\\ 
\leq& C - \E[\lambda_K(V^K_s)].
\end{align*}
For the last inequality, it suffices to note that there is a constant $A>0$ (still denoted by $C$)
such that $\phi_k(v)=C\lambda_K(v)-(v-v_{min})\lambda_K(v) \leq A-\lambda_K(v)$ 
for all $v\in[v_{min},\infty)$. Indeed, $\varphi_K(v)=\phi_K(v)+\lambda_K(v)=(v_{min}+C+1-v)\lambda_K(v)$
is bounded from above on $[v_{min},\infty)$, because 
$\varphi_K(v)\leq 0$ if $v\geq v_{min}+C+1$ and 
$\varphi_K(v)\leq (v_{min}+C+1)\sup_{[0,v_{min}+C+1]} \lambda$ else.

All in all, we have checked that for all $K\geq 1$, all $t\in [0,T]$,
\begin{equation}\label{tip}
\E[V^K_t] \leq  \E[V_0]+  \intot \Big( C \E[1+(V^K_s-v_{min})] - \E[\lambda_K(V^K_s)]\Big) \dd s.
\end{equation}
In particular, $\E[V^K_t] \leq  \E[V_0]+  \intot C \E[1+(V^K_s-v_{min})] \dd s$, whence 
$\sup_{K\geq 1} \sup_{[0,T]}\E[V^K_t]<\infty$ by the Gronwall lemma. But then, we use \eqref{tip} again to write
\[
\int_0^T \E[\lambda_K({V^K_s})] \dd s \leq \E[V_0]- \E[V^K_T]+  C \int_0^T \E[1+(V^K_s-v_{min})]\dd s
\]
Since $- \E[V^K_T] \leq -v_{min}$, we find that
$\sup_{K\geq 1} \int_0^T \E[\lambda_K(V^K_s)] \dd s <\infty$.
By Lemma~\ref{prelg}, we conclude that $D=\sup_{K\geq 1}\sup_{[0,T]}\Gamma_t((\E[\lambda_K(V^K_s)])_{s\geq 0}) <\infty$.
Consequently, for all $K\geq 1$, all $t\in [0,T]$, we have 
$V^K_t \leq V_0 + C \intot (1+(V^K_s-v_{min})) \dd s + D$ (because $F(v)\leq C(1+(v-v_{min}))$).
Using that $V_0$ is bounded and the Gronwall lemma, we deduce that 
there is a deterministic constant $M$ such that a.s., $\sup_{K\geq 1} \sup_{[0,T]} V^K_t\leq M$ as desired.

Finally, we conclude the existence proof: for any $K>M$, we a.s. have $\lambda_K(V^K_t)=\lambda(V^K_t)$
on $[0,T]$, so that $(V^K_t)_{t\in [0,T]}$ solves \eqref{nsde}. 
Furthermore,  $(V^K_t)_{t\in [0,T]}$ is $[v_{min},\infty)$-valued and bounded, whence  \emph{a fortiori}
$\E[\sup_{[0,T]}(V^K_t-v_{min})^p]<\infty$.
\end{proof}

\section{On stationary solutions for the limit soft model}\label{id}

The goal of this section is to show, with the help of some numerical computations, that, depending on the
parameters, there may generically be $1$ or $3$ stationary solutions for the limit soft model
(and sometimes $2$ in some critical cases).
In the whole section, we assume that $F(v)=I-v$ for some $I>0$. We also assume for simplicity that
$\theta=0$ (no delay), that $v_{min}=0$ and that $H(0)=\max_{[0,1]}H$, so that
the nonlinear SDE \eqref{nsde} rewrites
\begin{equation}\label{snsde}
V_t=V_0 + \intot \big(I - V_s + \sqrt{\gamma \E[\lambda(V_t)]} \big) \dd s
- \intot \int_0^\infty V_{s-}\indiq_{\{u\leq \lambda(V_\sm)\}}\pi(\dd s,\dd u),
\end{equation} 
where $\gamma=2\rho H(0)w^2>0$.
Finally, although such an explicit form is necessary only at the end of the section, 
we assume that $\lambda(v)=(v-\alpha)_+^p$ for some $\alpha > 0$
and some $p\in \N_*$.
Assumptions (S1) and (S2) are satisfied (for a large class of initial conditions) if $I\geq \alpha$, but we may
also study stationary solutions when $I \in (0,\alpha)$.

\begin{defin}\label{did}
We say that $g \in \cP([0,\infty))$ is an invariant distribution for \eqref{snsde} if, setting 
$m=\int_0^\infty \lambda (v) g(\dd v)$ and $a=I+\sqrt{\gamma m}$, the solution $(V^{a}_t)_{t\geq 0}$ to 
\begin{equation}\label{lsde}
V_t^a=V_0 + \intot \big(a - V^a_s ) \dd s
- \intot \int_0^\infty V^a_{s-}\indiq_{\{u\leq \lambda(V^a_\sm)\}}\pi(\dd s,\dd u)
\end{equation}
starting from some $g$-distributed $V_0$ is such that $\cL(V_t^a)=g$ for all $t\geq 0$.
\end{defin}

For $a>0$, we define the constant 
$$
K_a = \int_0^a \frac{1}{v-a}\exp\Big(-\int_0^v \frac{\lambda(x)}{a-x}\dd x\Big) \dd v.
$$
We clearly have $K_a=\infty$ if $a\in (0,\alpha]$, because $\lambda=0$ on $[0,\alpha]$, and
$K_a \in (0,\infty)$ for all $a>\alpha$ (because $\lambda$ is continuous and $\lambda(a)>0$).
As in \cite[Proposition 21]{fl}, we have the following result.

\begin{prop}\label{pid}
Fix $a>0$.
The linear SDE \eqref{lsde} has a pathwise unique solution (for any initial condition $V_0 \geq 0$),
and has a unique invariant probability measure $g_a \in \cP([0,\infty))$, given by
$$
g_a=\delta_a \hbox{ if $a\in(0,\alpha]$ and }
g_a(\dd v)= \frac{1}{K_a (v-a)}\exp\Big(-\int_0^v \frac{\lambda(x)}{a-x}\dd x\Big) \indiq_{\{v\in [0,a)\}}\dd v
\hbox{ if $a>\alpha$.}
$$
Furthermore, we have $\int_0^\infty \lambda (v) g_a(\dd v)=K_a^{-1}$ 
(with the convention that $1/\infty=0$ when $a\in (0,\alpha]$).
\end{prop}

The conditions are slightly different from those of \cite[Proposition 21]{fl} (mainly because $\alpha=0$ there),
but the extension is straightforward. 

\begin{rk} 
(i) $g\in \cP([0,\infty))$
is an invariant distribution for \eqref{snsde} if and only if there is $a >0$ such that $g=g_a$ and
$\varphi_\gamma(a)=I$, where $\varphi_\gamma(a)=a-\sqrt{\gamma/K_a}$.

(ii) For any fixed $\gamma>0$, the function $\varphi_\gamma$ is continuous on $[0,\infty)$, one has
$\varphi_\gamma(0)=0$ and $\lim_{a\to\infty} \varphi_\gamma(a)=\infty$, so that for any $I>0$, \eqref{snsde}
has at least one invariant distribution $g$, which is non-trivial if $I>\alpha$ 
(because $g=g_a$ for some $a\geq \varphi_\gamma(a)=I>\alpha$).
\end{rk}

\begin{proof}
Point (i) follows from Definition~\ref{did} and Proposition~\ref{pid}. Concerning point (ii), let us first
rewrite, using the substitutions $v=au$ and $x=ay$,
$$
K_a= \int_0^1 \frac 1{1-u} \exp\Big(-\int_0^u \frac{\lambda(a y)}{1-y}\dd y\Big) \dd u.
$$
It is then easy to prove that $a\mapsto K_a$ is continuous (and decreasing) on $(\alpha,\infty)$ and that
$\lim_{a \downarrow \alpha} K_a = \int_0^1 (1-u)^{-1} \dd u = \infty$, so that $a \mapsto K_a^{-1/2}$ 
(and thus $\varphi_\gamma$) is 
continuous on $[0,\infty)$. We obviously have $\varphi_\gamma(0)=0$, while $\lim_\infty \varphi_\gamma=\infty$
follows from the fact that $K_a \geq e^{-\lambda(1)}a^{-1}$ 
for all $a\geq 2$. 
Indeed, since $\lambda$ is non-decreasing
and since $\int_0^{1/a} (1-y)^{-1} \dd y \leq \int_0^{1/2} (1-y)^{-1} \dd y =\log 2 \leq 1$, we find
$$
K_a \geq \int_0^1 \exp\Big(-\lambda(a u) \int_0^u \frac{\dd y}{1-y}\Big) \dd u
\geq \int_0^{1/a}\exp(-\lambda(a u)) \dd u \geq \frac{e^{-\lambda(1)}}a
$$
as desired.
\end{proof}

Concerning the uniqueness/non-uniqueness of this invariant distribution, the theoretical computations
seem quite involved and we did not succeed. We thus decided to compute numerically 
$a\mapsto \varphi_\gamma(a)$ in a few situations.

Let us first compute a little, recalling that $\lambda(v)=(v-\alpha)_+^p$ with $p\in \N_*$. Let us define
$g_p(x)=x+x^2/2+\dots+x^p/p$ and observe that $\int_0^z (1-x)^{-1}x^p \dd x = -\log(1-z)-g_p(z)$ for
all $z \in [0,1)$. Separating the cases $u\leq \alpha/a$ and $u>\alpha/a$ and using, in the latter case,
the substitution $x=(ay-\alpha)/(a-\alpha)$, one verifies that, for all $u\in (0,1)$,
$$
\int_0^u \frac{\lambda(a y)}{1-y}\dd y= - (a-\alpha)^p\Big[\log\Big(\frac a{a-\alpha}(1-u)\Big) 
+ g_p\Big(\frac{au-\alpha}{a-\alpha}\Big)\Big]\indiq_{\{u>\alpha/a\}}.
$$
Then, a few computations (using the substitution $z=(au-\alpha)/(a-\alpha)$) show that, for all $a>\alpha$, 
\begin{align}
K_a =& \int_0^{\alpha/a} \frac{\dd u}{1-u} + \int_{\alpha/a}^1 \frac{\dd u}{1-u}
\Big(\frac a{a-\alpha}(1-u)\Big)^{(a-\alpha)^p} \exp\Big((a-\alpha)^p g_p\Big(\frac{au-\alpha}{a-\alpha}\Big) \Big),
\nonumber\\
K_a=& \log \frac{a}{a-\alpha} + \int_0^1 \frac{\dd z}{1-z} (1-z)^{(a-\alpha)^p} \exp((a-\alpha)^p g_p(z)). \label{ex1}
\end{align}
Naive methods to compute $K_a$ numerically do not work well, because with $a=\alpha$ 
(actually, the problem is when $a-\alpha>0$ is very small), one has to approximate 
$\int_0^1 (1-u)^{-1}\dd u$: a Monte-Carlo method with i.i.d. uniform random variables $U_i$ on $[0,1]$
gives $n ^{-1}\sum_1^n(1-U_i)^{-1}\simeq 15$ when $n=10^6$, while a 
Riemann approximation gives $\sum_1^n n^{-1}(1-(i/n))^{-1}\simeq 14.39$ with $n=10^6$. 
Both values are very far from the true one, which is $\infty$. One possibility is to proceed to the
substitution $z=1-e^{-r/(a-\alpha)^p}$, which gives
\begin{equation}\label{ex2}
K_a=\log \frac{a}{a-\alpha} + \frac 1{(a-\alpha)^{p}}
\int_0^\infty \exp(-r+(a-\alpha)^pg_p(1-e^{-r/(a-\alpha)^p})) \dd r.
\end{equation}
But this expression has other defaults.
The numerical computations below use a Monte-Carlo method based on 
\eqref{ex2} (with exponential random variables with parameter $1$) when $a\in (\alpha,\alpha+1)$ and 
based on \eqref{ex1} 
(with uniform random variables on $[0,1]$) when $a\geq \alpha+1$.

\begin{figure}[h]
\noindent\fbox{
\begin{minipage}{0.9\textwidth}
\centerline{\includegraphics[width=0.5\textwidth]{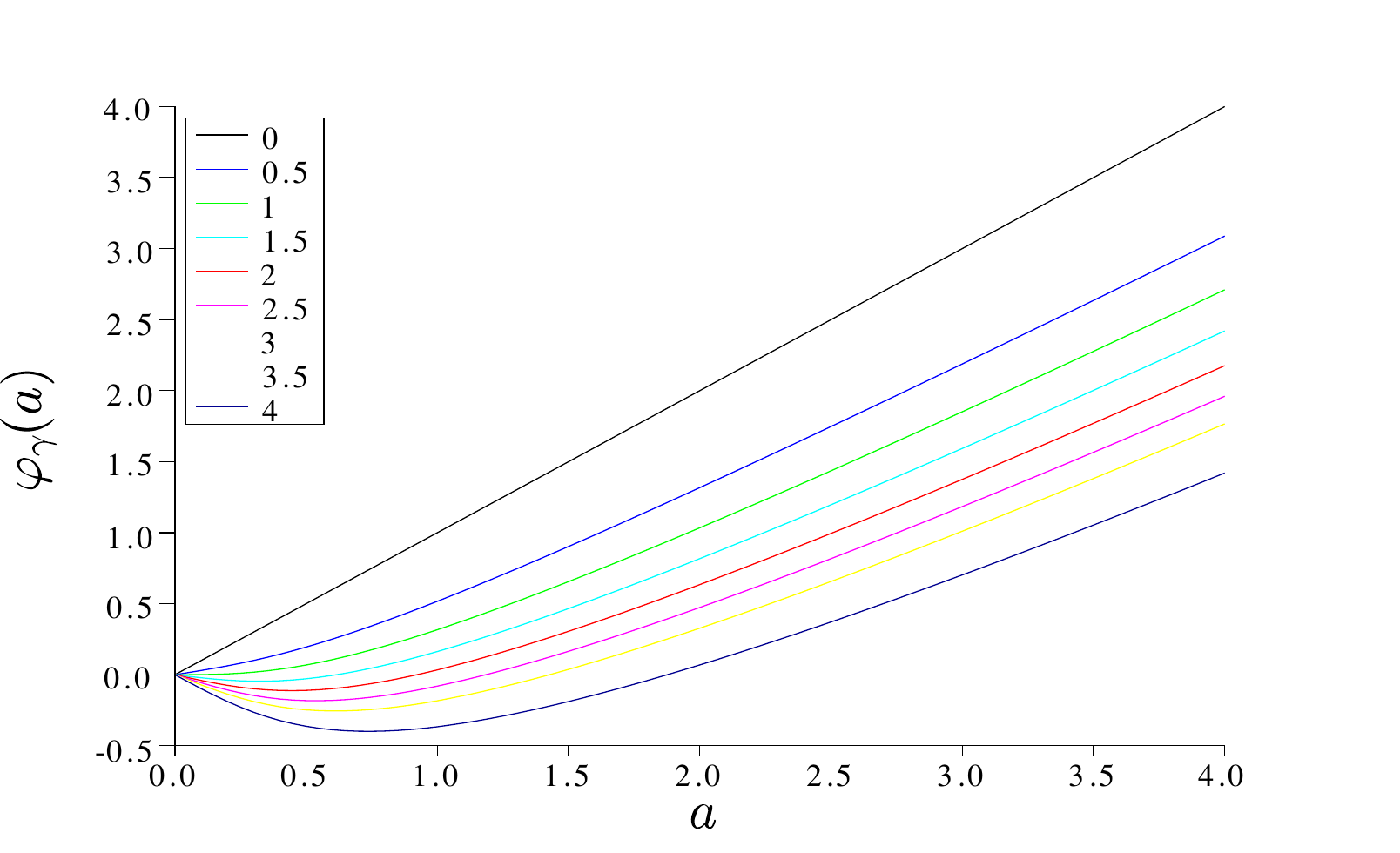} 
\includegraphics[width=0.5\textwidth]{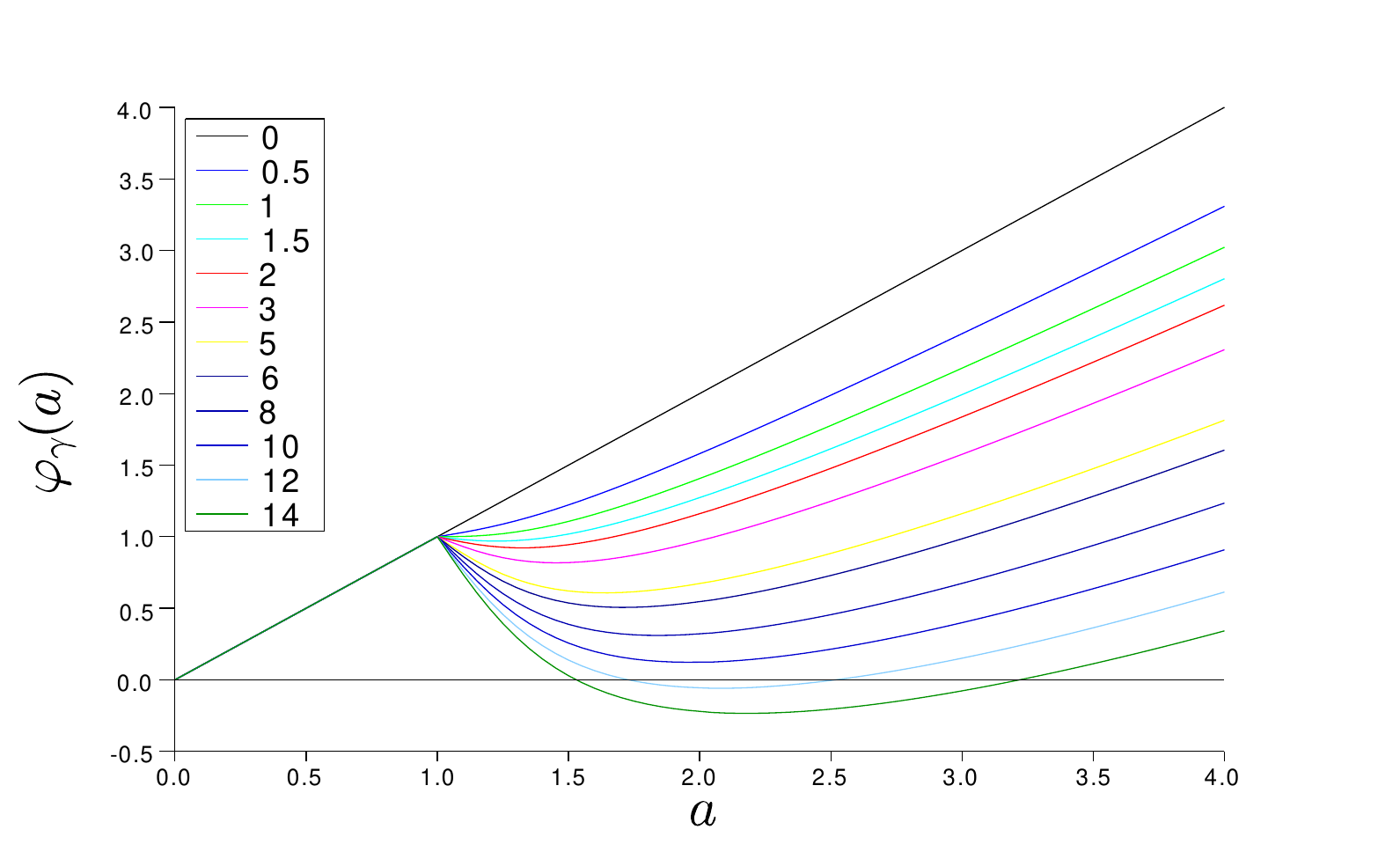}}
\centerline{\includegraphics[width=0.5\textwidth]{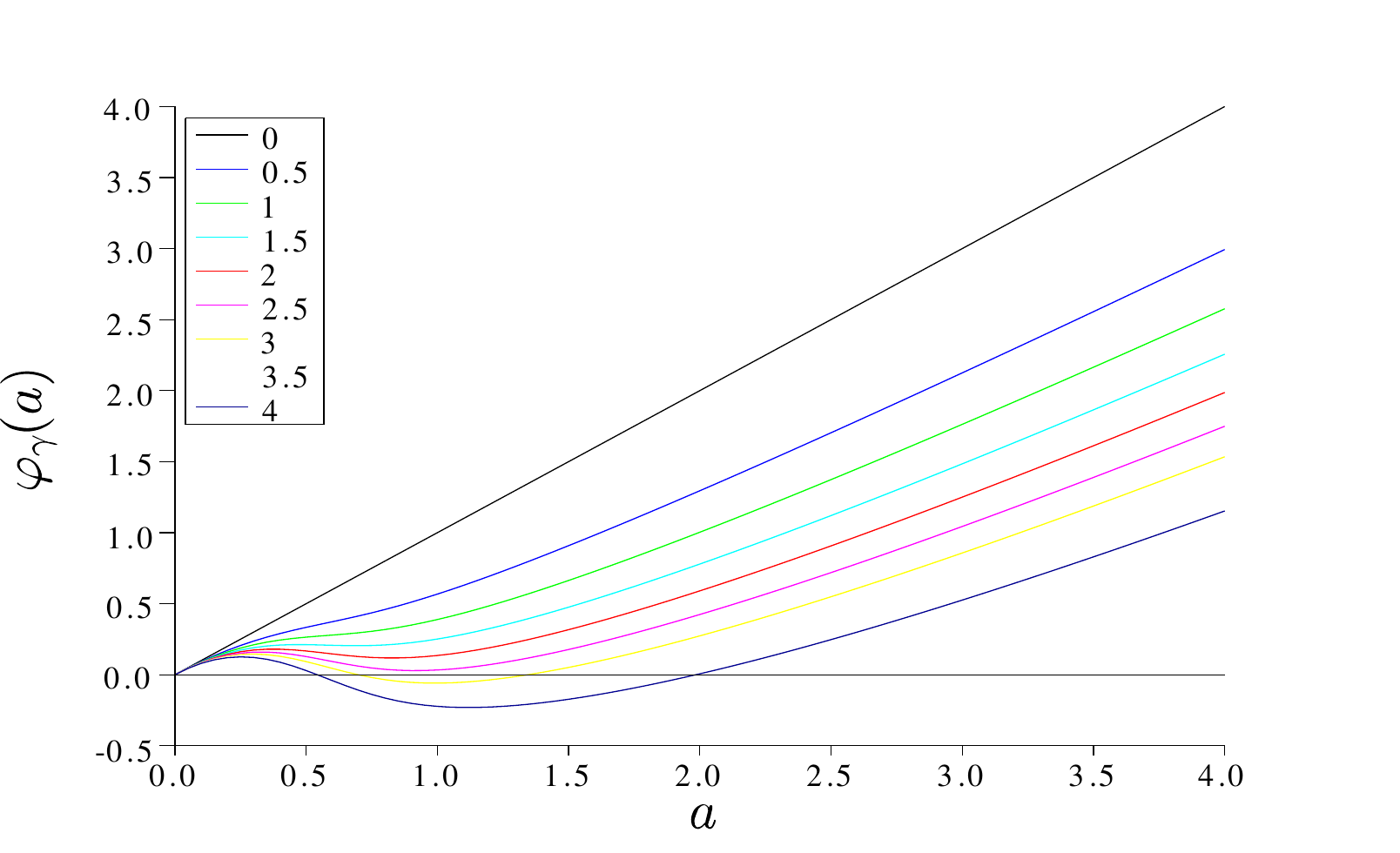} 
\includegraphics[width=0.5\textwidth]{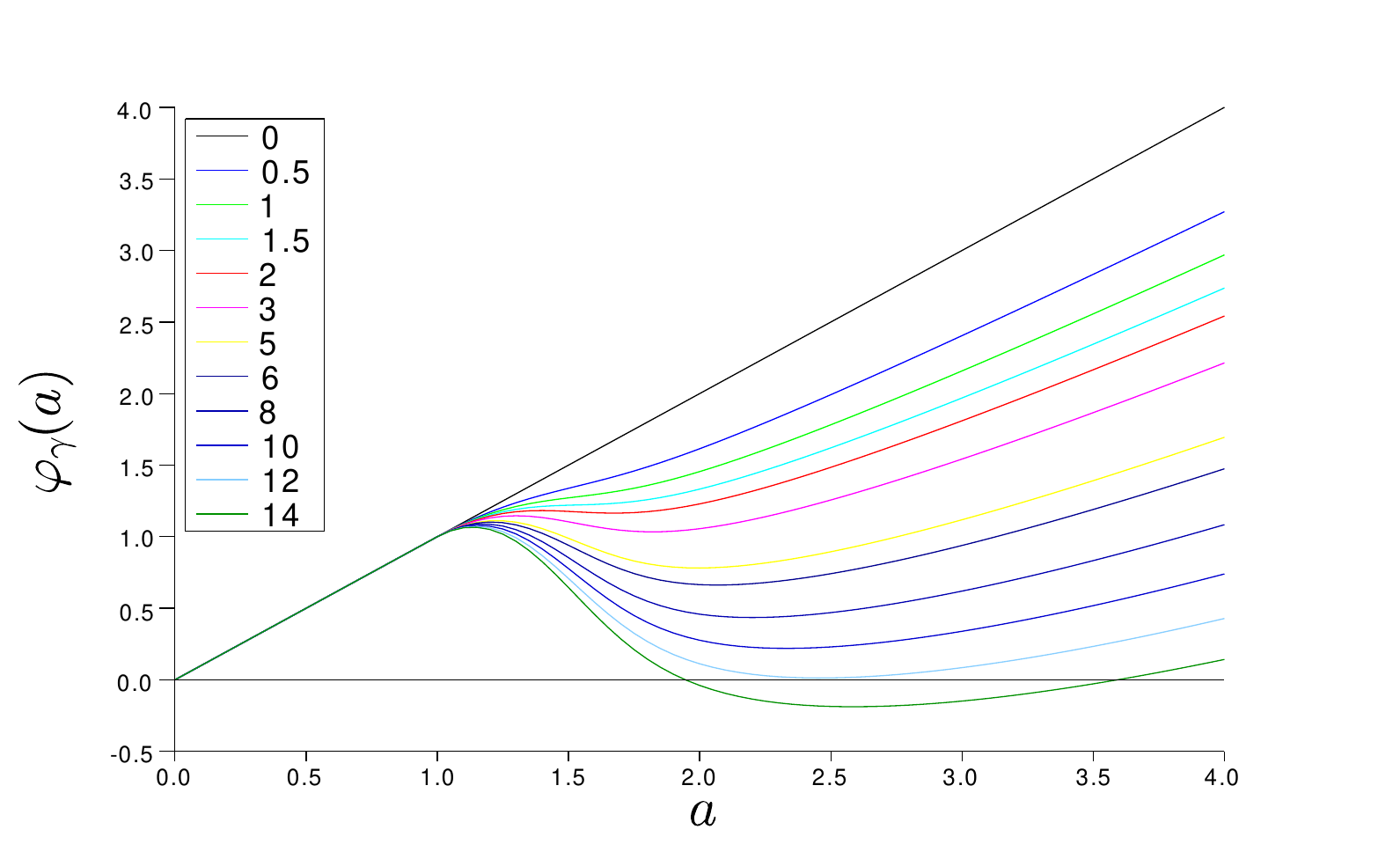} 
}\end{minipage}}
\caption{\small{Plots of $a\mapsto\varphi_\gamma(a)$ for different values of $\gamma$, when
$\lambda(v)=v^2$ (up left),  $\lambda(v)=(v-1)_+^2$ (up right),
$\lambda(v)=v^4$ (down left),  $\lambda(v)=(v-1)_+^4$ (down right),\label{dessid}}}
\end{figure}

Let us comment on Figure~\ref{dessid}. Recall that
for $\gamma>0$ and $I>0$, each stationary solution to \eqref{snsde} corresponds to one solution
$a$ to $\varphi_\gamma(a)=1$. 

$\bullet$ If $\lambda(v)=v^2$, for any $\gamma>0$, the equation $\varphi_\gamma(a)=I$ (with $I>0$ fixed)
seems to have exactly one solution, for all $I>0$.

$\bullet$ If $\lambda(v)=(v-1)_+^2$, $\lambda(v)=v^4$ or $\lambda(v)=(v-1)_+^4$, it seems that there are 
$0<\gamma_1<\gamma_2$ (e.g., $\gamma_1\simeq 1.5$ and $\gamma_2\simeq 12$ when $\lambda(v)=(v-1)_+^4$) such that

(a) if $\gamma \in (0,\gamma_1)$, then for all $I>0$, $\varphi_\gamma(a)=I$ has exactly one solution,

(b) if $\gamma \in (\gamma_1,\gamma_2)$, then there are $0<J_\gamma<I_\gamma$ such that,
if $I \in (0,J_\gamma)$, $\varphi_\gamma(a)=I$ has exactly one solution, 
if $I \in (J_\gamma,I_\gamma)$, $\varphi_\gamma(a)=I$ has exactly three solutions and
if $I>I_\gamma$, $\varphi_\gamma(a)=I$ has exactly one solution,

(c) if  $\gamma \in (\gamma_2,\infty)$, then there is $I_\gamma>0$ such that for all $I \in (0,I_\gamma)$,
$\varphi_\gamma(a)=I$ has exactly three solutions and
if $I>I_\gamma$, $\varphi_\gamma(a)=I$ has exactly one solution (which is nontrivial).

\section{Simulations}\label{num}

In all the simulations below, we have chosen the following values: the minimum potential is $v_{min}=0$, 
the length of the dendrites is $L=1$, the repartition density is $H(x)=2(1-x)$ on $[0,1]$, 
the front velocity is $\rho=1$ and the excitation parameter is $w=1$.
Concerning the particle systems presented in Subsection~\ref{mm}, 
we consider a fully mean-field interaction, i.e. $p_n=1$ and $N=n$.

{
The code we use to simulate the {\it soft} particles system presented in Subsection~\ref{mm}
relies on a rejection method. The only difficulty concerns the treatment of the dendrites, that
we need to incrementally update with new fronts. This is based on the recent algorithm of
Yakupov and Buzdalov \cite{ya}.
}

\subsection{An isolated dendrite with i.i.d. impulses}\label{iso}

We will observe in the next subsections a small temporal shift between the particle system and its mean-field
limit. To explain this phenomenon, we 
consider a single dendrite with length $1$, on which two fronts start from 
each $X_i$ at time $T_i$ (for $i=1,\dots,n$), where the family $(T_i,X_i)_{i\geq 1}$ is i.i.d. with density
$\indiq_{\{t\in[0,1]\}}\dd t H(x)\dd x$. The situation is thus very simple and, as seen in
Proposition~\ref{tac}, $A_t(\sum_{i=1}^n\delta_{(T_i,X_i)})$ represents the number of fronts hitting the soma of the 
dendrite before time $t$.
By Lemma~\ref{lcru}, $Y^n_t=n^{-1/2} A_t(\sum_{i=1}^n\delta_{(T_i,X_i)})$ 
goes to $y_t=\Gamma_t((\indiq_{\{s\in[0,1]\}})_{s\geq 0})$
as $n\to\infty$. By Remark~\ref{ttip}, we have 
$$
y_t= \sqrt{2\rho H(0)} \intot \indiq_{\{s\in[0,1]\}} \dd s = 2 \min(t,1).
$$
We want to show that there is a systematic bias. So, we fix $K=10000$, we simulate $K$ i.i.d. copies
$(Y^{i,n}_t)_{t \in [0,2]}$ of the process $(Y^{n}_t)_{t \in [0,2]}$, for different values of $n$, namely
$n=10000$, $n=40000$ and $n=80000$.
On Figure~\ref{isod}, we plot, as a function of time $t\in [0,2]$, 
the average values $K^{-1}\sum_{i=1}^K Y^{i,n}_t- y_t$.
We observe a systematic negative bias, which remains important for large values of $n$.
For example at time $1$ we have a bias around $-0.14$ (i.e. $7 \%$)
when $n=10000$ and $-0.08$ (i.e. $4 \%$) when $n=80000$.

\begin{figure}[h]
\noindent\fbox{\begin{minipage}{0.9\textwidth}
\centerline{\includegraphics[width=0.7\textwidth]{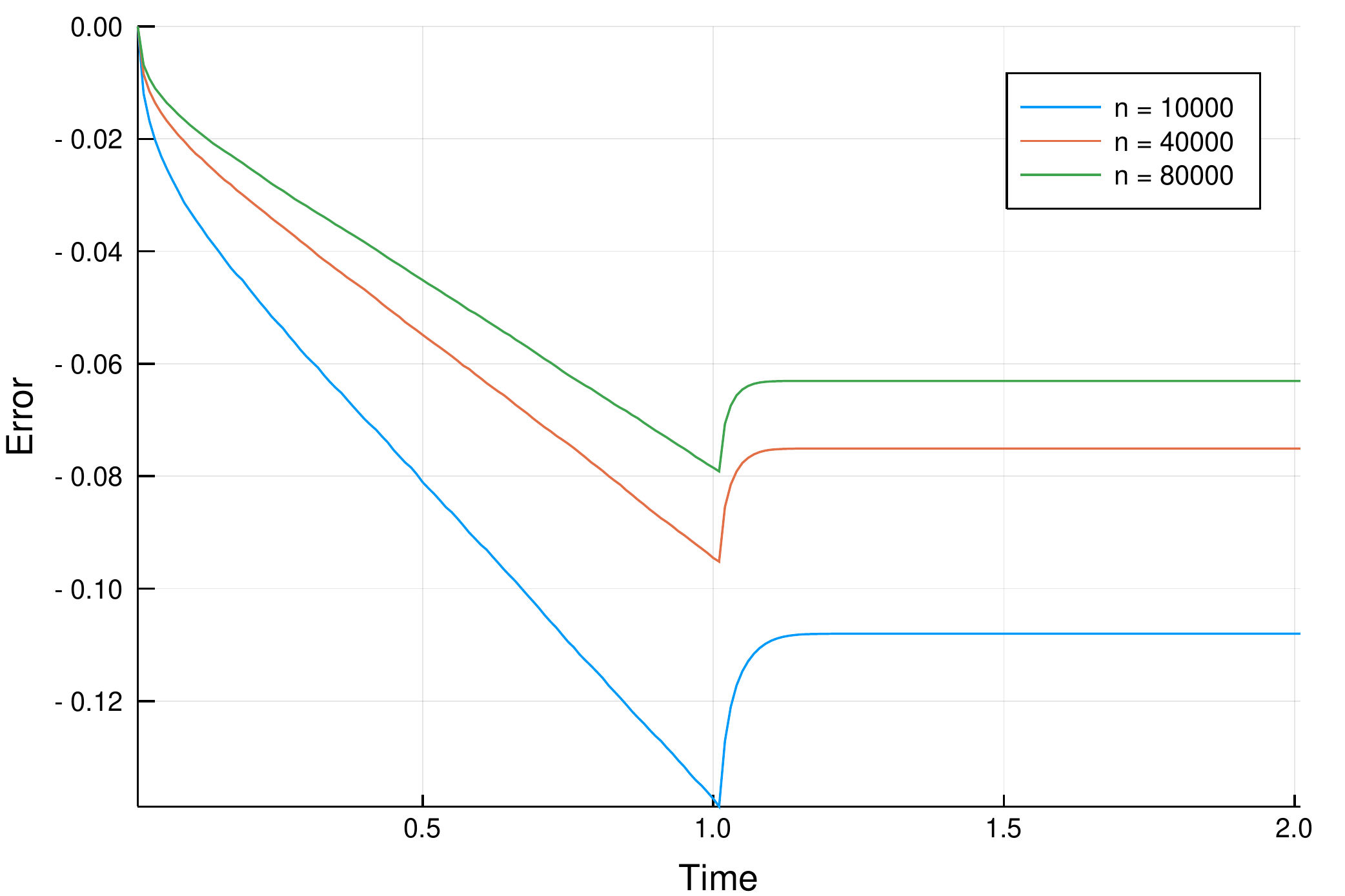} 
}\end{minipage}}
\caption{\small{Isolated dendrite. \label{isod}}}
\end{figure}

We see that a few late fronts arrive after time $1$ (while the limiting value stops increasing at time $1$)
and this slightly makes decrease the bias.

\subsection{The soft model without delay}\label{sss1}

\begin{figure}[h]
\noindent\fbox{\begin{minipage}{0.9\textwidth}
\centerline{\includegraphics[width=0.7\textwidth]{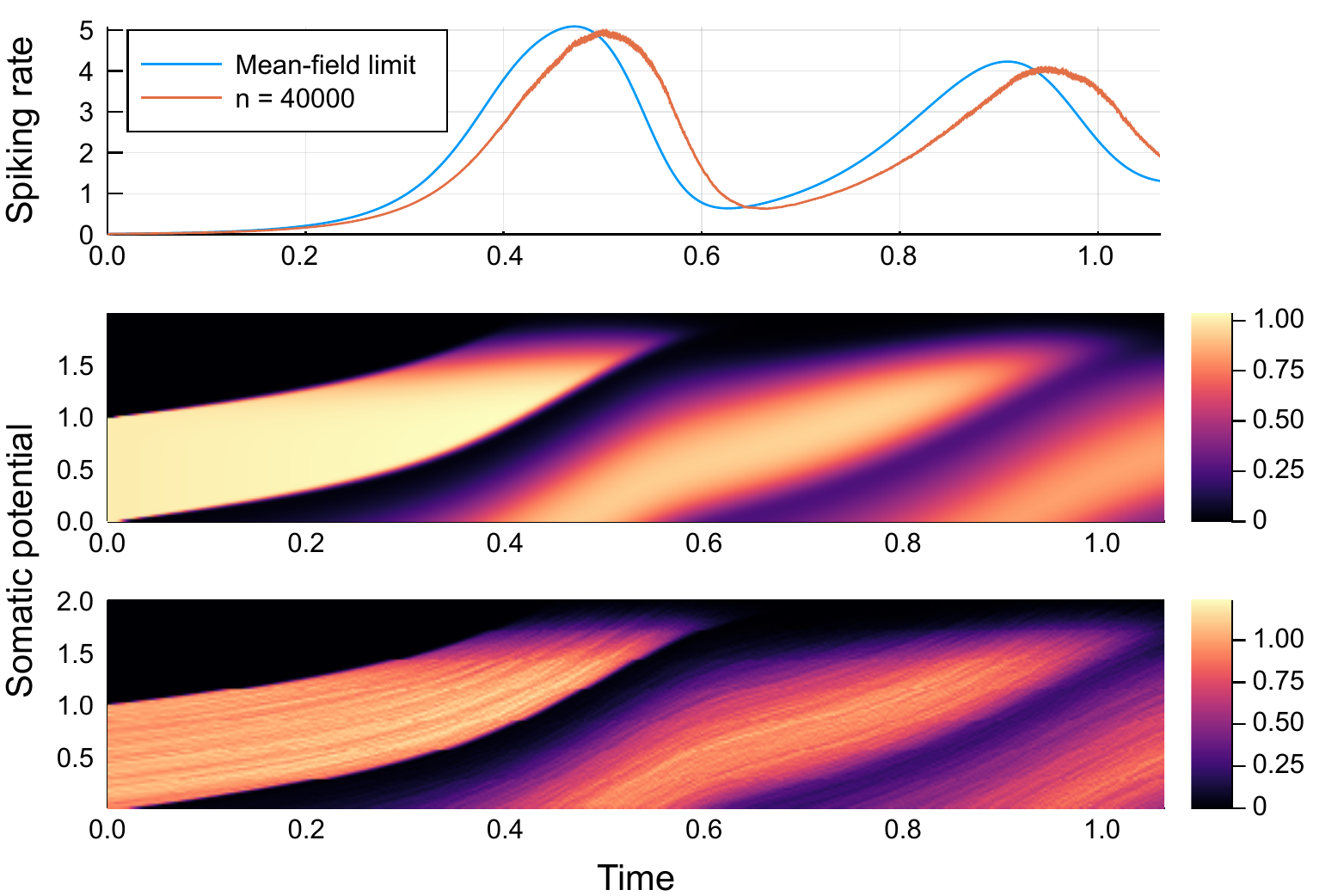} 
\hskip2cm \ref{sim1}.a}\end{minipage}}
\noindent\fbox{\begin{minipage}{0.9\textwidth}
\centerline{\includegraphics[width=0.7\textwidth]{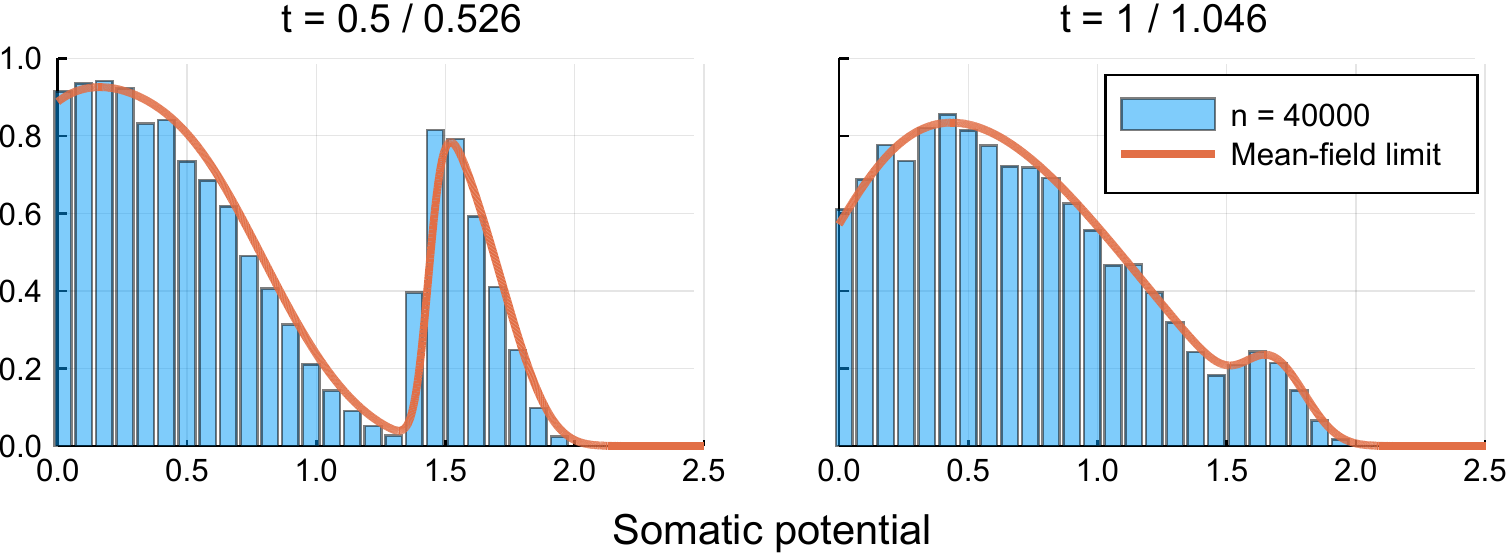} \hskip2cm\ref{sim1}.b}\end{minipage}}
\noindent\fbox{\begin{minipage}{0.9\textwidth}
\centerline{\includegraphics[width=0.7\textwidth]{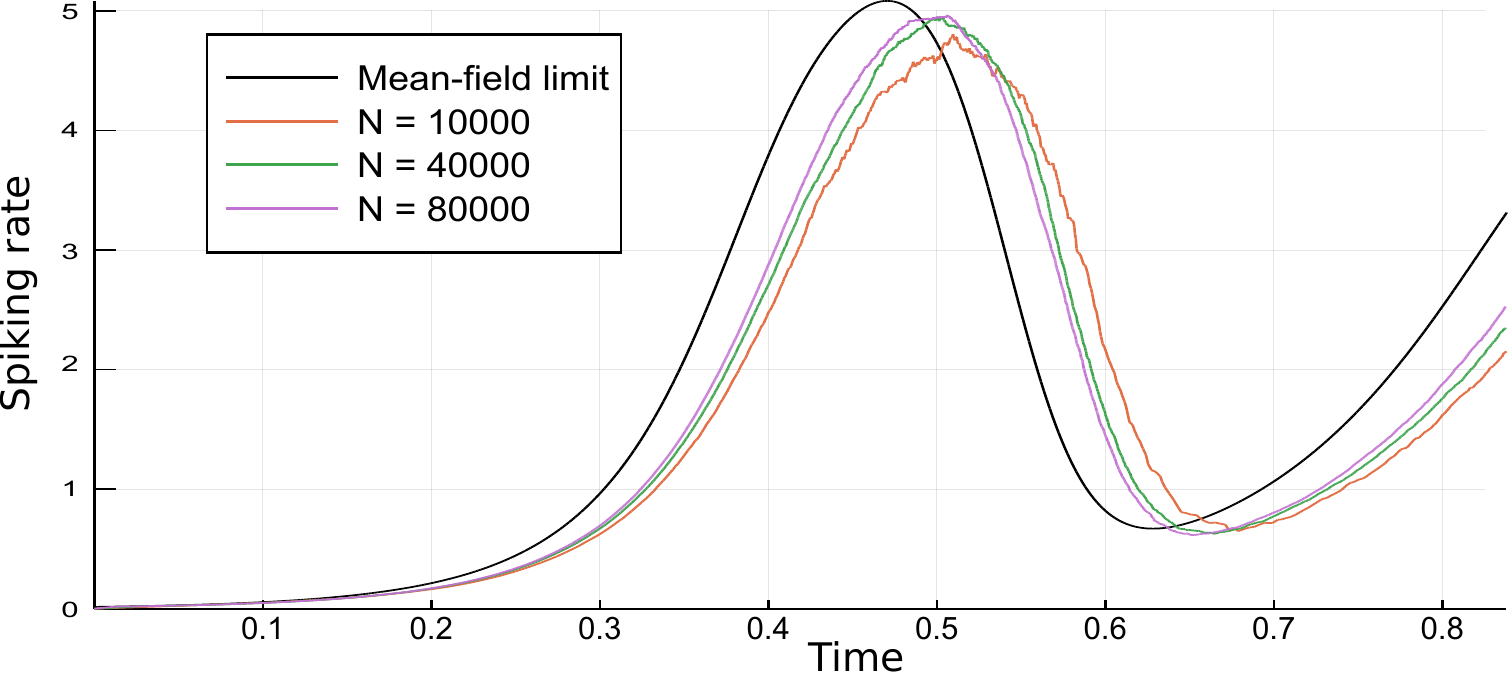}\hskip2cm \ref{sim1}.c}\end{minipage}}
\caption{\small{Soft model, $\theta\!=\!0$, $\lambda(v) \!=\! (v\!-\!0.2)_+^8$, $F(v) \!=\! 1 \!-\! 0.1v$,
{$f_0(v)\!=\indiq_{[0,1]}(v)$}.   \label{sim1}}}
\end{figure}

Here we consider the soft model with the following parameters: the delay is $\theta=0$,
the rate function is $\lambda(v) = \max(v-0.2,0)^8$, the drift function is $F(v)= 1 - 0.1v$
and the initial distribution is $f_0(v)=\indiq_{\{v \in [0,1]\}}$.
On Figure~\ref{sim1}.a, we plot on the first picture the maps $t\mapsto n^{-1}\sum_{i=1}^n \lambda(V^{i,n}_t)$,
for the particle system (soft model) described in Subsection~\ref{mm} with $n=40000$ particles, as well as 
$t\mapsto \E[\lambda(V_t)]$, for $(V_t)_{t\geq 0}$ the unique solution to the nonlinear SDE \eqref{nsde}.
We observe that the two curves are very similar, but there is a small temporal shift. 
This is related to what we explained
in Subsection~\ref{iso}. The second picture represents $(g(t,v))_{t\geq 0,v \geq 0}$,
where $g(t,\cdot)$ is the density of the law of $V_t$. The third picture represents
$(g_n(t,v))_{t\geq 0,v\geq 0}$, still with $n=40000$, where $g_n(t,\cdot)$ is a smooth version of the 
empirical measure $n^{-1}\sum_{i=1}^n \delta_{V^{i,n}_t}$. Here again, the second and third pictures seem rather close,
up to a small temporal shift.
On Figure~\ref{sim1}.b, the first picture represents $v\mapsto g(t,v)$ (with $t=0.5$) and $v\mapsto g_n(t,v)$ 
(with $t=0.526$). So, we took into account the temporal shift to make the histogram fit the continuous curve
as well as possible. The second picture is similar, with $t=1$ and $t=1.046$.
Finally, Figure~\ref{sim1}.c contains a plot of $t\mapsto \E[\lambda(V_t)]$ and of 
$t\mapsto n^{-1}\sum_{i=1}^n \lambda(V^{i,n}_t)$ for different values of $n$. We see that the temporal shift
decreases as $n$ increases, but the convergence seems to be rather slow.

Let us mention that $g(t, v)$ is computed here by solving numerically the PDE associated to the nonlinear SDE 
\eqref{nsde}, using an Euler scheme relying on finite differences in $t$ and in $v$, with a regular grid. 
There is a source term at $v_{min}=0$ involving the integral $\int_0^\infty \lambda(v)g(t,v)\dd v$,
that is incorporated in the spatial finite difference at the extremity $v=0$ of the space-grid.
We take absolute values and normalize at each time step to ensure the 
positivity of the solution and that its total mass equals $1$.
All the figures involving this scheme were compared to a simple interacting 
particle system (see the next subsection)
and we found very similar results.

\subsection{The soft model with delay}\label{sss2}

Here we proceed exactly as in Subsection~\ref{sss1}, with the same parameters, except that the delay $\theta=0.4$.
The results are presented in Figure~\ref{sim2}.
The unpleasant temporal shift is slightly smaller.

Let us mention that we use here a different scheme to approximate the law $g(t,\cdot)$
of $V_t$, based on a simple interacting particle system $(\bar V^{i,K}_t)_{i=1,\dots,K, t\geq 0}$, 
with $K=10^6$ particles. Indeed, the scheme of the previous section was not stable with a nonzero delay. Roughly, each particle solves the same SDE as \eqref{nsde} 
(with i.i.d. initial conditions and driving Poisson
measures), but with the nonlinear term
$\int_0^{(t-\theta)\lor 0} (\gamma \E[\lambda(V_s)])^{1/2}\dd s$ replaced by its empirical version
$\int_0^{(t-\theta)\lor 0} (\gamma K^{-1} \sum_{i=1}^K \lambda(\bar V_s^{i,K}))^{1/2}\dd s$. Of course, we also have 
to proceed to a time discretization.

\begin{figure}[h]
\noindent\fbox{\begin{minipage}{0.9\textwidth}
\centerline{\includegraphics[width=0.7\textwidth]{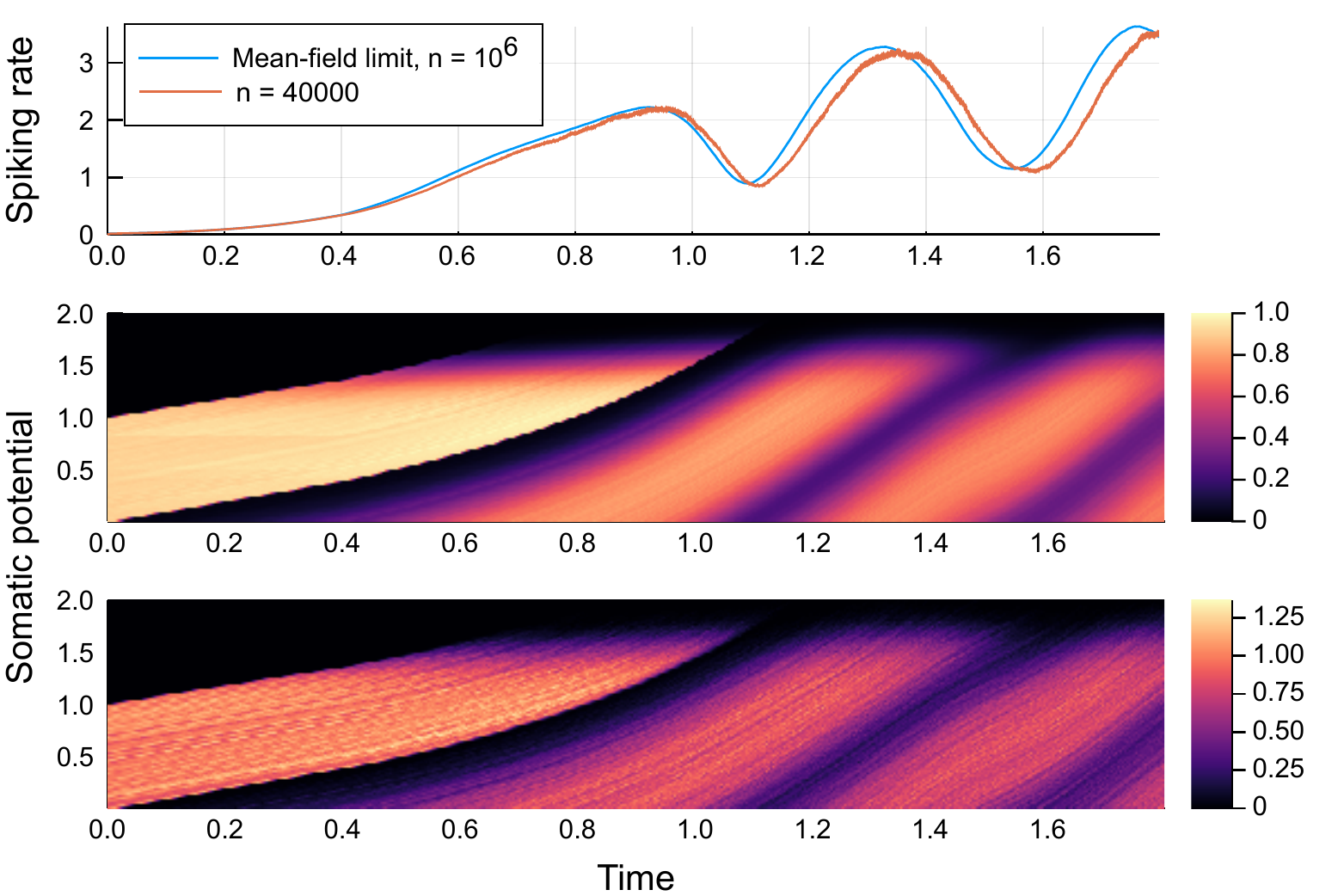} 
\hskip2cm \ref{sim2}.a}\end{minipage}}
\noindent\fbox{\begin{minipage}{0.9\textwidth}
\centerline{\includegraphics[width=0.7\textwidth]{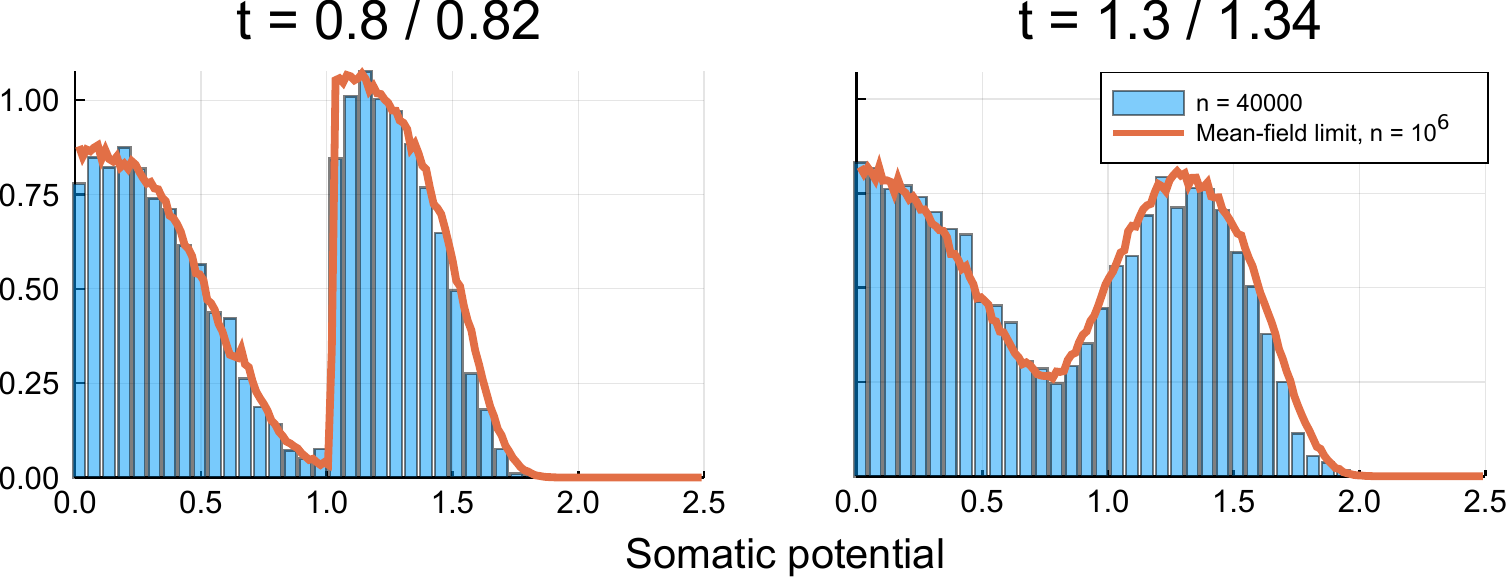} 
\hskip2cm\ref{sim2}.b}\end{minipage}}
\noindent\fbox{\begin{minipage}{0.9\textwidth}
\centerline{\includegraphics[width=0.7\textwidth]{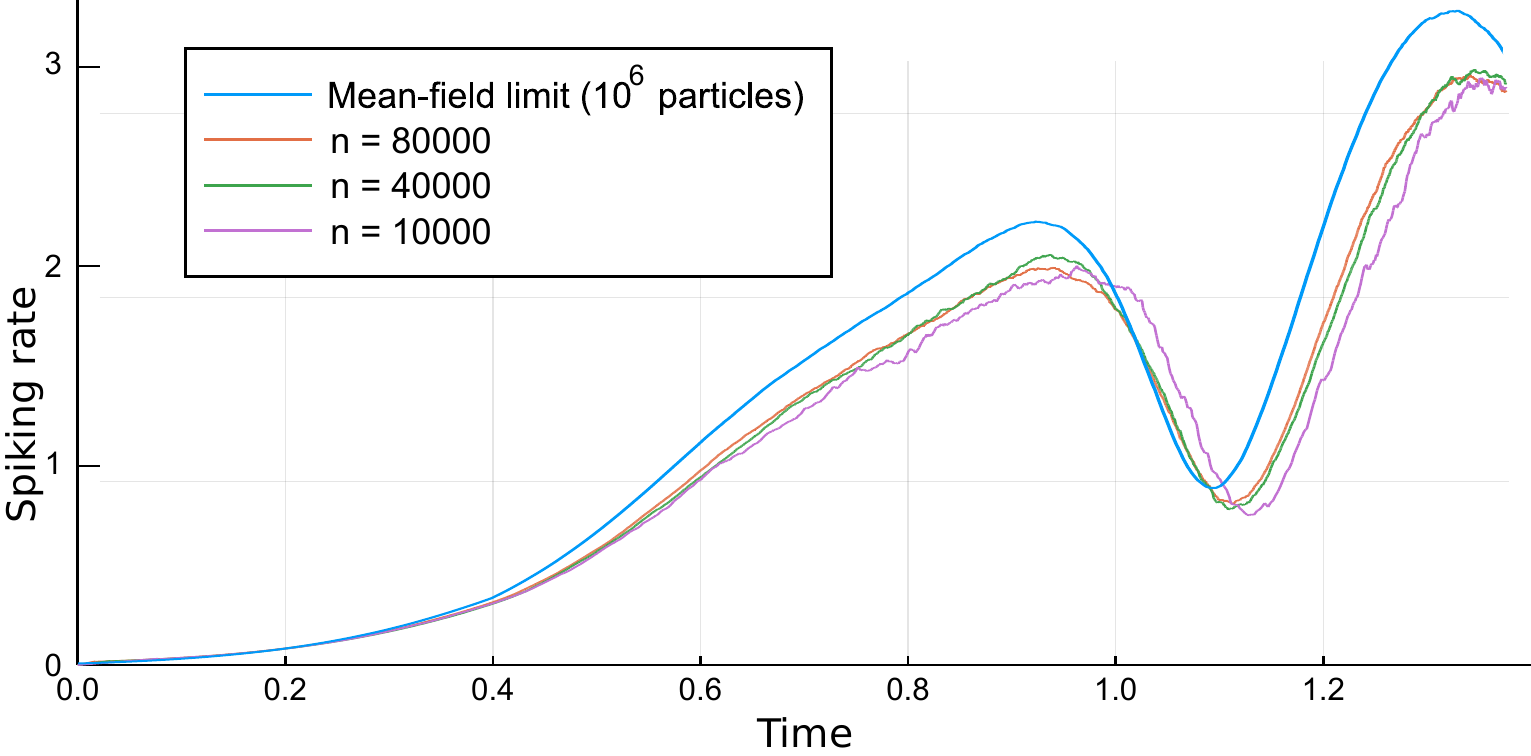}
\hskip2cm \ref{sim2}.c}\end{minipage}}
\caption{\small{Soft model, $\theta\!=\!0.4$, $\lambda(v) \!=\! (v\!-\!0.2)_+^8$, $F(v) \!=\! 1 \!-\! 0.1v$,
{$f_0(v)\!=\indiq_{[0,1]}(v)$}.   \label{sim2}}}
\end{figure}

\subsection{The soft model with another rate function}

Here again, we proceed exactly as in Subsection~\ref{sss1}, with the same parameters 
(in particular $\theta=0$), except that the rate function $\lambda(v)=v^8$ does not satisfy our assumptions,
since $\alpha=\inf\{v\geq v_{min}:\lambda(v)>0\}=v_{min}$ (recall that $v_{min}=0$).
The results are presented in Figure~\ref{sim3} and are not less convincing than those of the previous subsections.
It thus seems that our assumption that $\alpha>v_{min}$ is not necessary.

\begin{figure}[h]
\noindent\fbox{\begin{minipage}{0.9\textwidth}
\centerline{\includegraphics[width=0.7\textwidth]{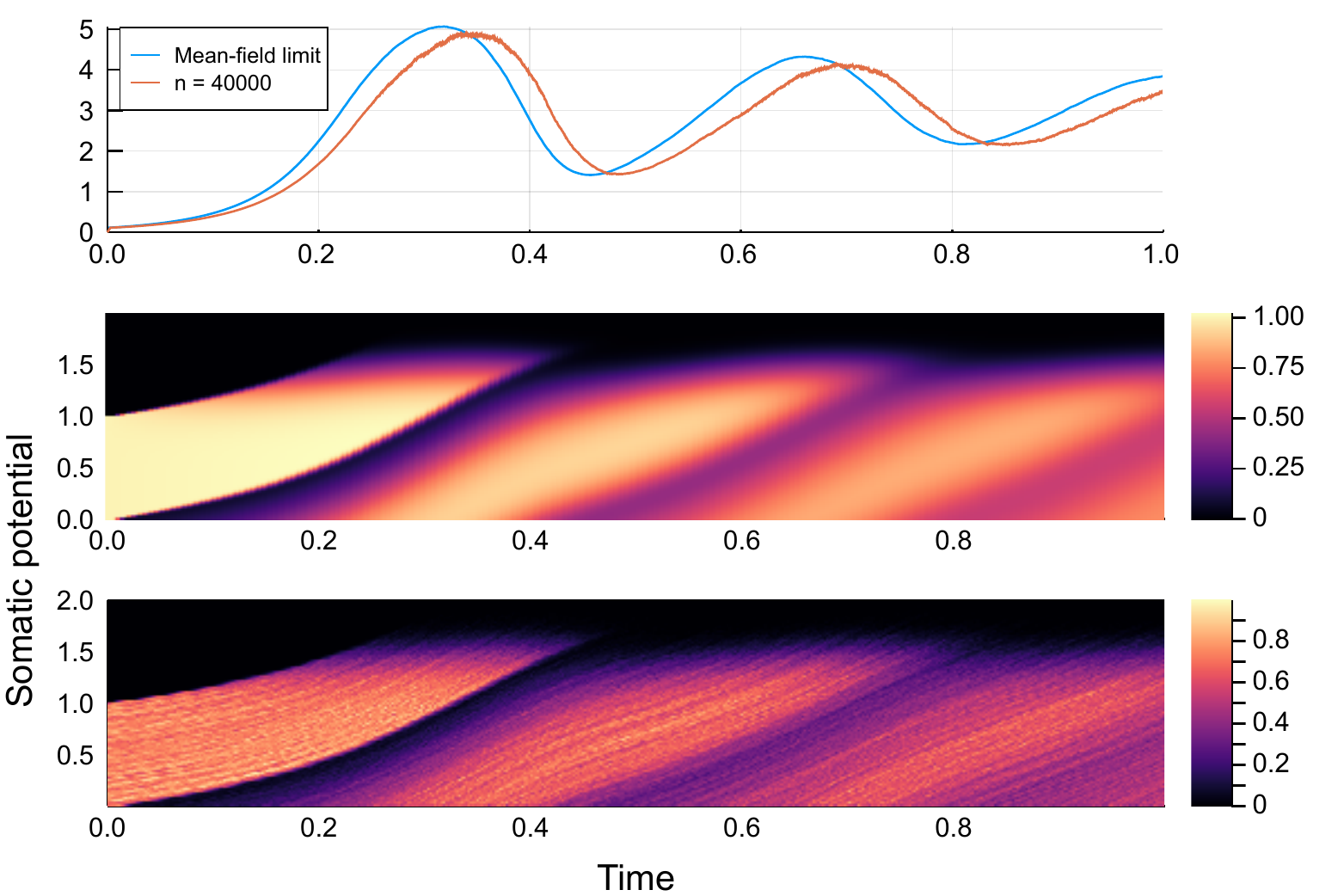} 
\hskip2cm \ref{sim3}.a}\end{minipage}}
\noindent\fbox{\begin{minipage}{0.9\textwidth}
\centerline{\includegraphics[width=0.7\textwidth]{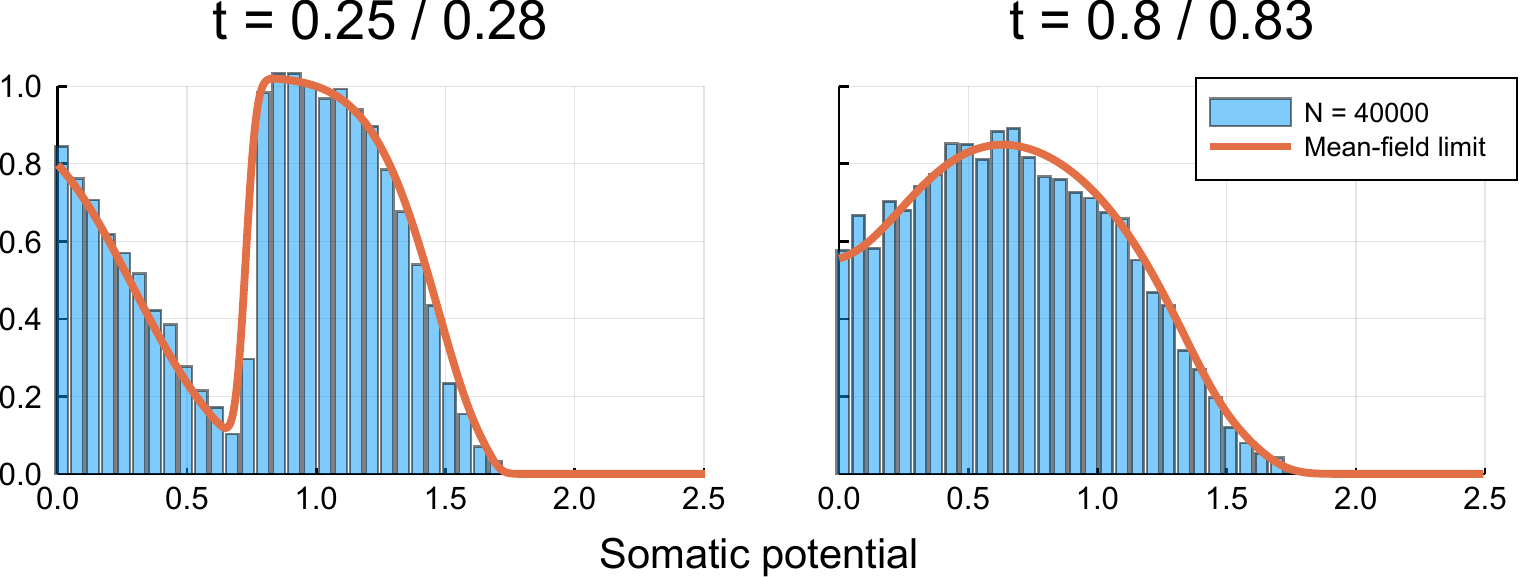} 
\hskip2cm\ref{sim3}.b}\end{minipage}}
\noindent\fbox{\begin{minipage}{0.9\textwidth}
\centerline{\includegraphics[width=0.7\textwidth]{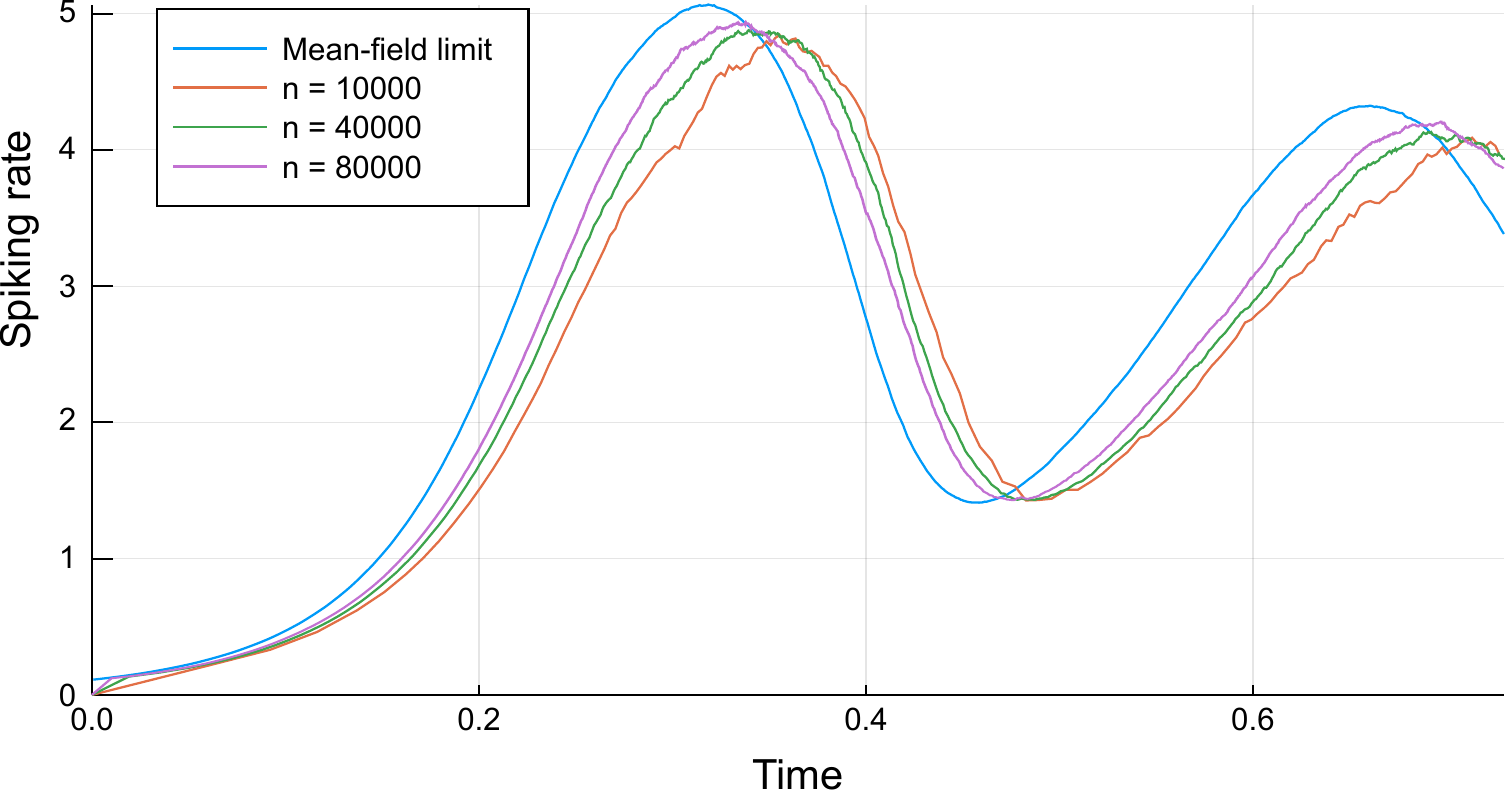}
\hskip2cm \ref{sim3}.c}\end{minipage}}
\caption{\small{Soft model, $\theta\!=\!0$, $\lambda(v) \!=\! v^8$, $F(v) \!=\! 1 \!-\! 0.1v$,
{$f_0(v)\!=\indiq_{[0,1]}(v)$}.   \label{sim3}}}
\end{figure}

\subsection{The hard model}

Concerning the hard model, we did not code the particle system described in Subsection~\ref{mm}.
However, we would like to validate numerically the explicit formula of Theorem~\ref{mr1}.
We consider the following set of parameters: $F(v)=I=0.5$, $\theta=0$, $v_{min}=0$, $v_{max}=1.2$, and
$$f_0(v)=\Big[\frac 1{2v_{max}}+\frac{\pi}{4v_{max}}\sin\frac {\pi v}{v_{max}}\Big]\indiq_{[0,v_{max}]}(v).$$

We compute numerically $(\kappa_t)_{t \geq 0}$ by solving the ODE
$\kappa'_t=G_0(v_{max}-It-\kappa_t)$ (with $\kappa_0=0$), using an Euler scheme, until time $a>0$ such that
$\kappa_a+Ia=v_{max}$ and by using that $\kappa'$ is $a$-periodic, see the proof of Theorem~\ref{mr1}.
On Figure~\ref{sim4}, we plot in red the curve $t\mapsto \kappa'_t$. Recall that
$\kappa'_t$ represents the excitation rate, i.e. the increase of potential of the neurons, during $[t,t+\dd t]$,
due to excitation.

Next, the hard model can be seen as the soft model with the choice
$\lambda(v)=\infty\indiq_{\{v>v_{max}\}}$, that we approximate by $\lambda(v)=\max(v-0.2,0)^{300}$.
{\color{black}
We then the mean-field particle system introduced in Subsection~\ref{sss2}) with $K=200000$ particles,
to approximate numerically
$t\mapsto \E[\lambda(V_t)]$, $(V_t)_{t\geq 0}$ being the solution to the nonlinear SDE \eqref{nsde}.
And we plot, in blue, the approximation of $t\mapsto 2\sqrt{\E[\lambda(V_t)]}$}, which also represents the excitation rate,
since it is the derivative of $t\mapsto \Gamma_t((\E[\lambda(V_s)])_{s\geq 0})=2\intot \sqrt{\E[\lambda(V_s)]}\dd s$, 
see Remark~\ref{ttip} and recall that $H(0)=2$.

The two curves are close to each other and this is rather convincing concerning our explicit formula. 
However the precision is not high, which is not surprising due to the (numerical)
singular behavior of $\lambda$ around $v=v_{max}$.

\begin{figure}[h]
\noindent\fbox{\begin{minipage}{0.9\textwidth}
\centerline{\includegraphics[width=0.9\textwidth]{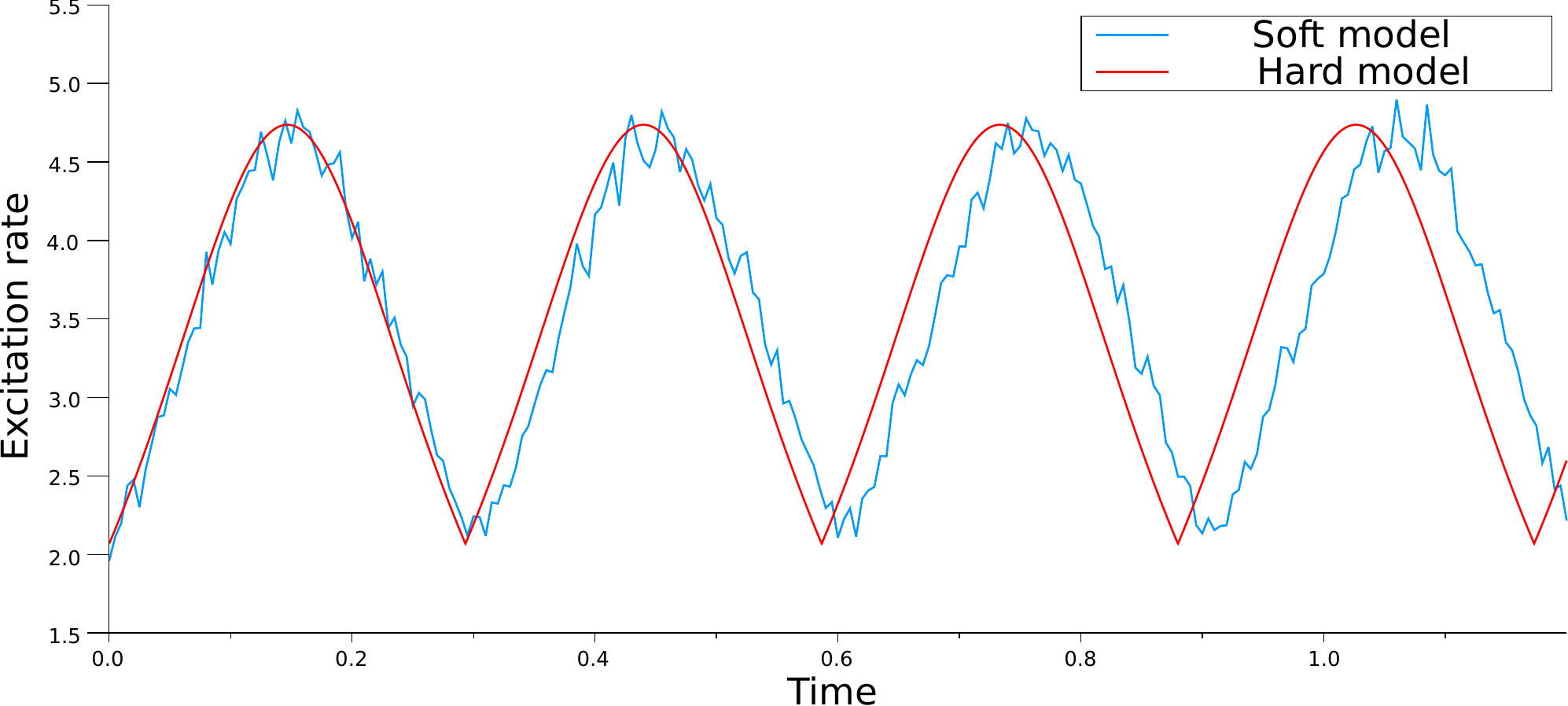}}
\end{minipage}}
\caption{
{\small Soft and hard models with $\theta=0$, $F\equiv 0.5$, $v_{max}=1.2$, 
$\lambda(v) \!=\! (v\!-\!0.2)_+^{300}$}
\centerline{\small{ and
$f_0(v)=[\frac 1{2v_{max}}+\frac{\pi}{4v_{max}}\sin\frac {\pi v}{v_{max}}]{\indiq_{ [0,v_{max}]}(v)}.$}}
\label{sim4}}
\end{figure}

\end{document}